\documentclass[a4paper,11pt]{amsart}
\usepackage{amssymb}
\usepackage{amsmath}
\usepackage{amsthm}
\usepackage{verbatim}
\usepackage[all]{xy}
\usepackage{url}
\usepackage{mathtools}

\usepackage{dsfont}
\renewcommand{\mathbb}{\mathds}

\theoremstyle{plain}

\newtheorem{theorem}[equation]{Theorem}
\newtheorem{proposition}[equation]{Proposition}

\newtheorem{lemma}[equation]{Lemma}
\newtheorem{corollary}[equation]{Corollary}

\theoremstyle{definition}

\newtheorem{remark}[equation]{Remark}
\newtheorem{example}[equation]{Example}
\newtheorem{examples}[equation]{Examples}
\newtheorem{definition}[equation]{Definition}

\newtheorem{notation}[equation]{Notation}
\newtheorem{construction}[equation]{Construction}
\newtheorem{problem}[equation]{Problem}
\newtheorem{main-problem}[equation]{Problem}
\newtheorem{subsec}[equation]{}

\numberwithin{equation}{section}

\newcommand{\R}{{\mathbb R}}
\newcommand{\C}{{\mathbb C}}
\newcommand{\A}{{\mathbb A}}
\newcommand{\G}{{\mathbb G}}
\newcommand{\Q}{{\mathbb Q}}

\newcommand{\Z}{{\mathbb Z}}

\newcommand{\real}{\bf}

\renewcommand{\AA}{{\real A}}
\newcommand{\BB}{{\real B}}
\newcommand{\CC}{{\real C}}
\newcommand{\DD}{{\real D}}
\newcommand{\EE}{{\real E}}
\newcommand{\FF}{{\real F}}
\newcommand{\GG}{{\real G}}
\newcommand{\HH}{{\real H}}
\newcommand{\NN}{{\real N}}
\newcommand{\WW}{{\real W}}

\renewcommand{\SS}{{\real S}}
\newcommand{\TT}{{\real T}}
\newcommand{\UU}{{\real U}}

\newcommand{\YY}{{\real Y}}

\newcommand{\X}{{{X}}}

\newcommand{\gG}{{\scriptscriptstyle \real G}}
\newcommand{\yY}{{\scriptscriptstyle \real Y}}

\newcommand{\rR}{{\scriptscriptstyle \mathbb R}}
\newcommand{\cCC}{{\scriptscriptstyle \mathbf C}}

\newcommand{\cZ}{{\mathcal{Z}}}
\newcommand{\cN}{{\mathcal{N}}}

\newcommand{\bs}{\boldsymbol}

\newcommand{\ag}{{\mathfrak{a}}}
\newcommand{\bg}{{\mathfrak{b}}}
\newcommand{\hg}{{\mathfrak{h}}}
\newcommand{\hR}{\bs{\hg}}
\newcommand{\ig}{{\mathfrak{i}}}
\newcommand{\kg}{\bs{\mathfrak{k}}}
\renewcommand{\ng}{{\mathfrak{n}}}
\newcommand{\tg}{{\mathfrak{t}}}
\newcommand{\tgR}{\bs{\tg}}
\newcommand{\pg}{\bs{\mathfrak{p}}}
\newcommand{\gl}{{\mathfrak{gl}}}
\newcommand{\rg}{{\mathfrak {r}}}
\newcommand{\sg}{{\mathfrak{s}}}
\newcommand{\sgR}{\bs{\sg}}
\newcommand{\zg}{{\mathfrak{z}}}
\newcommand{\zgR}{\bs{\zg}}
\newcommand{\cg}{{\mathfrak{c}}}
\newcommand{\cgR}{\bs{\cg}}
\newcommand{\g}{{\mathfrak{g}}}
\newcommand{\gR}{\bs{\g}}

\newcommand{\into}{\hookrightarrow}

\newcommand{\wh}{\widehat}

\DeclareMathOperator{\Inn}{Inn}
\DeclareMathOperator{\Aut}{Aut}
\DeclareMathOperator{\SAut}{SAut}
\DeclareMathOperator{\Out}{Out}
\DeclareMathOperator{\SOut}{SOut}
\DeclareMathOperator{\inn}{inn}
\DeclareMathOperator{\Lie}{Lie}

\newcommand{\ad}{{\rm ad\hs}}

\newcommand{\Hom}{{\rm Hom}}

\newcommand{\Gal}{{\rm Gal}}

\newcommand{\GL}{{\rm GL}}
\newcommand{\SO}{{\rm SO}}

\newcommand{\Ad}{{\rm Ad}}
\newcommand{\diag}{{\rm diag}}

\newcommand{\ov}{\overline}

\newcommand{\ssc}{{\rm sc}}

\newcommand{\isoto}{\overset{\sim}{\to}}
\newcommand{\onto}{\twoheadrightarrow}
\newcommand{\labelto}[1]{\xrightarrow{\makebox[1.5em]{\scriptsize ${#1}$}}}

\newcommand{\hs}{\kern 0.8pt}
\newcommand{\hssh}{\kern 1.2pt}
\newcommand{\hshs}{\kern 1.6pt}
\newcommand{\hssss}{\kern 2.0pt}

\newcommand{\hm}{\kern -0.8pt}
\newcommand{\hmm}{\kern -1.2pt}

\renewcommand{\H}{{\mathbb H}}

\renewcommand{\hbar}{{\bar h}}
\newcommand{\sbar}{{\bar s}}

\newcommand{\upgam}{\hs^\gamma\hm}
\newcommand{\gam}{{\gamma}}

\newcommand{\id}{{\rm id}}

\newcommand{\ff}{{\rm f}}
\newcommand{\uu}{{\rm u}}
\newcommand{\red}{{\rm red}}

\newcommand{\im}{{\rm im\,}}

\newcommand{\SL}{{\rm SL}}

\newcommand{\GmR}{{\G_{{\rm m},\R}}}
\newcommand{\GmC}{{\G_{{\rm m},\C}}}

\newcommand{\cc}{\raise 1.7pt \hbox{\Tiny{$\bullet$}}}

\newcommand{\Hr}{{\rm H}}
\newcommand{\Zr}{{\rm Z}}
\newcommand{\Br}{{\rm B}}
\newcommand{\Nr}{{\rm N}}

\newcommand{\Ga}{{\Gamma}}

\newcommand{\GGbar}{{\ov\GG}}
\newcommand{\xibar}{{\bar \xi}}
\newcommand{\gbar}{{\bar g}}
\newcommand{\Gbar}{{\ov G}}

\newcommand{\num}{{\scriptscriptstyle\#}}
\newcommand{\sigmai}{{\sigma_i}}
\newcommand{\gi}{{g_i}}
\newcommand{\Xg}{{\mathfrak X}}

\newcommand{\coker}{{\rm coker}}

\newcommand{\SmallMatrix}[1]{{\tiny\arraycolsep=0.4\arraycolsep\ensuremath
{\begin{pmatrix}#1\end{pmatrix}}}}

\newcommand{\vphi}{{\varphi}}

\newcommand{\tl}{\mathfrak{t}}

\newcommand{\parpr}{{\partial\hs'}}

\newcommand{\ntil}{{\tilde n}}
\newcommand{\obs}{{\rm obs}}

\newcommand{\ug}{{\mathfrak u}}

\newcommand{\h}{{\mathfrak{h}}}

\newcommand{\bkappa}{{\beta}}
\newcommand{\rr}{{(r)}}

\newcommand{\anti}{{\text{a-r}}}

\newcommand{\so}{{\mathfrak{so}}}
\newcommand{\eee}{{\mathfrak{e}}}
\newcommand{\eeeR}{{\boldsymbol{\mathfrak{e}}}}

\begin{document}

\title[Computing real Galois cohomology]%
{Computing Galois cohomology of a real\\ linear algebraic group}

\author{Mikhail Borovoi}
\address{Raymond and Beverly Sackler School of Mathematical Sciences,
       Tel Aviv University, 6997801 Tel Aviv, Israel}
\email{borovoi@tauex.tau.ac.il}

\author{Willem A. de Graaf}
\address{Department of Mathematics, University of Trento,
         Povo (Trento), Italy}
\email{degraaf@science.unitn.it}

\thanks{ Borovoi was partially supported
by the Israel Science Foundation (grant 1030/22).
}



\date{\today}

\begin{abstract}
Let $\GG$ be a linear algebraic group, not necessarily  connected or reductive,
over the field of real numbers $\R$.
We describe a method, implemented on computer, to find the first Galois cohomology set $\Hr^1(\R,\GG)$.
The output is a list of $1$-cocycles in $\GG$. Moreover, we describe an implemented algorithm
that, given a $1$-cocycle $z\in \Zr^1(\R, \GG)$, finds the cocycle in the computed list
to which $z$ is equivalent, together with an element of $\GG(\C)$ realizing the
equivalence.
\end{abstract}

\maketitle

\tableofcontents

\setcounter{section}{-1}

\section{Introduction}
\label{s:intro}

Let $\GG$ be a linear algebraic group over the field of real numbers $\R$.
See Section \ref{s:Nonab-abstract} for the definition
of the first Galois cohomology pointed set $\Hr^1(\R,\GG)$;
see also Serre's book \cite{Serre} for the definition of $\Hr^1(k,\GG)$
for an algebraic group $\GG$ over an arbitrary field $k$.
For $k=\R$, the set $\Hr^1(\R,\GG)$ is finite;
see Example~(a) in Section III.4.2
and Theorem~4 in Section III.4.3 of \cite{Serre}.
In terms of Galois cohomology one can state answers to many natural questions;
see Serre  \cite[Section III.1]{Serre} and  Berhuy \cite{Berhuy}.

In particular, one can use  Galois cohomology for the classification of orbits
in representations of real algebraic groups.
Such representations occur in many areas of mathematics and
physics. There are many situations where there are techniques
to determine the orbits of a complex Lie group (that is, of the group of complex points),
but where the orbits of the corresponding real Lie group (the group of real points) are of
great interest. We give two examples illustrating this.

The first example concerns
the classification of $k$-forms on $\R^n$, which is of importance in
differential geometry (see \cite{hit1,hit2}). This is tantamount to classifying
the orbits of $\GL(n,\R)$ on $\bigwedge^k \R^n$. For $k=3$ this classification
was carried out for $n=8$ by Djokovi\'c \cite{Djokovic} and  for $n=9$ in \cite{BGL21};
in \cite{BGL21} we considered the action of the group $\SL(9,\R)$ rather than $\GL(9,\R)$.
For these cases, the orbits of $\GL(8,\C)$ on $\bigwedge^3 \C^8$
and of $\SL(9,\C)$ on $\bigwedge^3 \C^9$ can
be determined using Vinberg's theory of $\theta$-groups
(see \cite{Vinberg76,VE1978}). In \cite{Djokovic} and \cite{BGL21} the real orbits were
determined using Galois cohomology.

Our second example concerns the study of
the Einstein equations of gravity in theoretical physics. Here Lie group
representations coming from a symmetric pair play a pivotal role;
see \cite{bcptr} for an overview. In \cite{cgpt} the real nilpotent orbits of such a
representation (of the group $\SL(2,\R)\times \mathrm{Sp}(6,\R)$ acting on a
28-dimensional vector space) were determined using a technique based on
$\mathfrak{sl}_2$-triples. Due to its computational complexity, it seems
unlikely that this technique can be applied to higher dimensional cases,
and one can hope that an approach based on Galois cohomology can yield a significant
increase in the range of cases that can be dealt with.

We describe known results on Galois cohomology of real algebraic groups.
The Galois cohomology of classical groups and adjoint groups
is well known. The Galois cohomology of compact groups was computed
by Borel and Serre \cite[ Theorem 6.8]{BS};
see also Serre's book \cite[Section III.4.5]{Serre}.
When $\GG$ is a connected compact $\R$-group with a maximal torus $\TT$,
the set $\Hr^1\GG\coloneqq\Hr^1(\R,\GG)$ is in a canonical bijection with the set
of orbits of the Weyl group $W=W(\GG,\TT)$
on the group $\TT(\R)_2$ of elements of order dividing 2 in $\TT(\R)$,
and the corresponding elements of order dividing 2 are explicit cocycles
representing the elements of $\Hr^1\GG$.

In the announcement \cite{B88}, which was detailed in \cite{Borovoi22-CiM}, the first-named author
computed the Galois cohomology set $\Hr^1(\R,\GG)$
for a connected reductive $\R$-group $\GG$ (not necessarily compact)
by the method of Borel and Serre,
in terms of a certain action of the Weyl group
on the first Galois cohomology $\Hr^1\TT$ of a {\em fundamental} torus $\TT\subset\GG$
(the maximal torus containing a maximal compact torus).

The result of \cite{B88} has been used in a few articles,
in particular, in \cite{Sch},  \cite{C}, and  \cite{NP}.
Using this result, Borovoi, Gornitskii, and Rosengarten \cite{BGR}
described the Galois cohomology of {\em quasi-connected} reductive $\R$-groups
(normal subgroups of connected reductive $\R$-groups).
In \cite{BE}, Borovoi and  Evenor  used the result of \cite{B88}
to describe {\em explicitly} the Galois cohomology of {\em simply connected} semisimple $\R$-groups.
In \cite{BT21} and \cite{BT*}, Borovoi and  Timashev  used ideas of Kac \cite{Kac68}, \cite{Kac69}
and the result of \cite{B88} to describe explicitly
the Galois cohomology of connected {\em reductive} $\R$-groups.

As outlined above,
the Galois cohomology set $\Hr^1\GG$ is necessary for classification problems
in Algebraic Geometry and Linear Algebra over $\R$;
see, for instance, \cite{Djokovic} and \cite{BGL21}.
Moreover, for Arithmetic Geometry one needs the Galois cohomology
$\Hr^1(\Q,\GG)$  for algebraic groups $\GG$ over the field of rational numbers $\Q$,
and in order to compute the Galois cohomology over $\Q$
one needs the Galois cohomology over $\R$;
see \cite[Theorem 1.5.1(1)]{BK}.

Our paper \cite{BGL21} contains a typical application of real Galois cohomology.
We classified the orbits of $\SL(9,\R)$ on
$\bigwedge^3 \R^9$, knowing the orbits of $\SL(9,\C)$ on $\bigwedge^3 \C^9$
(which had been classified by Vinberg and Elashvili \cite{VE1978}).
In such a situation one needs to compute
the first Galois cohomology $\Hr^1\GG$ in many cases (in \cite{BGL21}
it was computed ``by hand'' in more than 100 cases). Furthermore, these
sets are needed in explicit form, that is, one needs a list of cocycles in
$\GG$ representing all classes in $\Hr^1\GG$. Secondly, one needs a
method to decide equivalence of cocycles.

We remark that \cite{at} describes methods to compute the size (cardinality)
of the finite set  $\Hr^1\GG$ when $\GG$ is connected and reductive.
However, as said above, for our purposes
the size is not enough: we need explicit cocycles representing the elements of $\Hr^1\GG$.
Moreover, the explicit results of  \cite{BT21} and \cite{BT*}
for connected reductive $\R$-groups, and the less explicit results of \cite{BGR}
for quasi-connected reductive $\R$-groups, are not enough either:
we need to be able to compute $\Hr^1\GG$ for linear
algebraic groups that are not necessarily reductive, or connected, or quasi-connected.
In this paper we give algorithms, implemented on computer, to determine the Galois cohomology set
of an {\em arbitrary} real linear algebraic group.

For a linear algebraic group $\GG$ over $\R$ we consider
the following two problems.

\begin{main-problem}\label{mp:1}
Find a list
\[g_1,\dots, g_m\in\Zr^1\hs\GG\coloneqq\Zr^1(\R,\GG)\]
of 1-cocycles representing all cohomology classes
in $\Hr^1\GG\coloneqq \Hr^1(\R,\GG)$ so that
for any cohomology class we have exactly one cocycle.
\end{main-problem}

\begin{main-problem}\label{mp:2}
Assume that we have solved Problem \ref{mp:1} for $\GG$.
For any 1-cocycle $g\in \Zr^1\hs\GG$, determine  $i$ with $1\le i\le m$
such that $g\sim g_i$ and find an element $s\in \GG(\C)$ such that
\[s^{-1}\cdot g\cdot\upgam\hm s=g_i\]
where $\gamma$ denotes complex conjugation.
\end{main-problem}

We solve these problems using computer for a general linear algebraic $\R$-group,
not necessarily connected or reductive.

We describe our setup.
We write $\Gamma=\Gal(\C/\R)=\{1,\gamma\}$
for the Galois group of $\C$ over $\R$,
where $\gamma$ is the complex conjugation.
Let $\GG$ be a real linear algebraic group.
In the coordinate language, the reader may regard $\GG$ as
a subgroup in the general linear group $\GL_n(\C)$ (for some integer~$n$)
defined by polynomial equations with {\em real} coefficients in the matrix entries;
see Borel \cite[Section 1.1]{Borel-66}.
More conceptually, the reader may assume that $\GG$ is an affine group scheme
of finite type over $\R$; see  Milne \cite[Definition 1.1]{Milne-AG}.
We say just that $\GG$ is an $\R$-group.
Let $\GG(\R)$ and $\GG(\C)$ denote the groups of real points
and complex points of $\GG$, respectively.
We denote by  $G$
(the same letter, but in the usual non-bold font)
the base change $G=\GG\times_\R\C$ of $\GG$ from $\R$ to $\C$.
By abuse of notation we identify $G$ with $\GG(\C)$.
The Galois group $\Gamma$ acts on $G$, that is,
the complex conjugation $\gamma$
acts on $G$ by an anti-holomorphic involution
\[\sigma_\gG\colon G\to G,\quad g\mapsto \upgam g\ \  \text{for}\ g\in G.\]
Moreover, the anti-holomorphic involution $\sigma_\gG$ is anti-regular
in the sense of Definition \ref{d:reg-anti-reg} in Appendix \ref{app:A}.
Conversely, a pair $(G,\sigma)$,
where $G$ is a $\C$-group (a linear algebraic group over $\C$)
and $\sigma\colon G\to G$ is an anti-regular involution of $G$,
by Galois descent comes from a unique
(up to a canonical isomorphism) $\R$-group $\GG$;
see the references in Subsection \ref{ss:G-anti}.
Then by abuse of notation we write $\GG=(G,\sigma)$.

We describe the input of our computer program for Problem \ref{mp:1}.
We assume that $\GG\subset \GL_{n,\R}$.
If $\GG$ is connected (that is, $G$ is connected), then the only input is a list of real $n\times n$ -matrices that form a basis
of the real Lie subalgebra $\gR_\rR=\Lie\GG\subset \gl(n,\R)$.
When  $\GG$ is not connected, the input is a list of real matrices that form a basis of $\gR_\rR$,
thus defining the identity component of $\GG$, and a list of complex matrices
whose elements are representatives of the elements of the component group.
From this input, the main program computes a list of cocycles
whose classes exhaust $\Hr^1(\R,\GG)$, solving Problem \ref{mp:1}.
We also have a program for deciding equivalence of cocycles, which solves Problem \ref{mp:2}.

Using the algorithm indicated in Remark \ref{rem:iseltof},
we can easily compute a multiplication table of
the component group $\pi_0(G)$, as well as the $\Gamma$-action on $\pi_0(G)$.
Therefore, we may and will assume that we have this data as input to our algorithms as well.
Note that we do not need data like Dynkin diagram, based root datum, etc.
These are computed as needed.

We have implemented our algorithms in the language of the computational
algebra system {\sf GAP}; see \cite{GAP4}. For certain number theoretic
computations we have written a small interface to {\sf SageMath}, \cite{sage}.
Our programs are available on the website

\begin{verbatim}
https://degraaf.maths.unitn.it/galcohom.html
\end{verbatim}

We now give an outline of our methods for computing  $\Hr^1(\R, \GG)$.
For simplicity we write $\Hr^n\GG$ for $\Hr^n(\R,\GG)$ $(n=1,2)$.
When $\GG$ is connected and  reductive,
we compute $\Hr^1\GG$ by the method of \cite{B88} and \cite{Borovoi22-CiM}.
When $\GG$ is connected but not necessary reductive,
we reduce Problems \ref{mp:1} and \ref{mp:2}
to the reductive case using Sansuc's lemma.

For a general  $\R$-group $\GG$ (not necessarily connected)
we write the short exact sequence
\[1\to \GG^0\to \GG\to\pi_0(\GG)\to 1\]
where $\GG^0$ is the identity component of $\GG$,
and $\pi_0(\GG)$ is the component group.
We compute $\Hr^1\GG$ by d\'evissage, that is,
we compute $\Hr^1\GG$ via certain calculations
with the component group  $\pi_0(\GG)$ and with the identity component $\GG^0$.

First, we compute $\Hr^1\pi_0(\GG)$ for the finite $\Gamma$-group  $\pi_0(\GG)$ by brute force.
After that, for each cohomology class
$\xi_0=[c_0]\in \Hr^1\pi_0(\GG)$ we try to lift $\xi_0$ to $\Hr^1\GG$,
that is, we try to lift the cocycle $c_0\in \Zr^1\pi_0(\GG)$ to a cocycle
$c\in \Zr^1\GG$.
When trying to do that, we obtain an obstruction $\obs(\xi_0)$
living in the second nonabelian Galois cohomology set
$\Hr^2(\R,G^0,\bkappa)$ where $\bkappa$ is a certain {\em band (lien)}.
The cohomology class $\xi_0=[c_0]$ can be lifted to $\Hr^1\GG$
if and only if the obstruction $\obs(\xi_0)$ is {\em neutral}.
We check whether $\obs(\xi_0)$ is neutral or not,
using the criterion of neutrality of \cite[Theorem 5.5]{Borovoi-Duke}.
If the obstruction $\obs(\xi_0)$ is non-neutral, we discard $\xi_0$.
If it is neutral, we lift $c_0$ to a cocycle $c\in \Zr^1 \GG$,
and we obtain the preimage in $\Hr^1\GG$ of the singleton $\{\xi_0\}$
via computing $\Hr^1\hs_c \GG^0$,
where the connected  $\R$-group $_c\GG^0$
is obtained from $\GG^0$ by twisting using $c$.

\begin{remark}\label{r:brute-force}
For the brute force computation of  $\Hr^1\pi_0(\GG)$ we assume that
$\pi_0(\GG)$ is ``not too large''. In practice that means that its cardinality
should not be much bigger than about a million. In particular, we also assume
that we have a list of all elements of the component group. It is an
interesting problem in its own right to develop efficient algorithms to
compute the Galois cohomology of finitely generated
$\Gamma$-groups (finite or infinite). We will not go into that in this paper.
\end{remark}

The paper is roughly organized as follows.
The first sections (\ref{s:Abelian}
to \ref{s:Nonab-abstract}) contain theoretical background on Galois
cohomology. Sections \ref{app:B} and \ref{sec:tori} contain,
respectively, general algorithms concerning Levi decompositions of algebraic Lie algebras
and general algorithms concerning algebraic tori. In Section \ref{s:connected} we describe how we find
$\Hr^1\GG$ for a connected reductive group $\GG$, and  Section
\ref{s:connected-equivalence} contains a method for deciding equivalence of
cocycles for such groups.
Section \ref{s:H2}  is concerned with second nonabelian cohomology.
In Section \ref{s:non-reductive} it is
explained how to solve Problems \ref{mp:1} and \ref{mp:2} for connected
non-reductive groups. Section \ref{s:H2-quasi-torus}
describes methods to compute the (abelian) $\Hr^2$ for
quasi-tori. The second nonabelian cohomology is used in
Sections \ref{s:neutralizing} and \ref{s:neutralizing-nonred}, which deal
with the problem of neutralizing a 2-cocycle. This is necessary for the general
algorithm that computes the $\Hr^1(\R,\GG)$ for non-connected algebraic
groups, detailed in Section \ref{s:non-connected}. Finally, Section
\ref{sec:gap} briefly discusses the implementation of the algorithms and
gives the running times on some sample inputs.

{\sc Acknowledgements.} The authors are grateful to  H\^ong V\^an L\^e
for her suggestion to develop a computer program
computing the Galois cohomology of a real linear algebraic group.
We thank the anonymous referees for thorough refereeing and for useful comments,
which helped us to improve the exposition.

\subsection*{Notation and conventions}
In this paper, by an algebraic group we mean a {\em linear} algebraic group.
By letters $\GG,\HH,\dots$ in the boldface  font we denote {\em real} algebraic groups.
By the same letters, but in the usual (non-bold) font $G,H,\dots$,
we denote the corresponding complex algebraic groups
$G=\GG\times_\R\C$, $H=\HH\times_\R\C$,\dots,
and by the corresponding small Gothic letters $\g,\h,\dots$, we denote the (complex) Lie algebras of $G,H,\dots$.
We denote by $\gR_\rR=\Lie \GG$, $\hR_\rR=\Lie \HH$ the corresponding real Lie algebras;
then $\g=\gR_\rR\otimes_\R\C$, $\h=\hR_\rR\otimes_\R\C$.
We denote by $\GG(\R)$ and $\GG(\C)$ the Lie groups of the real points and the complex points of $\GG$, respectively;
then $\gR_\rR=\Lie \GG(\R)$, $\g=\Lie \GG(\C)$.
By abuse of notation, we identify the complex algebraic group $G=\GG\times_\R\C$ with the group of $\C$-points $\GG(\C)=G(\C)$.
In particular, $g\in G$ means $g\in G(\C)$.

We gather some of our notations here:
\begin{itemize}
\item $\R$ and $\C$ denote the fields of real and complex numbers, respectively;
\item $\Gamma=\Gal(\C/\R)=\{1,\gamma\}$ where $\gamma$ is the complex conjugation;
\item $\GmR$ and $\GmC$ denote the multiplicative groups over $\R$ and over $\C$, respectively;
\item $Z(G)$ denotes the center of an algebraic group $G$;
\item $\cZ_G(S)$ denotes the centralizer in $G$ of a set $S$;
\item $\zg_\g(S)$ denotes the centralizer in a Lie algebra $\g$ of a set $S$;
\item $\cN_G(H)$ denotes the normalizer in $G$ of an algebraic subgroup $H\subset G$;
\item $\Aut(G)$ denotes the automorphism group of $G$;
\item $\Inn(G)$ denotes the group of inner automorphisms of $G$;
\item $\Out(G)=\Aut(G)/\Inn(G)$;
\item $\SAut_\anti(G)$ denotes the set of anti-regular automorphisms of a complex algebraic group $G$; see Appendix \ref{app:A};
\item $\SAut(G)=\Aut(G)\cup\SAut_\anti(G)$, which is a group;
\item $\SOut(G)=\SAut(G)/\Inn(G)$;
\item $G^0$ is the identity component of $G$;
\item $\pi_0(G)\coloneqq G/G^0$ is the component group of $G$;
\item $\Hr^n\GG=\Hr^n(\R,\GG)$ for $n=1,2$ (when $n=2$ we assume that $\GG$ is commutative).
\end{itemize}

\section{Abelian cohomology}
\label{s:Abelian}

\begin{subsec}
\label{ss:H1-abelian}
Let $A$ be a $\Ga$-module, that is, an abelian group written additively,
endowed with an action of $\Gamma=\{1,\gamma\}$.
We consider the first cohomology group $\Hr^1(\Ga,A)$.
We write $\Hr^1\hm A$ for $\Hr^1(\Ga,A)$.
Recall that
\[\Hr^1\hm A=\Zr^1\hm A/\Br^1\hm A,\quad\
\text{where}\quad\ \Zr^1\hm A=\{a\in A\mid\upgam a=-a\}, \quad
\Br^1\hm A=\{\upgam a'-a'\mid a'\in A\}.\]

We define the second cohomology group $\Hr^2A$ by
\[ \Hr^2\hm A=\Zr^2\hm A/ \Br^2\hm A,\quad
\text{where}\ \ \Zr^2\hm A=A^\Gamma\coloneqq \{a\in A\mid \upgam a=a\},\ \,
\Br^2\hm A=\{\upgam a'+a'\mid a'\in A\}.\]

For $k\in\Z$ we define the coboundary operator
\[d^k\colon A\to A,\quad a\mapsto\upgam a-(-1)^k a.\]
In other words, $d^k=\gamma-(-1)^k\in \Z[\Gamma]$,
 where $\Z[\Gamma]=\Z\oplus\Z\gamma$  is the group ring of $\Gamma$.
Then  $d^{k+1}\circ d^k=0$.
We define the {\em Tate cohomology groups}
$\wh \Hr^k\hm A$ for all $k\in \Z$ by
\[\wh \Hr^k\hm A=\Zr^k\hm A/\Br^k\hm A,\]
where
\begin{align*}
\Zr^k\hm A=\ker d^k=\{a\in A\mid\upgam a=(-1)^k a\},\quad
\Br^k\hm A=\im d^{k-1}=\{\upgam a'+(-1)^k  a'\mid a'\in A\}.
\end{align*}
Then clearly
\[\wh\Hr^k\hm A=\Hr^1\hm A\ \ \text{if $k$ is odd,\quad and}
       \quad \wh \Hr^k\hm A=\Hr^2\hm A\ \ \text{if $k$ is even.}\]
\end{subsec}

In this paper, for any $k\in\Z$ we write $\Hr^k\hm A$ for $\wh\Hr^k\hm A$.
In particular,
\[\Hr^0\hm A=\Zr^0\hm A/\Br^0\hm A=A^\Gamma\hm/\{a'+\upgam a'\mid a'\in A\}.\]

\begin{remark}
In the standard exposition, our definitions become theorems
that hold specifically for the cohomology of a finite group $\Gamma$ of order 2;
see \cite[Theorem 5 in Section 8]{AW}.
\end{remark}

\begin{lemma}\label{l:2-xi}
For any $k\in \Z$ and $\xi\in \Hr^k\hm A$, we have $2\xi=0$.
\end{lemma}

\begin{proof}
See \cite[Section 6, Corollary 1 of Proposition 8]{AW}, or
\cite[Lemma 3.2.1]{BGL21}, or \cite[Lemma 1.3]{BT*}.
\end{proof}

\begin{subsec}
\label{ss:ABC-abelian}
Let
\[0\to A\labelto{i} B\labelto{j} C\to 0\]
be a short exact sequence of $\Ga$-modules.
It gives rise to a cohomology exact sequence
\begin{equation*}
\dots\Hr^{k-1} C\labelto{\delta^{k-1}}\Hr^k\hm A\labelto{i_*^k}
     \Hr^k B \labelto{j_*^k} \Hr^k C\labelto{\delta^k}\Hr^{k+1}\hm A\dots
\end{equation*}
We recall the formula for $\delta^k$.
We identify $A$ with $i(A)\subseteq B$.
Let $[c]\in \Hr^k C$, $c\in \Zr^k C$.
We lift $c$ to some $b\in B$ and set $a=d^k\hs b$.
Then $a\in \Zr^{k+1}\hm A$.
We set $\delta^k[c]=[a]\in \Hr^{k+1}\hm A$.
In particular, we have
\begin{equation*}
\delta^0[c]=[\upgam b-b]\ \,\text{for}\ c\in \Zr^0\hs C,\qquad
     \delta^1[c]=[\upgam b +b]\ \,\text{for}\ c\in \Zr^1\hs C.
\end{equation*}
\end{subsec}

\section{Hypercohomology}
\label{s:hyper}

\begin{definition}
A {\em short complex of $\Gamma$-modules}
is a morphism of $\Gamma$-modules $A_{1}\labelto{\partial} A_0$.
\end{definition}

\begin{subsec}
For $k\in\Z$ we define a differential
\[D^k\colon\, A_1\oplus A_0\to A_1\oplus A_0,\quad\ D^k(a_1,a_0)=
      \big(d^{k+1} a_1,\, d^k a_0-(-1)^k\partial a_1\big).\]
Then $D^k\circ D^{k-1}=0$.
We define the {\em  $k$-th Tate hypercohomology group} (see \cite[Section 3]{BK})
\begin{equation*}
\H^k(A_{1}\labelto{\partial} A_0)=\Zr^k(A_{1}\labelto{\partial} A_0)\hs/\hs
   \Br^k(A_{1}\labelto{\partial} A_0),
\end{equation*}
where
\begin{equation*}
\begin{aligned}
\Zr^k(A_{1}\labelto{\partial} A_0)&=\ker D^k\\
&=\{(a_{1},a_0)\in A_{1}\oplus A_0,\,\hs
      \mid\,d^{k+1} a_1=0,\,  d^k a_0-(-1)^k\partial a_{1}=0\},\\
\Br^k(A_{1}\labelto{\partial} A_0)&={\rm im\,} D^{k-1}\\
&=\{(d^k a'_1,\ d^{k-1} a'_0-(-1)^{k-1}\partial a'_1\hs) \,\mid\, (a'_1, a_0')\in A_1\oplus A_0\}.
\end{aligned}
\end{equation*}
For simplicity we write $\H^k(A_{1}\to A_0)$ instead of $\H^k(A_{1}\labelto{\partial} A_0)$.
\end{subsec}

\begin{examples}
\begin{enumerate}
\item We have an isomorphism
\[\Hr^k \hm A_0\isoto \H^k(0\to A_0),\quad [a_0]\mapsto [0,a_0].\]

\item We have an isomorphism
\[\H^k(A_1\to 0)\isoto \Hr^{k+1}\hm A_1, \quad [a_1,0]\mapsto [a_1].\]
\end{enumerate}
\end{examples}

The correspondence $(A_{1}\to A_0)\rightsquigarrow \H^k(A_{1}\to A_0)$ is a functor
from the category of short complexes of $\Gamma$-modules to the category of abelian groups.
Moreover, a short exact sequence of short complexes of $\Gamma$-modules
\[0\to(A_{1}\to A_0)\labelto\iota (B_{1}\to B_0)\labelto\varkappa (C_{1}\to C_0)\to 0\]
gives rise to a hypercohomology exact sequence
\begin{equation}\label{e:exact-complexes}
\dots \H^k(A_{1}\to A_0)\labelto{\iota_*^k} \H^k(B_{1}\to B_0)\labelto{\varkappa_*^k}
                         \H^k(C_{1}\to C_0)\labelto{\delta^k}\H^{k+1}(A_{1}\to A_0)\dots
\end{equation}

We specify the maps $\delta^k$ in \eqref{e:exact-complexes}.
We identify the complex $(A_1\to A_0)$ with its image in $(B_1\to B_0)$.
Let $(c_1,c_0)\in \Zr^k(C_0\to C_1)\subseteq (C_1\oplus C_0)$.
We lift $(c_1,c_0)$ to some $(b_1,b_0)\in B_1\oplus B_0$ and set $(a_1,a_0)=D^k(b_1,b_0)$.
Then $(a_1,a_0)\in \Zr^{k+1}(A_1\to A_0)$,
and we set $\delta^k[c_1,c_0]=[a_1,a_0]\in\H^{k+1}(A_1\to A_0)$.

\begin{example}
Applying \eqref{e:exact-complexes} to the short exact sequence of complexes
\[0\to (0\to A_0)\labelto{\lambda} (A_{1}\to A_0)\labelto{\mu} (A_{1}\to 0)\to 0\]
with \,$\lambda(0,a_0)=(0,a_0)$, \,$\mu(a_1,a_0)=(a_1,0)$,
\,we obtain an exact sequence
\begin{equation}\label{e:coho-hyper}
\dots\to \Hr^k\hm A_{1} \labelto{\partial_*^k} \Hr^k\hm A_0
  \labelto{\lambda^k_*} \H^k(A_{1}\to A_0)\labelto{\mu_*^k} \Hr^{k+1}\hm A_{1}\,
  \labelto{\partial_*^{k+1}}\, \Hr^{k+1}\hm A_0\to\dots
\end{equation}
where the maps  $\lambda_*^k$, $\mu_*^k$, and $\partial_*^{k+1}$
in \eqref{e:coho-hyper} are the following:
\begin{align*}
&\lambda_*^k\colon\,\Hr^k\hm A_0\hs\to\hs \H^k(A_{1}\to A_0),\qquad    [a_0]\mapsto [0,a_0],\\
&\mu_*^k\colon\,\H^k(A_{1}\to A_0)\hs\to\hs \Hr^{k+1}\hm A_{1},\quad   [a_1,a_0]\mapsto [a_1],\\
&\partial_*^{k+1}\colon \Hr^{k+1}\hm A_1\to \Hr^{k+1}\hm A_0,\qquad\quad
             [a_1]\mapsto(-1)^{k+1}[\partial a_1]=[\partial a_1],
\end{align*}
where $(-1)^{k+1}[\partial a_1]=[\partial a_1]$ because $2[\partial a_1]=0$ by Lemma \ref{l:2-xi}.
\end{example}

\begin{definition}
A morphism of short complexes
\begin{equation}\label{e:quasi}
\varphi\colon\,(A_{1}\labelto{\partial} A_0)\,\to\, (A'_{1}\labelto{\parpr} A'_0)
\end{equation}
is called a {\em quasi-isomorphism} if the induced homomorphisms
\[\ker\partial\to\ker\parpr \quad \text{ and }\quad \coker\partial\to\coker\parpr\]
are isomorphisms.
\end{definition}

\begin{examples}
\begin{enumerate}
\item If $ A_{1}\into A_0$ is injective, then
\   $(A_{1}\into A_0)\,\to\,\big( 0,\coker[ A_1\hm\into\hm A_0]\big)$ \
  is a quasi-isomorphism.

\item If $ A_{1}\onto A_0$ is surjective, then
   \  $\big(\ker\,[ A_1\hm\onto\hm A_0]\to 0\big)\,\to\,(A_{1}\onto A_0)$ \
   is a quasi-isomorphism.
\end{enumerate}
\end{examples}

\begin{lemma}[well-known]
\label{l:quasi}
  A quasi-isomorphism of complexes of $\Gamma$-modules  \eqref{e:quasi}
induces isomorphisms on the hypercohomology
\[\varphi_*^k\colon\hs \H^k(A_{1}\to A_0)\isoto \H^k(A'_{1}\to A'_0).\]
\end{lemma}

\begin{proof} See, for instance, \cite[Proposition VII(5.2)]{Brown}. \end{proof}

\begin{examples}
\begin{enumerate}
\item $\H^k(0\to A_0)=\Hr^k\hm A_0$.
     Hence, if a homomorphism $A_{1}\into A_0$ is injective, then by Lemma \ref{l:quasi} we have
    \[\H^k(A_{1}\into A_0)\,\cong\,\Hr^k\hs\hs\coker[ A_1\into A_0].\]
\item  $\H^k(A_{1}\to 0)=\Hr^{k+1}\hm A_{1}$.
     Hence, if a homomorphism $ A_{1}\onto A_0$ is surjective, then by Lemma \ref{l:quasi} we have
     \[\H^k(A_{1}\onto A_0)\,\cong\,\Hr^{k+1}\ker\hs[ A_1\onto A_0].\]
\end{enumerate}
\end{examples}

\begin{subsec}
Let
\[0\to A\labelto{i} B\labelto{j} C\to 0\]
be a short exact sequence of $\Ga$-modules,
where we identify $A$ with $i(A)\subseteq B$.
Then by Lemma \ref{l:quasi}, the quasi-isomorphism
\[ (A\to 0)\,\to\, (B\to C)\]
induces isomorphisms
\[ \Hr^{k+1}\hm A\labelto\sim \H^k(B\to C),\quad [a]\mapsto [a,0].\]
\end{subsec}

\section{Galois cohomology of tori and quasi-tori}
\label{s:Tori-quasi-tori}

\begin{subsec}
Let $\TT$ be an $\R$-torus.
Consider the {\em cocharacter group}
\[ \X_*(\TT)=\Hom(\GmC,T)\]
where $\GmC$ is the multiplicative group over $\C$.
The group $\Gamma=\{1,\gamma\}$ acts on $\X_*(\TT)$ by
\[(\upgam\nu)(z)=\upgam(\nu(\hs^{\gamma^{-1}}\hm z))
     \ \ \text{for } \nu\in\X_*(\TT),\ z\in\C^\times\]
(where in our case $\gamma^{-1}=\gamma$).
We see that $\X_*(\TT)$  is a $\Gamma$-lattice, that is,
a finitely generated free abelian group with $\Gamma$-action.

Let $L$ be a $\Gamma$-lattice.
We say that $L$ is {\em indecomposable} if it is not a direct sum
 of its two nonzero $\Gamma$-sublattices.
Clearly, any $\Gamma$-lattice is a direct sum of indecomposable lattices.
\end{subsec}

\begin{proposition}\label{p:indecomp}
Up to isomorphism, here are exactly three indecomposable $\Gamma$-lattices:
\begin{enumerate}
\item $\Z$ with trivial action of $\gamma$;
\item $\Z$ with the action of $\gamma$ by $-1$;
\item $\Z\oplus \Z$ with the action of $\gamma$ by the matrix \SmallMatrix{0&1\\1&0}.
\end{enumerate}
\end{proposition}

\begin{proof}
See Curtis and Reiner \cite[Theorem (74.3)]{CR}.
See also Casselman \cite[Theorem 2]{Casselman},  \cite[Appendix A]{BT*}, and Subsection \ref{lat:3} below for  elementary proofs.
\end{proof}

\begin{subsec}
Let $\TT$ be an $\R$-torus; then $\X_*(\TT)$
is a $\Gamma$-lattice.
Let $\vphi \colon \TT\to\SS$ be a homomorphism of $\R$-tori;
then $\vphi_*\colon\X_*(\TT)\to\X_*(\SS)$ is a homomorphism of $\Gamma$-lattices.
In this way we obtain an equivalence between the category
of $\R$-tori and the category of $\Gamma$-lattices.

We say that an $\R$-torus is {\em indecomposable}
if it is not a direct product of two nontrivial subtori.
Clearly, every $\R$-torus is a direct product of indecomposable $\R$-tori.
It is also clear that a torus $\TT$ is indecomposable
if and only if its cocharacter lattice $\X_*(\TT)$ is indecomposable.
\end{subsec}

\begin{corollary}[of Proposition \ref{p:indecomp};
see, for instance, Voskresenski\u{\i} {\cite[Section 10.1]{Vos}}]
\label{c:indecoposable-tori}
Up to isomorphism, there are exactly three indecomposable $\R$-tori:
\begin{enumerate}
\item $\GmR=(\C^\times,\, z\mapsto \bar z)$ with group of $\R$-points $\R^\times$;
\item $R_{\C/\R}^{(1)}\GmC=(\C^\times,\, z\mapsto\bar z ^{-1})$  with group of $\R$-points $$U(1)=\{z\in\C^\times\mid z\bar z=1\};$$
\item $R_{\C/\R}\GmC=(\hs\C^\times\times\C^\times,\, (z_1,z_2)\mapsto
     (\bar z_2,\bar z_1)\hs)$ with group of $\R$-points $\C^\times$.
\end{enumerate}
\end{corollary}

We have a canonical $\Gamma$-equivariant {\em evaluation homomorphism}
\begin{equation}\label{e:e-X*-T}
e\colon \X_*(\TT)\to \TT(\C),\quad  \nu\mapsto \nu(-1)\ \ \text{for }\nu\in \X_*(\TT).
\end{equation}

\begin{theorem}[\hs{\cite[Theorem 3.6]{BT*}}\hs]
\label{p:X*}
Let $\TT$ be an $\R$-torus. For any $k\in \Z$, the  homomorphism
\begin{equation*}
 e_*\colon \Hr^k\hs\hs\X_*(\TT)\to \Hr^k\hs \TT
 \end{equation*}
induced by the $\Gamma$-equivariant homomorphism \eqref{e:e-X*-T} is an isomorphism.
\end{theorem}

\begin{subsec}
Let $\TT_{1}\labelto{\partial}\TT_0$ be a short complex of $\R$-tori.
Consider the short complex of $\Gamma$-modules
$\X_*(\TT_{1})\labelto{\partial_*}\X_*(\TT_0)$.
Formula \eqref{e:e-X*-T} permits us to define
a morphism of short complexes of $\Gamma$-modules
\begin{equation}\label{e-quasi-tori}
\big(\hs\X_*(\TT_{1})\to\X_*(\TT_0)\hs\big)\,\to\,
     \big(\hs\TT_{1}(\C)\to\TT_0(\C)\hs\big),
\end{equation}
which in general is not a quasi-isomorphism.
\end{subsec}

\begin{proposition}[\hs{\cite[Proposition 3.12]{BT*}}\hs]
\label{p:quasi-tori}
The morphism of short complexes \eqref{e-quasi-tori}
induces isomorphisms on hyper\-cohomology
\[\H^k\big(\hs\X_*(\TT_{1})\hm\to\hm\hm\X_*(\TT_0)\hs\big)\,
       \labelto\sim\, \H^k(\TT_{1}\hm\to\hm\hm\TT_0),\quad\ \
       [\nu_1,\nu_0]\mapsto[\nu_1(-1),\nu_0(-1)].\]
\end{proposition}

\begin{subsec}
\label{ss:Hk-quasi}
Following Gorbatsevich, Onishchik, and Vinberg \cite[Section 3.3.2]{GOV},
we say that a {\em quasi-torus} over $\R$
is a commutative algebraic $\R$-group $\AA$
such that all elements of $\AA(\C)$ are semisimple.
In other words, $\AA$ is an $\R$-group of multiplicative type;
see Milne \cite[Corollary 12.21]{Milne-AG}.
In other words, $\AA$ is an $\R$-subgroup of some $\R$-torus $\TT$;
see, for instance, \cite[Section 2.2]{BGR}.

Set $\TT'=\TT/\AA$; then we have a short exact sequence
\begin{equation*}
1\to\AA\to\TT\to\TT'\to 1,
\end{equation*}
whence we obtain a quasi-isomorphism
\[(\AA\to 1)\hs\to\hs(\TT\to\TT')\]
and an isomorphism
\begin{equation}\label{e:ATT'}
\H^{k-1}\big(\X_*(\TT)\to\X_*(\TT')\big)\labelto{\sim}
     \H^{k-1}(\TT\to\TT')\labelto{\sim} \Hr^k\hm\AA.
\end{equation}
which computes the Galois cohomology of a quasi-torus $\AA$ in terms of lattices.

We write an explicit formula for the isomorphism \eqref{e:ATT'}.
Let
\[(\nu,\nu')\in \Zr^{k-1}\big(\X_*(\TT)\to\X_*(\TT')\big).\]
Consider
\[\big(\nu(-1),\nu'(-1)\big)\in \Zr^{k-1}(\TT\to\TT').\]
Choose  a preimage $t\in T$ of $\nu'(-1)\in T'$ and set
\begin{equation}\label{e:nu-nu'-a}
a=\nu(-1)\cdot d^{k-1}(t^{(-1)^k}).
\end{equation}
Then $(a,1)\sim \big(\nu(-1),\nu'(-1)\big)$, and therefore
$[a]\in \Hr^k\AA$ is the image of $[\nu,\nu']$ under the isomorphism \eqref{e:ATT'}.
\end{subsec}

\section{Nonabelian $\Hr^1$ for $\Gamma$-groups and $\R$-groups}
\label{s:Nonab-abstract}

In this section we consider the cohomology of a group $\Gamma$ of order 2.
See Serre's book \cite[Section I.5]{Serre}
for the case of an arbitrary profinite group $\Gamma$.

\begin{subsec} \label{ss:d-H1}
Let $A$ be a $\Ga$-group (written multiplicatively),
that is, a group (not necessarily abelian) endowed with an action of $\Ga$.
We consider the first cohomology $\Hr^1(\Ga,A)$.
We write $\Hr^1\hm A$ for $\Hr^1(\Ga, A)$. Recall that
\[ \Hr^1\hm A=\Zr^1\hm A/\sim,\quad\ \text{where}\
      \quad \Zr^1\hm A=\{a\in A\mid a\cdot\upgam a=1\},\]
and two 1-cocycles (elements of $\Zr^1\hm A$ ) $a_1,\,a_2$ are equivalent
(we write $a_1\sim a_2$) if there exists $a'\in A$ such that
\[a_2=a'\cdot a_1 \cdot (\upgam a')^{-1}.\]
The set $\Hr^1\hm A$ has a canonical {\em neutral element} $[1]$,
the class of the cocycle $1\in \Zr^1\hm A$.
The correspondence $A\rightsquigarrow \Hr^1\hm A$ is a functor
from the category of $\Gamma$-groups to the category of pointed sets.

If the group $A$ is abelian, then
\[\Hr^1\hm A=\Zr^1\hm A/\Br^1\hm A,\]
where the abelian subgroups $\Zr^1\hm A$ and $\Br^1\hm A$
were defined as in Subsection \ref{ss:H1-abelian}.
Thus $\Hr^1\hm A$ is naturally an abelian group.
\end{subsec}

\begin{construction}
Let
\begin{equation}\label{e:ABC}
 1\to A\labelto{i} B\labelto{j} C\to 1
 \end{equation}
be a short exact sequence of  $\Gamma$-groups.
Then we have a cohomology exact sequence
\[
1\to A^\Ga\labelto{i} \Br^\Ga\labelto{j} C^\Ga\labelto{\delta}
      \Hr^1\hm A\labelto{i_*} \Hr^1 B\labelto{j_*} \Hr^1 C;
\]
see Serre \cite[I.5.5, Proposition 38]{Serre}.
We recall the definition of the map $\delta$.

Let $c\in C^\Gamma$. We lift $c$ to an element $b\in B$
and set $a=b^{-1}\cdot\upgam b\in B$.
It is easy to check that in fact $a\in\Zr^1\hm A\subseteq A$.
We set $\delta(c)=[a]\in \Hr^1\hm A$.
\end{construction}

\begin{construction}
Assume that  the normal  subgroup $i(A)$ of $B$ in \eqref{e:ABC}  is {\em central} in $B$.
Then we have a cohomology exact sequence
\[
1\to A^\Ga\labelto{i} \Br^\Ga\labelto{j} C^\Ga\labelto{\delta}
   \Hr^1\hm A\labelto{i_*} \Hr^1 B\labelto{j_*} \Hr^1 C \labelto{\delta^1} \Hr^2 A;
\]
see Serre \cite[I.5.7, Proposition 43]{Serre}.
We recall the definition of the map $\delta^1$
in our case of the group $\Gamma$ of order 2.

Let  $c\in \Zr^1\hs C\subseteq C$; then $c\upgam\hm c=1$.
We lift $c$ to some element $b\in B$; then $b\hs\upgam\hs b\in A$.
We set $a=b\hs\upgam b$.
We have $\upgam a=\upgam\hs b\cdot b$.
Since $\upgam a\in A$ and $A$ is central in $B$, we have
\[\upgam a=\upgam\hs b^{-1}\hs \upgam a\hs \upgam\hs b=
      \upgam\hs b^{-1}\hs\upgam\hs b\hs\hs b\hs \upgam\hs b=b\hs\upgam b=a.\]
Thus  $a\in \Zr^2A$, and we set
$\delta^1[c]=[a]=[b\hs\upgam\hs b]\in \Hr^2\hm A$.
\end{construction}

Note that when the groups $A$, $B$, and $C$ are abelian,
we have $\Zr^0\hs C=C^\Gamma$ and
$\delta(c)=\delta^0[c]$ for $c\in C^\Gamma$,
where $\delta^0\colon \Hr^0\hs C\to \Hr^1\hm A$
is the map of Section \ref{ss:ABC-abelian}.
Moreover, then our map $\delta^1$ coincides
with the map $\delta^1$ of Subsection \ref{ss:ABC-abelian}.

\begin{notation}
Let $\GG$ be an algebraic $\R$-group, not necessarily abelian. We write
\[\Hr^1\hs\GG=\Hr^1(\R,\GG)\coloneqq \Hr^1(\Gamma,\GG(\C)).\]
\end{notation}

\begin{subsec}
Let
\begin{equation*}
 1\to \AA\labelto{i} \BB\labelto{j} \CC\to 1
 \end{equation*}
be a short exact sequence of real algebraic groups
(not necessary linear or connected).
Then we have a short exact sequence of $\Gamma$-groups
\[ 1\to \AA(\C)\labelto{i} \BB(\C)\labelto{j} \CC(\C)\to 1,\]
whence we obtain a cohomology exact sequence
\begin{equation*}
1\to \AA(\R)\labelto{i} \BB(\R)\labelto{j} \CC(\R)\labelto{\delta}
     \Hr^1\hm\AA\labelto{i_*^1} \Hr^1\BB\labelto{j_*^1} \Hr^1\CC.
\end{equation*}
\end{subsec}

\section{Algorithms for Levi decompositions}\label{app:B}

Let $G\subset \GL(n,\C)$ be a connected complex algebraic group. Then its
Lie algebra $\g$ decomposes as a direct sum of subalgebras
\begin{equation}\label{eq:levidec}
  \g = \sg \oplus \tg\oplus \ng, \text{ where }
\end{equation}
\begin{itemize}
\item $\sg$ is semisimple,
\item $\tg$ is the Lie algebra of a torus,
\item $\ng$ is an ideal of $\g$ consisting of nilpotent elements,
\item $\rg=\tg\oplus \ng$ is the solvable radical of $\g$, and $\ng$ is the
  set of nilpotent elements of $\rg$,
\item $[\sg,\tg]=0$ and $\sg\oplus \tg$ is reductive.
\end{itemize}

See \cite[Section 4.3.3]{wdg} for a proof and algorithms to compute bases
of the subalgebras $\sg$, $\tg$ and $\ng$. Note that $\ng$ is uniquely
determined, but $\sg$ and $\tg$ are not. However, if we have a second
decomposition $\g=\sg'\oplus\tg'\oplus\ng$, then there is an element  $g\in G$ with
$g\sg g^{-1} = \sg'$, $g\tg g^{-1} = \tg'$.
In this section we discuss further algorithmic
problems related to this type of decomposition. Our main algorithms are
contained in Subsections \ref{sec:B3} and \ref{sec:B5}. The other subsections
have solutions to other algorithmic problems which are necessary for the
main algorithms.

The methods of this section are necessary for solving Problems
\ref{mp:1}, \ref{mp:2} for non-reductive groups. We briefly indicate how this
works. Let $\GG$ be a connected non-reductive $\R$-group with Lie algebra $\g$.
Let $\g =\sg\oplus\tg\oplus\ng$ be a decomposition as above. Then
$\sg\oplus \tg$ is the Lie algebra of a maximal connected reductive subgroup
$\GG^{(r)}$ of $\GG$. Our algorithms for the reductive case take this Lie algebra
as input and  so we can compute $\Hr^1 \GG^{(r)}$. Now Sansuc's lemma
(see Proposition \ref{p:Sansuc}) gives a  bijection between
$\Hr^1 \GG^{(r)}$ and $\Hr^1 \GG$.

In our solution to Problem \ref{mp:2} for
$\GG$ we use a reduction to the same problem for $\GG^{(r)}$
(see Subsection \ref{sub:pb02nonred}). To make this
work, we need an explicit homomorphism $\pi^{(r)}\colon \GG \to \GG^{(r)}$. In
Subsection \ref{sec:B5} we show how to construct such a homomorphism.

There is one more application in our paper of the algorithms in the present
section. When solving Problem \ref{mp:1} for a non-connected group, we need to
neutralize a 2-cocycle. In order to do that in the non-reductive case,
we need to be able to explicitly find an element conjugating two
maximal reductive subalgebras of the Lie algebra, see Subsection \ref{sub:2coc}.
In Subsection \ref{sec:B3} we show how to find such a conjugating element.

\begin{subsec}\label{sec:B1}
Let $\bg$ be a solvable Lie algebra over a field of characteristic 0.
Let $\hg_1$, $\hg_2$ be two Cartan subalgebras of $\bg$. Here we describe
an algorithm for finding $x_1,\ldots, x_k\in [\bg,\bg]$ such that
$\exp( \ad x_1)\cdots \exp(\ad x_k)(\h_1) = \h_2$. The algorithm is inspired
by the proof of the existence of such elements, as given for example in
\cite[Theorem 3.6.4]{wdg1}.

For this we need the technical tool of the Fitting decomposition. Let $\ag$ be
a Lie algebra with nilpotent subalgebra $\hg$. Then we define
$$\ag_0(\h) = \{ x\in \ag \mid \text{ for all $y\in \hg$
    there is an integer $t>0$ such that } (\ad y)^t(x)=0\}.$$
Also define $\h^0 (\ag) = \ag$ and $\hg^{i+1}(\ag) = [\hg, \hg^i(\ag)]$. Then
there is $i$ with $\hg^i(\ag) = \hg^{i+1}(\ag)$. We set $\ag_1(\h) =
\hg^i(\ag)$. Then
$$\ag = \ag_0(\h)\oplus \ag_1(\h)$$
(\cite[Section 3.1]{wdg1}, \cite[Chapter II, Theorem 4]{jac}), which is
called the {\em Fitting decomposition} of $\ag$ relative to $\hg$. We have
that both summands are stable under $\ad h$ for $h\in \ag_0(\h)$
(\cite[Proposition 3.2.5]{wdg1}, \cite[Chapter III, Proposition 2]{jac}).
Now suppose that $\hg$ is a Cartan subalgebra of $\ag$. Then $\ag_0(\h)=\h$
(\cite[Section 3.2]{wdg1}, \cite[Chapter III, Proposition 1]{jac}), and so the
Fitting decomposition is $\ag = \hg\oplus \ag_1(\h)$.

We can also consider the 1-dimensional subalgebra $\langle x\rangle$ spanned
by an element $x\in \ag$. Then we write $\ag_0(x) = \ag_0(\langle x\rangle)$.
An element $x\in \ag$ is called {\em regular} if the dimension of $\ag_0(x)$ is
minimal. In this case $\ag_0(x)$ is a Cartan subalgebra of $\ag$
(\cite[Corollary 3.2.8]{wdg1}, \cite[Chapter III, Theorem 1]{jac}).
Let $x\in \ag$ be regular, and set $\h=
\ag_0(x)$. Then $\ag_1(\h)$ is stable under $\ad x$, and $\ad x$ has no
eigenvalue 0 on $\ag_1(\h)$. So the restriction of $\ad x$ to
$\ag_1(\h)$ is bijective. Furthermore, every Cartan subalgebra $\h$ of
$\ag$ contains regular elements, that is, elements $x\in \h$ such that
$\h=\ag_0(x)$ (\cite[Proposition 3.5.2]{wdg1}). The proof of that proposition
also gives a straightforward algorithm to find a regular element in $\h$.
This works as follows. Let $h_1,\ldots,h_\ell$ be a basis of $\h$.
Write $n=\dim \ag$, then there is a homogeneous polynomial $g$ (which we do
not know) of degree $n-\ell$ in $\ell$ indeterminates, such that $\sum_i
c_i h_i$ is regular if and only if $g(c_1,\ldots,c_\ell)\neq 0$. Let $\Omega$
be a subset of the base field of size $2n$ and choose $(c_1,\ldots,c_\ell)\in
\Omega^\ell$ randomly and uniformly. Then by \cite[Corollary 1.5.2]{wdg1}
the probability that $\sum_i c_ih_i$ is {\em not} regular is less than
$\tfrac{1}{2}$. So after a few tries we expect to find a regular element.

Now let $\bg$ be a solvable Lie algebra and let $\h$ be a Cartan subalgebra
of $\bg$. We show that we can find an abelian ideal $\ig$ of $\bg$ such that
\begin{equation}\label{eq:abid}
  \h + \ig \subsetneq \bg \text{ or } \ig = \bg_1(\h).
\end{equation}
Indeed, let $\ig$ be any abelian ideal of $\bg$ (for example the last nonzero
term of the derived series of $\bg$). Suppose that $\h+\ig = \bg$, then we
claim that $\bg_1(\h)$ is an abelian ideal of $\bg$. For that, let $u\in \h$
be a regular element. Let $k$ be such that $(\ad u)^k(\h)=0$. Then, as $\ad u$
is nonsingular on $\bg_1(\h)$, we have $(\ad u)^k(\bg) = \bg_1(\h)$. But also
$(\ad u)^k (\bg) = (\ad u)^k(\h+\ig) \subset \ig$. In particular $\bg_1(\h)$
is abelian, and our claim follows. As a consequence we can replace
$\ig$ by $\bg_1(\h)$ and we have \eqref{eq:abid}.

We revert to the notation at the beginning of this subsection, and let
$\h_1$, $\h_2$ be two Cartan subalgebras of $\bg$. If $\h_1=\h_2$ then
there is nothing to do. So we assume that $\h_1\neq \h_2$ (in particular this
means that $\bg$ is not nilpotent). As shown above, we can find
an abelian ideal $\ig$ of $\bg$ with \eqref{eq:abid} where $\h=\h_2$.
Let $\varphi :
\bg \to \bg/\ig$ be the projection map. By \cite[Lemma 15.4A]{hum}
$\varphi(\h_1)$, $\varphi(\h_2)$ are Cartan subalgebras of $\bg/\ig$.

First suppose that $\ig = \bg_1(\h_2)$; then
$\bg/\ig$ is isomorphic to $\h_2$.
It follows that $\varphi(\h_1)=\varphi(\h_2)$. In particular, we have
$$\bg = \h_2 \oplus \bg_1(\h_2) = \h_1 \oplus \bg_1(\h_2).$$
Let $u\in \h_2$ be regular and write $u=h_1+y$ with $h_1\in \h_1$ and $y\in
\bg_1(\h_2)$. As $[u,\bg_1(\h_2)] = \bg_1(\h_2)$, there is $z\in \bg_1(\h_2)$
with $y=[z,u]$. Such an element $z$ can be found by solving a set of linear
equations. We claim that $z\in [\bg,\bg]$ and $\exp(\ad z) (\h_1) = \h_2$.
The first statement follows from  $[u,\bg_1(\h_2)] = \bg_1(\h_2)$.
Since $\bg_1(\h_2) = \ig$ is a commutative ideal, we have $(\ad z)^2=0$ and
$[z,y]=0$. Hence
$$\exp(\ad z) (h_1) = (1+\ad z)(u-y) = u-y+[z,u-y] = u.$$
Hence $\exp(\ad z)(\h_1)$ is a Cartan subalgebra of $\bg$ having a regular
element in common with $\h_2$, forcing $\exp(\ad z) (\h_1) = \h_2$.
So we are done in this case.

Secondly, suppose that $\h_2 + \ig \subsetneq \bg$. Then
by a recursive call we can find
$\bar x_1,\ldots, \bar x_m\in [\bg/\ig,\bg/\ig]$ such that
$\exp( \ad \bar x_1)\cdots \exp(\ad \bar x_m) (\varphi(\h_1)) = \varphi(\h_2)$.
Let $x_i\in [\bg,\bg]$ be such that $\varphi(x_i) = \bar x_i$\hs, and set
$\h_0 = \exp( \ad x_1)\cdots \exp(\ad x_m) (\h_1)$. Then $\varphi(\h_0) =
\varphi(\h_2)$ so $\h_0,\h_2$ are Cartan subalgebras of $\ag=\h_2+\ig$, which
is of smaller dimension than $\bg$. So by a recursive call we find
$y_1,\ldots,y_n \in [\ag,\ag]$ such that
$\exp(\ad y_1)\cdots \exp(\ad y_n) (\h_0)=\h_2$,
and we return $y_1,\ldots,y_n,x_1,\ldots,x_m$.
\end{subsec}

\begin{subsec}\label{sec:B2}
Here we let $\g$ be a Lie algebra over a field of characteristic 0. Let $\rg$ be its solvable
radical; then by the Levi-Malcev theorem there is a semisimple subalgebra $\sg$ of
$\g$ with $\g = \sg \oplus \rg$. Now let $\sg_1,\sg_2$ be two such subalgebras.
Here we describe an algorithm for finding an element $z\in [\g,\rg]$ such that
$\exp(\ad z)(\sg_1)=\sg_2$. The algorithm is inspired by the proof of the
existence of such an element given in \cite[\S 6, Theorem 5]{bou1}.
We distinguish a few cases.

In the first case we have $[\g,\rg]=0$. Then $\sg_1=\sg_2=[\g,\g]$ and there
is nothing to do.

In the second case we have $[\rg,\rg]=0$ and $[\g,\rg]=\rg$. Then for
$x\in \sg_1$ we write $x=x_2+x_r$ where $x_2\in \sg_2$, $x_r\in \rg$. Define
$h(x) = -x_r$. Then $h : \sg_1\to \rg$ is a linear map. In the proof of
\cite[\S 6, Theorem 5]{bou1} it is shown that there is  $a\in \rg$
with $h(x) = [a,x]$ for $x\in \sg_1$. Such an element $a$ can be found by solving
a system of linear equations. As $\rg=[\g,\rg]$ we have $a\in [\g,\rg]$.
Since $[\rg,\rg]=0$ it follows that $(\ad a)^2=0$. Hence for $x\in \sg_1$ we
see that $\exp(\ad a)(x) = x+[a,x] = x+h(x) =x_2\in \sg_2$. So the element
$a$ does the job.

In the third case we have $[\g,\rg]\neq 0$ and $[\rg,\rg]\neq 0$ or
$[\g,\rg]\neq \rg$. (That is, we are not in the one of the first two cases.)
As $[\g,\rg]$ is nilpotent, its center $\cg$ (which is an ideal of $\g$) is
nonzero. We have $\cg\neq \rg$ as otherwise $[\rg,\rg]=0$ and $[\g,\rg]=\rg$.
Set $\g'= \g/\cg$. The radical of $\g'$ is $\rg'=\rg/\cg$. Let $f :
\g\to \g'$ denote the projection. Then $f(\sg_1)$, $f(\sg_2)$ are semisimple
complements to $\rg'$. So by a recursive call we can find an element $a'\in
[\g',\rg']$ with $\exp( \ad a') f(\sg_1) = f(\sg_2)$. Let $a\in [\g,\rg]$ be
such that $f(a) =a'$ and set $\sg_0 = \exp(\ad a)(\sg_1)$. Set $\ag =
f^{-1}(\sg_2) = \sg_2+\cg$. Then $\ag$ is a proper subalgebra of $\g$
(that is, $\ag\subsetneq\g$)
with solvable radical $\cg$ and semisimple complement $\sg_2$. But also $\sg_0
\subset \ag$ and hence is a semisimple complement to $\cg$. So again by
a recursive call we can find an element $b\in [\ag,\cg]$ with $\exp( \ad b)(\sg_0)
=\sg_2$. As $a\in [\g,\rg]$ and $b\in \cg$ we have $[a,b]=0$ so that
with $z=a+b$ we get $\exp(\ad z) = \exp(\ad b)\exp(\ad a)$. The element
$z$ lies in $[\g,\rg]$ and we have $\exp(\ad z)(\sg_1) = \sg_2$.
\end{subsec}

\begin{subsec}\label{sec:B3}
Let $G\subset \GL(n,\C)$ be a connected algebraic group with Lie algebra
$\g\subset \gl(n,\C)$. Let $\g= \sg_1\oplus\tg_1\oplus \ng$ and
$\g=\sg_2\oplus\tg_2\oplus \ng$ be two decompositions of the form
\eqref{eq:levidec}. Here we show how to find $x_1,\ldots,x_k\in \ng$ such
that with $\psi = \exp(\ad x_1)\cdots \exp(\ad x_k)$ we have
$\psi(\sg_1) = \sg_2$ and $\psi(\tg_1)=\tg_2$.

First we use the algorithm of Subsection \ref{sec:B2} to find $x_k\in [\g,\rg]
\subset \ng$ such that $\exp(\ad x_k)(\sg_1) = \sg_2$. Set $\tg_0 =
\exp(\ad x_k)(\tg_1)$. Also set $\hat\g = \zg_{\g}(\sg_2)$, the centralizer of
$\sg_2$ in $\g$. Then $\hat\g = \tg_0 \oplus \zg_{\ng}(\sg_2) = \tg_2\oplus
\zg_{\ng}(\sg_2)$. As $\tg_0$, $\tg_2$ are commuting subalgebras consisting
of semisimple elements, they lie in Cartan subalgebras $\h_0$, $\h_2$ of
$\hat\g$; see \cite[Lemma 4.14.2]{wdg1}. So by the algorithm of
Subsection \ref{sec:B1} we can find $x_1,\ldots,x_{k-1}\in [\hat\g,\hat\g]\subset
\zg_{\ng}(\sg_2)$ with $\exp(\ad x_1)\cdots \exp(\ad x_{k-1})(\h_0)=\h_2$.
As $x_i\in \zg_{\ng}(\sg_2)$ we see that $\exp(\ad x_i) (\sg_2) = \sg_2$.
We have that $\tg_0$, $\tg_2$ are the sets of semisimple elements of $\h_0$, $\h_2$
respectively. Hence $\exp(\ad x_1)\cdots \exp(\ad x_{k-1})(\tg_0)=\tg_2$.
We conclude that $x_1,\ldots,x_k$ is a
sequence that we are looking for.

Also we remark that for a nilpotent $u\in \g$ we have $\exp(\ad u) (x) =
(\exp u) x (\exp u)^{-1}$. So if we set $g=\exp(x_1)\cdots \exp(x_k)$ then
$g\sg_1 g^{-1} = \sg_2$ and $g\tg_1g^{-1} = \tg_2$.
\end{subsec}

\begin{subsec}\label{sec:B4}
Let $G$, $\g$ be as in Subsection \ref{sec:B3}. Let $\g= \sg\oplus\tg\oplus \ng$
be a decomposition of the form \eqref{eq:levidec}. Here we show that given a
Cartan subalgebra $\h_0$ of $\g$ we can find a Cartan subalgebra $\h_{\sg}$ of
$\sg$, a Cartan subalgebra $\h$ of $\g$ such that $\h_{\sg}\oplus\tg\subset
\h$ and $x_1,\ldots,x_k\in \ng$ such that $\exp(\ad x_1)\cdots \exp(\ad x_k)
(\h_0)=\h$.

Let $\rg=\tg\oplus \ng$ be the solvable radical of $\g$. Let $f : \g\to
\g/\rg$ be the projection map. Then $f(\h_0)$ is a Cartan subalgebra of
$\g/\rg$ (\cite[Lemma 15.4A]{hum}). The restriction of $f$ to $\sg$ is
an isomorphism. Let $\h_\sg$ be the inverse image of $f(\h_0)$ in $\sg$.
Then $\h_\sg$ is a Cartan subalgebra of $\sg$. Write $\ug = \h_\sg\oplus \tg$,
and let $\ag = \zg_{\g}(\ug)$ be its centralizer in $\g$. Let $\h$ be a Cartan
subalgebra of $\ag$. By \cite[Theorem 4.4.4.8]{win} $\h$ is a Cartan subalgebra
of $\g$. It is clear that $\h$ contains $\ug$. We have $f(\h_0) =
f(\h_{\sg})\subset f(\h)$ and $f(\h)$ is a Cartan subalgebra of $\g/\rg$. It
follows that $f(\h) = f(\h_0)$. Set $\bg = \h+\rg$; then also $\h_0\subset\bg$.
Since $\bg/\rg$ is isomorphic to $\h_{\sg}$ which is commutative, it follows
that $[\bg,\bg]\subset \rg$ so that $\bg$ is solvable. Furthermore, $\h$ and
$\h_0$ are Cartan subalgebras of $\bg$. So by the algorithm of Subsection
\ref{sec:B1} we can find $x_1,\ldots,x_k\in [\bg,\bg]$ such that
$\exp(\ad x_1)\cdots \exp(\ad x_k)(\h_0)=\h$.

Finally, since $x_i \in [\g,\g]\cap \rg$, it follows that the elements $x_i$ are
nilpotent elements of $\rg$ (\cite[\S 5, Theorem 1]{bou1}). This in turn
implies that they lie in $\ng$.
\end{subsec}

\begin{subsec}\label{sec:B5}
Let $G$ and $\g$ be as in Subsection \ref{sec:B3}. Let $\g= \sg\oplus\tg\oplus \ng$
be a decomposition of the form \eqref{eq:levidec}. Let $G^\rr$ and $G^\uu$ denote
the connected subgroups of $G$ whose Lie algebras are $\sg\oplus \tg$ and
$\ng$ respectively. Then $G^\uu$ is the unipotent radical of $G$ and
$G=G^\rr\ltimes G^\uu$. Here we consider the projection $\pi : G \to G^\rr$ and
give an algorithm for computing $\pi(g)$ for $g\in G$.

A first remark is that we know the differential $\mathrm{d} \pi : \g \to
\sg\oplus \tg$. We have $\mathrm{d}\pi (x+y) = x$, where $x\in \sg\oplus \tg$
and $y\in \ng$. Let $g\in G$. We can compute its Jordan decomposition
$g=su$ where $s$ is semisimple and $u$ is unipotent. Then $\pi(g) =
\pi(s)\pi(u)$. Moreover, $z=\log(u)\in \g$ and $\pi(u) = \exp(
\mathrm{d}\pi (z) )$. So we are left with computing $\pi(s)$
for a semisimple element $s\in G$.

For this we work with Cartan subgroups of $G$. By definition a {\em Cartan
subgroup} of $G$ is the centralizer of a maximal torus in $G$
(\cite[Definition 28.5.1]{tauyu}). Cartan subgroups have the following properties:
\begin{enumerate}
\item A Cartan subgroup is connected (\cite[28.5.2]{tauyu}).
\item Let $T$ be a maximal torus of $G$ and $H=\cZ_G(T)$ the corresponding
  Cartan subgroup. Then $H=T\times H_u$, where $H_u$ is the set of unipotent
  elements of $H$  (\cite[28.5.2]{tauyu}).
\item A subalgebra $\h$ of $\g$ is a Cartan subalgebra of $\g$ if and only
  if it is the Lie algebra of a Cartan subgroup
  (\cite[Proposition 29.2.5]{tauyu}).
\end{enumerate}

Let $s\in G$ be semisimple, and consider its centralizer $\cZ_G(s)$. The
Lie algebra of this group is $\zg_{\g}(s) = \{x\in \g \mid xs=sx\}$.
The element $s$ lies in a maximal
torus $T$ of $G$. The corresponding Cartan subgroup $\cZ_G(T)$ is contained in
$\cZ_G(s)$. Hence $\zg_{\g}(s)$ contains a Cartan subalgebra of $\g$. This implies
that any Cartan subalgebra of $\zg_{\g}(s)$ is a Cartan subalgebra of $\g$.
Let $\h_0$ be a Cartan subalgebra of $\zg_{\g}(s)$. Let $H_0 = T_0\times U_0$ be
the corresponding Cartan subgroup. Then $s\in T_0$ as the latter is a maximal
torus of $\cZ_G(s)$. Using the algorithm of Subsection \ref{sec:B4} we can
find a Cartan subalgebra $\h_\sg$ of $\sg$, a Cartan subalgebra $\h$ of
$\g$ containing $\sg\oplus \tg$ and $x_1,\ldots,x_k\in \ng$ with
$\exp(\ad x_1)\cdots \exp(\ad x_k)(\h_0) = \h$. Set $h=\exp(x_1)\cdots
\exp(x_k)$. Then $h\h_0 h^{-1} =\h$. Let $H$ be the Cartan subgroup of $G$
corresponding to $\h$. As Cartan subgroups are connected, we also have
$hH_0 h^{-1} = H$. The set of semisimple elements of $\h$ is exactly
$\h_{\sg}\oplus \tg$. So if we write $H=S\times U$, where $S$ is a maximal torus
and $U$ consists of unipotent elements, then $S\subset G^\rr$.
The $\exp(x_i)$ lie in $G^\uu$, so that $\pi(h)=1$. We see that
$\pi(s) = \pi(hsh^{-1}) = hsh^{-1}$ as $hsh^{-1}\in S$.

\begin{remark}\label{rem:iseltof}
Using similar ideas we can devise an algorithm for testing whether an element
$g\in \GL(n,\C)$ lies in a connected linear algebraic group $G\subset
\GL(n,\C)$, given only its Lie algebra $\g\subset \gl(n,\C)$. Indeed,
we compute the Jordan decomposition $g=su$, where $s\in \GL(n,\C)$ is
semisimple and $u\in \GL(n,\C)$ is unipotent. Then $g\in G$ if and only if
both $s,u\in G$. But $u\in G$ if and only if $\log(u)\in \g$. So we are left
with deciding whether $s\in G$. We check whether $s\g s^{-1} = \g$; if not,
then $s\not\in G$. Otherwise, we compute a Cartan subalgebra $\h_0$ of the centralizer
$\zg_{\g}(s)$ of $s$ in $\g$. By (2) above, we have $\h_0 = \tl_0\oplus\ug_0$ where $\tl_0$, $\ug_0$
are the sets of semisimple, respectively nilpotent elements of $\h_0$.
By computing Jordan decompositions of elements of $\h_0$, we
find a basis of the subspace $\tl_0$.
The Lie algebra $\zg_{\g}(s)$ is
the Lie algebra of the algebraic group $Z_G(s)$.
The Cartan subalgebra $\h_0$ of $\zg_{\g}(s)$
is therefore the Lie algebra of a Cartan subgroup of $Z_G(s)$. This
Cartan subgroup is the centralizer of a unique maximal torus $T_0$ of $Z_G(s)$.
The Lie algebra of $T_0$ therefore has to be $\tl_0$.
If $s\in G$, then, as above,
it follows that $s\in T_0$. Hence $s\in G$ if and only if $s\in T_0$.
The latter can be decided using the methods of Subsection \ref{tor:1} below.
\end{remark}
\end{subsec}

\section{Algorithms for algebraic tori}\label{sec:tori}

In this section we describe a few solutions to algorithmic problems for
algebraic tori. By definition, a torus is a connected diagonalizable
algebraic subgroup of $\GL(n,\C)$.

Many of our algorithms use tori. A basic operation is the construction of an
explicit homomorphism of a torus to a product of ``standard'' tori. This is
performed in Subsection \ref{tor:3}. This gives an immediate algorithm for
computing the first Galois cohomology set of a torus, which is explained in
the same subsection.

In Sections \ref{s:connected} and \ref{s:connected-equivalence}
we will solve Problems \ref{mp:1} and \ref{mp:2} for connected reductive groups
by reducing them to the same problems
for the maximal torus containing a maximal compact torus, and then
acting with the corresponding Weyl group. For both computations the algorithms
given in this section are essential. Also the algorithm for computing the
abelian $\Hr^2$ for a quasi-torus, given in Section \ref{s:H2-quasi-torus},
uses the decomposition of a torus into standard tori.

Throughout we assume that
a torus $T$ is given by a basis  $B_\tg$ of its Lie algebra $\tl \subset \gl(n,\C)$.

Our algorithms use computations in lattices. A {\em lattice} is a finitely
generated subgroup of $\Z^n$.
We write the elements of $\Z^n$ as row vectors, and use the
convention that for  $v\in \Z^n$ its components are denoted $v_1,\ldots,v_n$,
that is, $v=(v_1,\ldots,v_n)$.

A lattice $L\subset \Z^n$ has a basis, that is, a finite
set $\{v^1,\ldots,v^m\}\subset L$ such that each $v\in L$ can uniquely
be written as $v=k_1v^1+\cdots +k_m v^m$ where $k_i\in \Z$. To this basis
we associate the matrix with rows $v^1,\ldots,v^m$. The integer $m$ is
called the {\em rank} of $L$.

\begin{subsec}\label{lat:1}
Let $A$ be an integral $m\times n$-matrix of rank $m$, where $m\leq n$.
Denote the first nonzero entry of row $i$ by $A(i,j_i)$.
Then $A$ is said to be in {\em Hermite normal form} if $j_1 < j_2< \cdots <
j_m$, \,$A(i,j_i)>0$ for $1\leq i\leq m$, and $0\leq A(k,j_i) < A(i,j_i)$ for
$1\leq k <i\leq m$.

Let $B$ be the matrix associated to a basis of the lattice $L\subset \Z^n$ of
rank $m$. Then there is a unique $m\times n$-matrix $A$ in Hermite
normal form such that $A=PB$ where $P\in \GL(m,\Z)$;
see \cite[\S 8.1, Proposition 1.1, Corollary 1.2]{sims}.
The rows of $A$ also form a basis of $L$. Moreover $A$ and $P$ can be computed
by integer row operations; see \cite[\S 8.1]{sims}.

Let $A$ be an integral $m\times n$-matrix of rank $m$ where $m\leq n$.
Then $A$ is said to be in {\em Smith normal form} if $A(i,i)$ are positive
for $1\leq i\leq m$, $A(i,i)$ divides $A(i+1,i+1)$ for $1\leq i\leq m-1$, and
$A$ has no other nonzero entries.

Let $B$ be the matrix associated to a basis of the lattice $L\subset \Z^n$ of
rank $m$. Then there is a unique $m\times n$-matrix $A$ in Smith
normal form such that $A=PBQ$ where $P\in \GL(m,\Z)$, $Q\in \GL(n,\Z)$;
see \cite[\S 8.2, Proposition 3.2]{sims}. Furthermore, $A$ and $P,Q$
can be computed by integer row and column operations; see \cite[\S 8.2]{sims}.
Write $d_i = A(i,i)$ and define the map $f : \Z^n\to (\Z/d_1\Z)\times
\cdots \times (\Z/d_m\Z) \times \Z^{n-m}$ in the following way. For $v\in\Z^n$
we set $w= vQ$; then $f(v) = ([w_1]_{d_1},\ldots,[w_m]_{d_m},w_{m+1},\ldots, w_n)$.
Then $f$ is a homomorphism with kernel $L$
(\cite[\S 8.1, Proposition 3.3]{sims}), so that $f$ induces an
isomorphism
$$f : \Z^n/L \to (\Z/d_1\Z)\times \cdots \times (\Z/d_m\Z) \times \Z^{n-m}.$$
Let $u^1,\ldots,u^n$ denote the rows of $Q^{-1}$. It follows that they form a
basis of $\Z^n$ such that $d_i u_i\in L$ for $1\leq i\leq m$.
\end{subsec}

\begin{subsec}\label{lat:2}
A lattice $L\subset \Z^n$ is said to be {\em pure} if $\Z^n/L$ is torsion-free.
So $L$ is pure if and only if the Smith normal form of the matrix
associated to a basis of $L$ has only 1's on the diagonal.

\begin{lemma}\label{lem:Btr}
Let $L\subset \Z^n$ be a pure lattice of rank $m$. Let $B$ be the matrix
associated to a basis of $L$. Then the rows of $B^T$ span $\Z^m$.
\end{lemma}

\begin{proof}
Let $A=PBQ$ be the Smith normal form of $B$. As shown above, we have
$A=( I_m ~|~ 0_{m\times (n-m)})$ where $I_m$ is the $m\times m$-identity matrix and
$0_{m\times (n-m)}$ is the $m\times (n-m)$-matrix whose entries are all 0.
Transposing, we get
$$Q^TB^T = \begin{pmatrix} P^{-T} \\ 0_{(n-m)\times m} \end{pmatrix}.$$
Since $Q^T\in \GL(n,\Z)$, the rows of $Q^TB^T$ span the same lattice as
the rows of $B^T$. Since $P^{-T} \in \GL(m,\Z)$, its rows span $\Z^m$.
\end{proof}
\end{subsec}

\begin{subsec}\label{lat:3}
Now we describe an algorithm for the following. Let $L\subset \Z^n$ be
a lattice of rank $m$ with an automorphism $\tau$ of order 2. The algorithm
finds a basis $e^1,\ldots,e^r$, $f^1,\ldots,f^s$, $g^1,\ldots,g^t$,
$h^1,\ldots,h^t$ of $L$ such that $\tau(e^i)=e^i$, $\tau(f^j)=-f^j$,
$\tau(g^k)=h^k$, $\tau(h^k)=g^k$.
This, in particular, proves Proposition \ref{p:indecomp}.
 The algorithm closely follows the proof
of the theorem asserting the existence of such a basis in \cite[Appendix A]{BT*}.
The input to the algorithm is the $m\times m$-matrix $T$ of $\tau$ with respect
to a given basis $v^1,\ldots,v^m$ of $L$. So instead of working with
$L$, we work with the coefficient vectors of its elements in $\Z^m$.
For such a coefficient vector $u\in \Z^m$ we also write $\tau(u)$ for the
vector $Tu$.
The algorithm is recursive, that is, the procedure calls itself with
lattices of smaller rank as input. It takes the following steps:
\begin{enumerate}
\item If $m=1$, then $L$ is spanned by $v$ with $\tau(v)=v$ or $\tau(v)=-v$
  so there is nothing to do.
\item From now on we assume $m\geq 2$. If $\tau(v)=-v$ for all $v\in L$,
  there is nothing to do. Otherwise we compute an element $e\in \Z^m$ such that
  $\tau(e)=e$. We may assume that the components of
  $e$ have greatest common divisor equal to 1. In other words, the lattice
  $L'\subset \Z^m$ spanned by $e$  is pure.
\item Let $B$ be the matrix with single row $e$. Let $A=PBQ$ be its Smith
  normal form. Let $u^1,\ldots,u^m$ be the rows of $Q^{-1}$. Then
  $u^1\in L'$ and the images of $u^2,\ldots,u^m$ span $\Z^m/L'$.
  Using this basis, we compute the matrix $T'$ of $\tau$ acting on $\Z^m/L'$.
\item  We call  the procedure recursively with input $T'$. The result is
  a basis of $\Z^{m-1}$. By taking the corresponding linear combinations of
  $u^2,\ldots,u^m$, we obtain elements $e_i$, $f_j$, $g_k$, $h_k$ of
  $\Z^m$ with the stated properties modulo $L'$.
\item As shown in \cite[Appendix A]{BT*}, we have $\tau(e_i) = e_i$.
\item We have $\tau(g_k) = h_k + le$ for some $l\in \Z$. We replace
  $h_k$ by $h_k+le$, after which $\tau(g_k) = h_k$, $\tau(h_k)=g_k$.
\item We have $\tau(f_j) = -f_j + le$ for a certain $l\in \Z$.
  If $l=2k$ is even, then after replacing $f_j$ by $f_j-ke$ we have
  $\tau(f_j) = -f_j$. If $l=2k+1$ is odd, then replacing $f_j$ by
  $f_j-ke$ gives $\tau(f_j) = -f_j+e$.
\item Let $f_1,\ldots,f_q$ be all $f_j$ with $\tau(f_j) = -f_j+e$. If $q$ is
  0, then we are done and we add the vector $e$ to the set of elements $e_i$.
  Otherwise, for
  $2\leq j\leq q$ replace $f_j$ by $f_j-f_1$, so that $\tau(f_j) = -f_j$
  for $2\leq j\leq m$. Set $g=f_1$ and $h=-f_1+e$ so that $\tau(g) = h$
  and $\tau(h)=g$.
\end{enumerate}
\end{subsec}

\begin{subsec}\label{tor:1}
Let $T\subset \GL(n,\C)$ be a torus. We suppose that we know a basis $B_\tg$ of its
Lie algebra $\tl\subset \gl(n,\C)$. Here we describe a method to compute
an explicit isomorphism $\mu : (\C^\times)^d \to T$, and a method to compute
the inverse image of an element of $T$ relative to $\mu$.

First we assume that $\tl$ and hence $T$ are diagonal. We also assume that the
entries of the elements of the basis $B_\tl$ of  $\tl$ are elements of an algebraic number field
$F\supset \Q$.
 We consider the lattice
$$\Lambda =\Big \{ e=(e_1,\ldots,e_n)\in \Z^n \ \big|\  \sum_{i=1}^n e_i a(i,i) = 0
     \text{ for all } a\in B_\tl \Big\}$$
where $a(i,i)$ are the diagonal entries of the matrix $a$.
By writing the  diagonal entries $a(i,i)$ as linear combinations of a
$\Q$-basis of $F$, we can describe $\Lambda$ as the set of integral solutions
to a system of linear equations with coefficients in $\Q$. By an algorithm
based on the Smith normal form algorithm (the purification algorithm,
see \cite[\S 6.2]{wdg}\hs)
it is then possible to compute a basis of $\Lambda$. We note that $\Lambda$ is
pure. We have
\begin{align*}
\tl &= \Big\{ \diag(a_1,\ldots,a_n) \ \big|\ \sum_{i=1}^n e_i a_i = 0 \text{ for all }
e\in \Lambda\Big\},\\
T &= \Big\{ \diag(\alpha_1,\ldots,\alpha_n) \ \big|\  \prod_{i=1}^n \alpha_i^{e_i} = 1
\text{ for all } e\in \Lambda\Big\}
\end{align*}
(see \cite[Example 4.2.5]{wdg}). Now define
$$\Lambda^{\perp} = \Big\{ (m_1,\ldots,m_n)\in \Z^n \ \big|\  \sum_{i=1}^n m_ie_i= 0
\text{ for all } e\in \Lambda\Big\}.$$
Again by the purification algorithm we can compute a
basis of $\Lambda^\perp$.

Let $m=(m_1,\ldots,m_n)\in \Lambda^\perp$. Then
for all $t\in \C^\times$ we have $\diag(t^{m_1},\ldots,t^{m_n})\in T$. In other words,
the map $\lambda_m : \C^\times\to T$, $\lambda_m(t) = \diag(t_1^{m_1},\ldots,
t_n^{m_n})$ is a cocharacter of $T$. Let $m^1,\ldots,m^d$ be a basis of
$\Lambda^\perp$. Then we define $\lambda : (\C^\times)^d\to T$ by $\lambda(t_1,\ldots,
t_d) = \lambda_{m^1}(t_1)\cdots \lambda_{m^d}(t_d)$. By Lemma \ref{lem:invlam}
below this map has a regular inverse, and therefore it is an isomorphism of
algebraic groups.

We need an algorithm to compute the inverse images of $\lambda$.
Let $M$ be the matrix with rows $m^1,\ldots,m^d$. By Lemma \ref{lem:Btr}
there is a matrix $P\in \GL(n,\Z)$ such that
$$PM^T = \begin{pmatrix} I_d \\ 0_{(n-d)\times d}\end{pmatrix}.$$
This matrix $P$ can be computed by the Hermite normal form algorithm.
The next lemma provides an immediate method for computing inverse images.
As a byproduct it gives a method for testing whether a given diagonal
matrix lies in $T$.

\begin{lemma}\label{lem:invlam}
Let $A=\diag(\alpha_1,\ldots,\alpha_n)\in \GL(n,\C)$. For $1\leq j\leq n$ set
$t_j = \prod_{k=1}^n \alpha_k^{P(j,k)}$. Then $A\in T$ if and only if
$t_{d+1}=\ldots =t_n=1$. In this case $A=\lambda(t_1,\ldots,t_d)$.
\end{lemma}

\begin{proof}
Note that $A\in T$ if and only if the equations
\begin{equation}\label{eq:teq1}
\prod_{i=1}^d t_i^{M^T(l,i)} = \alpha_l \text{ for } 1\leq l\leq n
\end{equation}
have a solution. Let $R$ be an $n\times n$ integral matrix, and suppose
that the complex numbers $t_i\in \C^\times$ satisfy \eqref{eq:teq1}. Then
a straightforward calculation shows that the $t_i$ satisfy the equations
\begin{equation}\label{eq:teq2}
  \prod_{i=1}^d t_i^{(RM^T)(l,i)} =\, \prod_{j=1}^n \alpha_j^{R(l,j)}.
\end{equation}
Suppose that $R\in \GL(n,\Z)$ and that the elements $t_i$ satisfy \eqref{eq:teq2}.
Then a similar calculation shows that they satisfy \eqref{eq:teq1}.
Hence the two sets of equations are equivalent in that case.

Now we take $R=P$. Then we see that $t_1,\ldots,t_d$ satisfy \eqref{eq:teq1}
if and only if $t_j = \prod_{k=1}^n \alpha_k^{P(j,k)}$ for $1\leq j\leq d$ and
$1= \prod_{k=1}^n \alpha_k^{P(j,k)}$ for $d+1\leq j\leq n$.
\end{proof}

Now we let $T$ be not necessarily diagonal. We can compute a matrix
$C\in \GL(n,\C)$ such that $T'=C^{-1}T C$ and $\tl'=C^{-1}\tl\hs C$ consist of
diagonal matrices. As outlined in Subsection \ref{tor:1} we can compute an isomorphism
$\lambda' : (\C^\times)^d \to T'$, so we also have the isomorphism
$\lambda : (\C^\times)^d\to T$ by $\lambda(t_1,\ldots,t_d) = C\lambda'(t_1,\ldots,
t_d)C^{-1}$. Furthermore, by Lemma \ref{lem:invlam} we can compute the
inverse of $\lambda'$ and hence the inverse of $\lambda$, and we can
decide whether a given element  $g\in\GL(n,\C)$ lies in $T'$.
Hence we can decide the same for $T$.
\end{subsec}

\begin{subsec}\label{tor:3}
Let $T\subset \GL(n,\C)$ be a torus equipped with
an anti-regular involution $\sigma : T\to T$, $\sigma(g) = \upgam g$.
In order to work explicitly with $\sigma$, we assume that it is given by
$\upgam g =N_\sigma\bar{g} N_\sigma^{-1}$, where $N_\sigma\in \GL(n,\C)$ satisfies
the cocycle condition $N_\sigma\overline{N}_\sigma=1$
and $g\mapsto \bar g$ is the standard complex conjugation of the entries of $g$.

Here we describe how to explicitly write $T$ as a product of indecomposable tori
as in Corollary \ref{c:indecoposable-tori}. This leads to an immediate algorithm
for computing $\Hr^1 \TT$.

Let $T'$, $C$, $\lambda,\lambda'$ be as in Subsection \ref{tor:1}. We want to
describe $\sigma$ in terms of $\lambda$. For this we write
$\upgam \lambda(t_1,\ldots,t_d) =\lambda(\bar v_1,\ldots,\bar v_d)$,
 where $\bar v_i = v_i(\bar t_1,\ldots,\bar t_d)$.
 On the one hand, we have $\lambda(\bar v_1,\ldots,\bar v_d)=
C\lambda'(\bar v_1,\ldots,\bar v_d)C^{-1}$. On the other hand,
$\upgam \lambda(t_1,\ldots,t_d) = N_\sigma\overline{C}\lambda'(\bar t_1,\ldots,
\bar t_d)\overline{C}^{-1}N_\sigma^{-1}$.
Set $B=C^{-1}N_\sigma\overline{C}$; since the numbers
$t_i$ are arbitrary,
we can pass from $\bar t_i$ to $t_i$ and conclude
\begin{equation}\label{eq:conj}
\lambda'(v_1,\ldots,v_d) = B\lambda'(t_1,\ldots,t_d)B^{-1}
\end{equation}
where $v_i = v_i(t_1,\ldots,t_d)$. The maps $\lambda(t_1,\ldots,t_d)\mapsto
t_i$ are characters of $T$ generating its character group. Also the maps
$\lambda(t_1,\ldots,t_d)\mapsto v_i(t_1,\ldots,t_d)$ are characters of $T$.
Hence each $v_i(t_1,\ldots,t_d)$ is a monomial in the variables $t_i$ (where each
$t_i$ appears with an integral exponent). These monomials
can be computed using \eqref{eq:conj}.
We define the $d\times d$-matrix $\tau$ by
$$v_i(t_1,\ldots,t_d) = t_1^{\tau(i,1)}\cdots t_d^{\tau(i,d)}.$$
Then $\tau$ is the matrix of an involution of $\Z^d$,
which by abuse of notation we also denote by $\tau$.

Now we show how to write $T$ as a product of indecomposable tori.
Using the algorithm of Subsection \ref{lat:3},
we compute a basis $f^1,\ldots,f^k$, $e^1,\ldots,e^l$, $g^1,\ldots,g^r$,
$h^1,\ldots,h^r$ of $\Z^d$ such that $\tau f^i = -f^i$, $\tau e^i = e^i$,
$\tau g^i=h^i$, $\tau h^i = g^i$. Let $A$ be the matrix with columns
$$f^1,\ldots,f^k, e^1,\ldots,e^l, g^1,h^1,\ldots,g^r,h^r.$$
Then $A$ is an invertible integral matrix such that
$\nu=A^{-1}\tau A$ has diagonal elements $-1$, $1$ and $\begin{pmatrix}
  0&1\\1&0\end{pmatrix}$.
Define a homomorphism $\mu : (\C^\times)^d\to T$ by
$$\mu(u_1,\ldots,u_d) = \lambda( \prod_{j=1}^d u_j^{A(1,j)},\ldots,
\prod_{j=1}^d u_j^{A(d,j)})\hs.$$
If we set $u_j = \prod_k t_k^{A^{-1}(j,k)}$, then
$\mu(u_1,\ldots,u_m) = \lambda(t_1,\ldots,t_m)$. Hence $\mu$  is an
isomorphism. So we can  represent elements of $T$ as $\mu(u_1,\ldots,
u_d)$. Secondly, we have algorithms to switch between the representations of
elements of $T$ as $\lambda(t_1,\ldots,t_d)$ and $\mu(u_1,\ldots,u_d)$.

A straightforward computation shows that
$$\upgam \mu(u_1,\ldots,u_d) = \mu(\prod_{j=1}^d \bar u_j^{\nu(1,j)},\ldots,
\prod_{j=1}^d \bar u_j^{\nu(d,j)}).$$
Define $\mu_i : \C^\times \to T$ by
$\mu_i(u) = \mu(1,\ldots,1,u,1,\ldots,1)$ where $u$ appears at the $i$-th place.
Also set $\phi_i = \mu_i$ for $1\leq i\leq k$, $\epsilon_i =
\mu_{k+i}$ for $1\leq i\leq l$ and define $\delta_i :(\C^\times)^2 \to T$ by
$\delta_i(u,v)= \mu_{k+l+2i-1}(u)\mu_{k+l+2i}(v)$.
Let $F_i$, $E_i$, and $D_i$ be the subtori that are the images of,
respectively, the homomorphisms $\phi_i$, $\epsilon_i$ and $\delta_i$.
Then $T$ is the product of these subtori. Moreover,
$$\upgam\phi_i(u) = \phi_i(\bar u^{-1}),\,\upgam\epsilon_i(u) =
     \epsilon_i(\bar u),\,\upgam \delta_i(u,v) = \delta_i(\bar v, \bar u).$$
Hence each of these subtori is defined over $\R$ and is isomorphic to one of the tori given in
Corollary \ref{c:indecoposable-tori}.
We denote the corresponding $\R$-tori by $\FF_i$\hs, $\EE_i$\hs, and $\DD_i$\hs.

It is straightforward to compute the Galois cohomology of each of these
indecomposable tori. We have that $\Hr^1 \FF_i =
\{ [\phi_i(1)], [\phi_i(-1)]\}$ and $\Hr^1 \EE_i = \Hr^1 \DD_i = 1$.
Therefore, $\Hr^1 \TT$ consists of $[\phi_1(\varepsilon_1)\cdots
  \phi_k(\varepsilon_k)]$,
where $\varepsilon_i = \pm 1$. In particular, we have that $|\Hr^1 \TT| = 2^k$.
\end{subsec}

\section{$\Hr^1$ for connected reductive groups: representatives}
\label{s:connected}

In this section we describe our computational methods for solving Problem \ref{mp:1},
that is, for computing $\Hr^1\GG$,
in the case when our $\R$-group  $\GG$ is connected and  reductive.

Let $G\subset \GL(n,\C)$ be a connected reductive algebraic group equipped with
an anti-regular involution $\sigma : G\to G$, $\sigma(g) = \upgam g$.
In order to work explicitly with $\sigma$ we assume that it is given by
$\upgam g =N_\sigma\hs\bar{g} N_\sigma^{-1}$, where $N_\sigma\in \GL(n,\C)$
satisfies the cocycle condition $N_\sigma\overline N_\sigma=1$
and $g\mapsto \bar g$ is the standard complex conjugation of the entries of $g$.
We write $\GG=(G,\sigma)$.

In \cite{Borovoi22-CiM} a construction is given of the first Galois
cohomology set $\Hr^1\GG$. This construction yields an algorithm to
compute $\Hr^1\GG$. Here we describe this algorithm in detail.

The construction runs as follows. Let $\TT_0$ be a maximal compact torus in $\GG$.
Set $\TT=\cZ_\GG(\TT_0)$; then $\TT$ is a maximal torus in $\GG$. Set
$\NN_0 = \cN_\GG(\TT_0)$; then $\NN_0\subseteq\NN\coloneqq \cN_\GG(\TT)$.
We set $\WW_0 = \NN_0/\TT\subseteq \NN/\TT\eqqcolon\WW$.
As usual, we write $W_0=\WW_0(\C)\subset\WW(\C)=W$.
We define a right action of $W_0$ on $\Hr^1 \TT$: let
$n\in N_0$ represent $w\in W_0$, and let $[z]\in \Hr^1 \TT$; then
$$[z]\cdot w = [ n^{-1} z\bar{n} ].$$
In \cite{Borovoi22-CiM} it is shown that this is a well defined action of $W_0$ on $\Hr^1 \TT$, and
Theorem 3.1 of \cite{Borovoi22-CiM} states that the natural map
$\Hr^1 \TT\to \Hr^1 \GG$ induces a bijection between $(\Hr^1 \TT)/W_0$ and $\Hr^1 \GG$.
In order to use this construction we must explain how to compute the various
objects constructed above.

The Lie algebra  $\g = \Lie G$ has a complex conjugation coming from $\sigma$.
Since we assume that $\sigma$ is given by $\upgam g = N_\sigma\hs\bar{g}N_\sigma^{-1}$ for $g\in G$, the complex
conjugation on $\g$ is given by the same formula, $\upgam x = N_\sigma\hs\bar{x}N_\sigma^{-1}$
for $x\in \g$. We set
$$\gR_\rR =\Lie \GG= \{ x \in \g \mid \upgam x = x\}$$
which is a real form of $\g$.
The input to our algorithm is a basis of the real Lie algebra $\gR_\rR$.

We have $\gR_\rR = \sgR_\rR \oplus \zgR_\rR$, where
$\sgR_\rR=[\gR_\rR,\gR_\rR]$ is a semisimple Lie algebra and
$\zgR_\rR$ is the center of $\gR_\rR$ consisting
of commuting diagonalizable (over $\C$) matrices.
Then $\zgR_\rR=\zgR_s\oplus\zgR_c$
where $\zgR_c$ is compact (consists of matrices with purely imaginary eigenvalues) and $\zgR_s$ is
split  (consists of matrices with real eigenvalues).
In \cite{dfg} algorithms are described for computing a Cartan decomposition
$\sgR_\rR = \kg\oplus\pg$ of $\sgR_\rR$.
Let $\hat \tgR_{\rR,0}$ be a Cartan subalgebra of $\kg$. Let $\tgR_{\rR,0}$ be the sum of
$\hat \tgR_{\rR,0}$ and the compact part  $\zgR_c$ of $\zgR_\rR$.
Then $\tgR_{\rR,0}$ is the Lie algebra
of a maximal compact torus $\TT_0$ of $\GG$. Let $\tgR_\rR$ be the centralizer of
$\tgR_{\rR,0}$ in $\gR_\rR$; then $\tgR_\rR$ is the Lie algebra
of the maximal torus $\TT=\cZ_\GG(\TT_0)$ of $\GG$.
We have $\tgR_\rR\cap\kg=\hat\tgR_{\rR,0}$.

As in Subsection \ref{tor:3}, we compute the maps $\lambda$, $\mu$,
$\phi_i$, $\epsilon_i$, $\delta_i$, and subsequently $\Hr^1 \TT$.
Now we explain how to compute the action of $W_0$ on $\Hr^1 \TT$.

Let $\sg\subset \g$ be the complexification of $\sgR_\rR$.
We compute the root system $\Phi$ of $\g$ with respect to the Cartan
subalgebra $\tl$.  Let $x_1,\ldots,x_\ell$, $y_1,\ldots,y_\ell$, $h_1,\ldots,h_\ell$
be a canonical generating set
of $\sg$ corresponding to a set of simple roots
of $\Phi$;  see \cite[page 126]{jac} or
\cite[\S 2.9.3]{wdg}.
 Then the element $\exp(x_i) \exp(-y_i)\exp(x_i)$ in $G$ induces the
$i$-th simple reflection of the Weyl group $W$. We can write any element of $W$
as a product of simple reflections and thus find an element of
$N=\cN_G(\tl)$ representing it.
Furthermore, $W_0$ is the subgroup of $W$ consisting of the elements of $W$
leaving $\tl_0$ invariant. Since $W$ acts trivially on $\zgR$,
we see that an element $w\in W$ lies in $W_0$ if and only if $w$ leaves
$\hat\tl_0$ invariant. Hence
we can decide whether $w\in W$ lies in $W_0$
by computing an element $n\in N$ representing it and checking whether
$n\hat\tl_0 n^{-1} =\hat\tl_0$.

Now let $w\in W_0$ and $[z]\in \Hr^1 \TT$.
We use the notation of Subsection \ref{tor:3}. Let $n\in N$ represent $w$ and
set $z' = n^{-1} z \bar{n}$. Since we can invert the map $\lambda$ and
switch between $\lambda$ and $\mu$, we can compute $a_i,b_i,u_i,v_i\in \C^\times$
such that
$$z' = \phi_1(a_1)\cdots \phi_k(a_k)\epsilon(b_1)\cdots \epsilon(b_l)
\delta_1(u_1,v_1)\cdots\delta_r(u_r,v_r).$$
Since each $\phi_i(a_i)$ is a cocycle in $\FF_i$, we have that $a_i\in \R$.
If $a_i>0$ then we set $a_i'=1$. If $a_i<0$ then
we set $a_i'=-1$. Then the cocycle $\phi_i(a_i)$ in $\Zr^1\FF_i$
is equivalent to $\phi_i(a_i')$.
Because the Galois cohomology groups of the subtori $\EE_i$ and
$\DD_i$ are trivial, the cocycle $z'$ is equivalent to $\phi_1(a_1')\cdots
\phi_k(a_k')$. So we can compute the permutation action of $W_0$ on $\Hr^1 \TT$.
Of each orbit of this action we take one element. Those form $\Hr^1 \GG$.

\begin{remark}\label{rem:realweyl}
It is possible to compute $W_0$ by running over all elements of $W$ and
to decide for each element whether it belongs to $W_0$. However, when $W$ gets
bigger this becomes more and more difficult to execute in practice.
We explain how to compute a large subgroup of $W_0$ in one stroke.
Consider the adjoint group $\GG^\ad=\GG/Z(\GG)$ (which can be identified with the identity
component of the automorphism group of $\gR_\rR$).
Let $G^\ad_r = \GG^\ad(\R)$ be the group of $\R$-points of $\GG^\ad$
(which can be identified with the subgroup of $G^\ad$
consisting of the elements preserving the real form $\gR_\rR$ of $\g$).
Then $W_r = \cN_{G^\ad_r}(\tgR_\rR)/\cZ_{G^\ad_r}(\tgR_\rR)$ is called the
{\em real Weyl group} of $\tgR_\rR$; see Knapp \cite[(7.92a)]{knapp}.
Let $\theta\colon \sgR_\rR\to \sgR_\rR$ be the involutive automorphism of $\sgR_\rR$
acting as $+1$ on $\kg$ and as $-1$ on $\pg$.
We identify $\sgR_\rR=\Lie \GG^\ad$; then the involution $\theta$ of $\sgR_\rR$
induces an involution $\theta_\gG$
of the adjoint $\R$-group $\GG^\ad$ called a {\em Cartan involution}.
Let $K$ denote the group of fixed points of $\theta_\gG$ in $G^\ad_r$,
which is a compact Lie group with Lie algebra $\kg$.
Then we  have $W_r = \cN_K(\tgR_\rR)/\cZ_K(\tgR_\rR)$; see
Knapp \cite[(7.92b)]{knapp}.
Since all elements of $W_r$ have representatives
in $K$, they preserve both $\tgR_\rR$ and $\kg$, and hence they preserve
$\tgR_\rR\cap\kg=\hat\tgR_{\rR,0}$ and $\hat\tgR_{\rR,0}\oplus\zgR_c=\tgR_{\rR,0}$.
Thus $W_r$ preserves $\TT_0$, whence  $W_r\subseteq W_0$.
We  remark that in \cite{dg21} an algorithm is given to compute $W_r$. So we can consider
the decomposition of $W$ as a union of cosets $w_iW_r$, $1\leq i\leq s$. Then
$w_iW_r$ lies in $W_0$ if and only if $w_i$ lies in $W_0$. So we only need to
check the coset representatives.
\end{remark}

\section{$\Hr^1$ for connected reductive groups: equivalence of cocycles}
\label{s:connected-equivalence}

In this section we describe our computational methods of solving Problem \ref{mp:2}
when $\GG$ is a connected reductive $\R$-group.
Namely, we consider the following problem. Let $G$ be as in Section \ref{s:connected}.
Let $g_1,g_2\in G$ be two cocycles, that is, $g_i \upgam g_i = 1$ for
$i=1,2$. We want to decide whether the cocycles $g_1$, $g_2$ are equivalent,
in other words, whether there exists $g\in G$ with $g^{-1} g_1 \upgam g =
g_2$, and find such an element $g$ in the affirmative case. Note that it is enough
to solve this problem when  $g_2$ is a fixed representative of a class in $\Hr^1 \GG$
that we have computed.

\begin{subsec}\label{equiv:1}
We look at the special case where $\GG$ is a torus. The three elementary
tori are easy to deal with:
\begin{enumerate}
\item Let $\FF$ be a torus consisting of $\phi(u)$ for $u\in \C^\times$ with
$\upgam \phi(u) = \phi(\bar u^{-1})$. Let $\phi(u_0)$ be a $1$-cocycle, which
means that $u_0\in \R$. If $u_0 >0$ then we set $v_0 = \sqrt{u_0}$
and obtain $\phi(v_0)^{-1} \phi(u_0) \upgam \phi(v_0) = \phi(1)$. If
$u_0<0$ then we set $v_0 = \sqrt{-u_0}$ yielding $\phi(v_0)^{-1} \phi(u_0)
\upgam \phi(v_0) = \phi(-1)$
\item Let $\EE$ be a torus consisting of $\epsilon(u)$ for $u\in \C^\times$ with
$\upgam \epsilon(u) = \epsilon(\bar u )$. Let $\epsilon(u_0)$ be a
$1$-cocycle in $\EE$, which is the same as saying that $u_0\bar u_0 =1$.
Write $u_0 = a+bi$, then $a^2+b^2=1$. If $u_0\neq 1$ then set $v_0 = b+(-a+1)i$.
If $u_0=1$ then set $v_0=1$. In both cases,
$\epsilon(v_0)^{-1} \epsilon(u_0) \upgam \epsilon(v_0) = \epsilon(1)$.
\item Let $\DD$ be a torus consisting of $\delta(u,v)$ for $u,v\in \C^\times$ with
$\upgam \delta(u,v) = \delta(\bar v, \bar u)$. Let $\delta(u_0,v_0)$ be a
$1$-cocycle, which is the same as saying $u_0\bar v_0 = 1$. Then
$\epsilon(u_0,1)^{-1}\epsilon(u_0,v_0)\upgam \epsilon(u_0,1) = \epsilon(1,1)$.
\end{enumerate}
In combination with the algorithms given in Subsections \ref{tor:1} and
\ref{tor:3}, this leads to an immediate algorithm for solving our problem for
general tori.
\end{subsec}

\begin{subsec}\label{equiv:2}
Let $\GG$ be a connected reductive $\R$-group.
Let $\TT_0$ and $\TT_0'$ be two maximal compact $\R$-tori in $\GG$.
It is known that there exists an element $g_\rR\in\GG(\R)$ such that
\[g_\rR^{-1}\cdot \TT'_0\cdot g_\rR=\TT_0\hs.\]
We consider the problem of finding such $g_\rR$ {\em assuming that we know an element $g\in G=\GG(\C)$ such that
$g^{-1}\cdot \TT'_0\cdot g=\TT_0$\hs.}

Write $\TT=\cZ_\GG(\TT_0)$, which is a maximal torus of $\GG$.
Write
\[\NN_0=\cN_\GG(\TT_0),\quad\ \NN=\cN_\GG(\TT);\]
then $\NN_0\subseteq \NN$.
We have
\[ g^{-1}\cdot\TT'_0\cdot g=\TT_0,\quad \upgam g^{-1}\cdot\TT'_0\cdot \upgam g=\TT_0,\]
whence
\[ g^{-1}\,\upgam g\cdot \TT_0\cdot\upgam g^{-1}\hs g=\TT_0\hs.\]
Set
\begin{equation}\label{e:coboundary}
z=g^{-1}\hs\upgam g;
\end{equation}
then
\[z\cdot \TT_0\cdot z^{-1}=\TT_0,\]
whence $z\in N_0$. It follows from \eqref{e:coboundary} that $z$ is a 1-cocycle,
that is, $z\in\Zr^1\hs\NN_0$.

Write $g_\rR=g\ntil$ for some $\ntil\in G$; then $\ntil^{-1}\cdot\TT_0\cdot \ntil=\TT_0$, whence $\ntil\in N_0$.
We have
\[1=g_\rR^{-1}\cdot\hm\upgam g_\rR=(g\ntil)^{-1}\cdot\upgam(g\ntil)=
    \ntil^{-1}\cdot g^{-1}\,\upgam g\cdot\hm \upgam n=\ntil^{-1}\cdot z\cdot\hm\upgam\ntil.\]
    Thus $[z]=[1]\in\Hr^1\NN_0$.
Conversely, if we find $\ntil\in N_0$ such that $\ntil^{-1}\cdot z\cdot\hm\upgam\ntil=1$ and put $g_\rR=g\ntil$,
then $g_\rR\in\GG(\R)$ and $g_\rR^{-1}\cdot \TT'_0\cdot g_\rR=\TT_0$\hs.

We have a short exact sequence
\[1\to \TT\labelto{}\NN_0\labelto\pi\WW_0\to 1,\]
where $\WW_0=\NN_0/\TT$.
It is easy to see that the group $W_0\coloneqq \WW_0(\C)$ embeds into $\Aut \TT_0$.
Since $\TT_0$ is compact, the complex conjugation $\gamma$
acts on the cocharacter group $\X_*(\TT_0)$ as $-1$,
and so it acts trivially on $\Aut\TT_0=\GL\big(\X_*(\TT_0)\big)$.
It follows that $\Gamma$ acts trivially on $W_0$, that is,
\[W_0=\WW_0(\C)=\WW_0(\R)\subseteq \WW(\R).\]
Conversely, let  $w\in\WW(\R)$; then $w$
acts on the real torus $\TT$ by a real automorphism,
and it acts on $\TT(\R)$ by a continuous automorphism.
Since $\TT_0$ is the maximal compact subtorus in $\TT$,
we have $w(\TT_0)=\TT_0$, whence $w\in \WW_0(\R)$ and $\WW(\R)\subseteq \WW_0(\R)=W_0$\hs.
Thus $W_0=\WW(\R)$.

We know that $[z]=[1]\in\Hr^1\hs\NN_0$\hs.
Hence,
\[[\pi(z)]=\pi_*[z]=[1]\in\Hr^1\hs\WW_0\hs,\]
that is, $\pi(z)=w^{-1}\hs\upgam w$ for some $w\in W_0$.
Since $\Gamma$ acts trivially on $W_0$, we see that $\pi(z)=1\in W_0$.
Thus $z\in \Zr^1\hs\TT$.

Since $[z]=[1]\in\Hr^1\hs\NN_0$, there exists $\ntil\in N_0$
such that $\ntil^{-1}\cdot z\cdot\hm\upgam \ntil=1$.
We describe a method to find such $\ntil$.

Consider the right action of $N_0$ on $\Zr^1\hs\TT$:
\[t*n=n^{-1}\cdot t\cdot\hm\upgam n\quad\text{for}\ n\in N_0, \ t\in \Zr^1\hs\TT.\]
This action induces a well-defined right action of $W_0$ on $\Hr^1\hs\TT$  (which might not preserve the group structure in $\Hr^1\hs\TT$);
see \cite[Section 3]{Borovoi22-CiM}.
We know that there exists $w_1=n_1\cdot T$ such that $[z]*w_1=[1]\in \Hr^1 \TT$;
 we can find such $n_1$ by brute force,
by computing $\Hr^1 \TT$, $W_0$ and the right action of $W_0$ on $\Hr^1\hs\TT$.
Then we can find $t\in T$ such that
\[t^{-1}\cdot(n_1^{-1} z \hs\upgam n_1)\cdot\hm\upgam t=1.\]
We set $\ntil=n_1t$; then  $\ntil^{-1}\cdot z\cdot\hm\upgam \ntil=1$.
We set $g_\rR=g\ntil$; then
\[g_\rR\in\GG(\R)\quad \text{and} \quad g_\rR^{-1}\cdot \TT'_0\cdot g_\rR=\TT_0\hs,\]
as required.
\end{subsec}

\begin{subsec}\label{equiv:3}
Here we look at a procedure that we need to tackle Problem \ref{mp:2} for
a connected reductive group that is not necessarily commutative.

Let $\GG$ be as in Section \ref{s:connected}. Let $\TT_0\subset \GG$ be a maximal
compact torus in $\GG$. This means that $\TT_0(\R)$ is a compact torus in $\GG(\R)$
and that $\TT_0(\R)$ is not properly contained in another compact torus of
$\GG(\R)$. Let $\CC\subset \GG$ be any compact torus in $\GG$.
It is known that there is  $g\in \GG(\R)$ with $g\hs \CC g^{-1}\subset \TT_0$.
Here we describe an algorithm to find such an element $g$.

For this we may assume that $\GG$ is semisimple. We assume that the group
$\GG$ is given by its Lie algebra $\gR_\rR$, which is given by a basis
consisting of elements fixed under $\gamma$.
Furthermore, $\TT_0$ and $\CC$ are given by
bases of their Lie algebras $\tgR_{\rR,0}$ and $\cgR_\rR$\hs, respectively. The problem is
equivalent to finding an element $g\in G(\R)$ such that $g\hs\cgR_\rR g^{-1} \subset \tgR_{\rR,0}$.
We assume also that we are given a Cartan decomposition $\gR_\rR = \kg\oplus \pg$
such that $\tgR_{\rR,0}\subset \kg$.
With the following steps we find an automorphism $\tau$ of $\gR_\rR$ such that
$\tau(\cgR_\rR)\subset \tgR_{\rR,0}$.
\begin{enumerate}
\item Find a maximally compact ``torus'' $\tgR_{\rR,0}'$ in $\gR_\rR$ containing $\cgR_\rR$.
  An algorithm for that is given in \cite[\S 3.1]{dfg}.
\item Find a Cartan decomposition $\gR_\rR = \kg'\oplus \pg'$ such that
  $\tgR_{\rR,0}' \subset \kg'$. An algorithm for that is given
  in \cite[\S 3.2]{dfg}.
\item Now \cite[\S 7]{dg} has an algorithm for constructing an automorphism
  $\tau : \gR_\rR\to \gR_\rR$ mapping $\kg'$ to $\kg$, $\pg'$ to $\pg$ and
  $\tgR_{\rR,0}'$ to $\tgR_{\rR,0}$.
\end{enumerate}

By extension of scalars we obtain an automorphism $\tau$ of $\g$.
If $\tau$ is inner then \cite{cmt} (see also \cite[\S 5.7]{wdg})
has an algorithm for writing $\tau$ as a product
of $\exp( \ad x)$, where $x\in \g$ is nilpotent. This algorithm fails if
$\tau$ is not inner. So first we use it to decide whether $\tau$ is
inner. If it is not inner then we run over all diagram automorphisms of
$\g$ and find an outer automorphism $\alpha$ of $\g$ such that $\alpha\circ\tau$ is
inner. We replace $\tau$ by $\alpha\circ\tau$.

Now we use the algorithm of \cite{cmt} again to write $\tau$ as a product
of $\exp( \ad x)$, where $x\in \g$ is nilpotent. We let $g\in G$ be equal
to the same product, where we replace $\exp(\ad x)$ by $\exp(x)$.
For nilpotent elements $y\in \g$ and $x\in \g$ we have $\exp( \ad x)(y) = (\exp x) y
(\exp x)^{-1}$. Hence $g\cg g^{-1} \subset \tl_0$.

It can happen, however, that $g\not\in \GG(\R)$.  In that case we follow the
procedure outlined in \ref{equiv:2}. We set $\TT=\cZ_G(\TT_0)$ and $W_0 =
N_G(T_0) /T$ which is a subgroup of the Weyl group of $\g$ with respect to the
Cartan subalgebra $\tl = \zg_\g(\tl_0)$. Then we perform the following steps.
\begin{enumerate}
\item Set $z = g^{-1}\hs \upgam\hm g$.
\item By running over $W_0$ we find $w\in W_0$ such that $[z]*w = [1]$ in
  $\Hr^1 \TT$.
\item We find $n\in N_G(T_0)$ such that $w= nT$ (see Section \ref{s:connected}).
\item By the algorithm of Subsection \ref{equiv:1} we find an element $t\in T$ such that
  $t^{-1}(n^{-1}z \upgam n ) \upgam t = 1$.
\item We replace $g$ by $gnt$. Then $g\in \GG(\R)$ and $g\hs\tl_0'\hs g^{-1} = \tl_0$.
\end{enumerate}
\end{subsec}

\begin{subsec}\label{equiv:4}
Now we consider the general case, that is we let $\GG$ be a connected reductive
group, and suppose that we have computed $\Hr^1 \GG$ as in Section \ref{s:connected}. This
used a maximal compact torus $\TT_0$, which we will use also here. We let
$c_1,\ldots,c_m$ be a fixed set of cocycles, representing the classes in
$\Hr^1 \GG$. We let $g\in G$ be another cocycle; we want to find $i$ with
$1\leq i\leq r$ and $h\in G$ such that $h^{-1}g\upgam h = c_i$.
The approach we use for this follows the proof of surjectivity in
\cite[Theorem 3.1]{Borovoi22-CiM}.

First we compute the multiplicative Jordan decomposition $g=su$ where
$s$ is semisimple, $u$ is unipotent and $su=us$. As $u$ is unipotent,
we have $\log(u) \in \g$. We set
$$u' = \exp( \tfrac{1}{2} \log(u) );$$
then $u'\in G$ and it can be shown that $(u')^{-1} su \upgam u' = s$.
Moreover, $s$ is a cocycle in $\GG$.

Now set $C=\cZ_G(s)$. Then $C$ is reductive, defined over $\R$, and $s$ lies
in the identity component $C^0$ of $C$ (see
\cite[proof of Theorem 3.1]{Borovoi22-CiM}).
We write $\CC$ for the corresponding $\R$-group.
Let $\TT'$ be a maximal torus of $\CC^0$; then $s\in T'$.

We compute the Lie algebra of $\CC$, $\cgR_\rR = \zg_{\gR_\rR}(s)$.
We compute a Cartan subalgebra $\tgR_\rR'$ of $\cgR_\rR$.
Then $\tgR_\rR'$ is the Lie algebra of a maximal
torus $\TT'$ of $\CC^0$. Given $\tgR_\rR'$ we can compute the Galois cohomology
$\Hr^1 \TT'$ (see Subsection \ref{tor:3}). As seen in Subsection \ref{equiv:1},
we can compute an element $t\in T'$ such that $s_1=t^{-1} s \upgam t$ is one of the fixed representatives
of the classes in $\Hr^1 \TT'$. As seen in Subsection \ref{tor:3}, these representatives all
lie in the compact part of $\TT'$. We compute a Cartan decomposition
$\cgR_\rR = \kg'\oplus \pg'$ such that
$\tgR_\rR' = (\tgR_\rR'\cap \kg')\oplus (\tgR_\rR'\cap\pg')$.
Set $\tgR_{\rR,0}' = \tgR_\rR'\cap \kg'$. Then $\tgR_{\rR,0}'$ is the Lie algebra of a
compact torus $\TT_0'$ and $s_1\in T_0'$ is a cocycle.

By the algorithm outlined in Subsection \ref{equiv:3},
we find an element $v\in \GG(\R)$ with $v^{-1}
\TT_0' v \subseteq \TT_0$.
 Then $s_2=v^{-1} s_1 v$ is a cocycle in $\TT=\cZ_\GG(\TT_0)$.
By running over $W_0$, we find $w\in W_0$ and $i$ such that $[s_2]\cdot w =
[c_i]$ in $\Hr^1 \TT$. We find $n\in N_G(T)$ such that $w= nT$. Finally we
use the algorithm of Subsection  \ref{tor:1} to find $t_2\in T$ such that $t_2^{-1}(n^{-1}
s_2 \upgam n)\upgam t  = c_i$. We set $h=u't v n t_2$.
\end{subsec}

\section{Second nonabelian cohomology}
\label{s:H2}

In this section we consider the case when $\Gamma$ is a group of order 2.
For  the case when $\Gamma$ is an arbitrary profinite group, see  \cite{Springer}, or \cite{Borovoi-Duke},
or \cite[Section 2]{FSS}, or \cite[Section 2.1]{Florence}, or \cite[Section 2.1]{LA}.

\begin{subsec}\label{ss:Z2}
Let $A$ be an {\em abstract} group and not necessarily a $\Gamma$-group.
We write
\begin{align*}
\Zr^2(\Gamma,A)=\big\{\hs (a,f)\, \mid\,  a\in A,\ f\in\Aut A,\ f^2=\inn(a),\ f(a)=a\big\}.
\end{align*}
Here $\inn(a)\colon A\to A$ is the inner automorphism $a'\mapsto aa'a^{-1}$.
For short, we write $\Zr^2\hm A$ for $\Zr^2(\Gamma, A)$.

We define a left action of $A$ on $\Zr^2\hm A$ as follows.
Let $s\in A$ and $(a,f)\in\Zr^2\hm A$.
We set
\[s*(a,f)= (a',f')\ \ \text{where}\ \, a'=s\cdot f(s)\cdot a,\ \, f'=\inn(s)\circ f.\]
We show that $(a',f')\in\Zr^1\hm A$. Indeed,
\begin{align*}
&f^{\prime\,2}=\inn(s)\circ f\circ \inn(s)\circ f=\inn(s)\circ\inn\big(f(s)\big)\circ f^2
              =\inn\big(s\cdot f(s)\cdot a\big)=\inn(a'),\\
&f'(a')=s\hs f(a')\hs s^{-1}=s\hm\cdot\hm f(s)\, f(f(s))\, f(a)\hm\cdot\hm s^{-1}
       =s\hs f(s)\hm\cdot\hm asa^{-1}\hm\cdot\hm as^{-1}=s\hs f(s)\, a=a',
\end{align*}
as required.

Let $s',s\in A$.
Write \,$s*(a,f)=(a',f')$, \,$s'*(a',f')=(a'',f'')$.
We have
\begin{align*}
&f''=\inn(s')\circ\inn(s)\circ f=\inn(s's)\circ f,\\
&a''=s'\cdot f'(s')\cdot a'=s'\cdot s\hs f(s')\hs s^{-1}\cdot s\hs f(s)\hs a
    =(s's)\cdot f(s's)\cdot a.
\end{align*}
We see that
\[(s's)*(a,f)=s'*(s*(a,f)).\]
Thus $*$ is indeed a left action of $A$ on $\Zr^2\hm A$.

Let  $(a,f),(a',f')\in \Zr^2\hm A$.
We write $(a,f)\sim (a',f')$ if there exists $s\in A$ such that
\[(a',f')=s*(a,f),\ \ \text{that is,}\ \ a'=s\cdot f(s)\cdot a,\ \, f'=\inn(s)\circ f.\]
\end{subsec}

\begin{definition} \label{d:H2}
$\Hr^2(\Gamma, A)=\Zr^2(\Gamma, A)\hs/\hm\sim$.
\end{definition}
Note that our definition of $\Hr^2(\Gamma, A)$ is not standard.
For short, we write $\Hr^2\hm A$ for $\Hr^2(\Gamma, A)$.
We write $[a,f]\in \Hr^2\hm A$ for the cohomology class
of a 2-cocycle $(a,f)\in \Zr^2\hm A$.

We say that a cocycle $(a,f)\in \Zr^2\hm A$ is {\em neutral} if $a=1$.
Then $f^2=\id_A$.
We say that a cohomology class $\eta\in \Hr^2\hm A$ is {\em neutral}
if it is the class of a neutral cocycle.
We denote by $\Nr^2\hm A\subseteq \Hr^2\hm A$
the subset of neutral cohomology classes, which might contain more than one element.

\begin{subsec}
Set $\Out A=\Aut A\hs/\Inn A$.
By a {\em band (lien)} we mean an element $\bkappa\in\Out A$ such that $\bkappa^2=1$.
We write
\begin{align*}
&\Zr^2(A,\bkappa)=\{(a,f)\in \Zr^2\hm A\mid f\cdot\Inn A=\bkappa\},\\
&\Hr^2(A,\bkappa)=\Zr^2(A,\bkappa)/\hm\sim.
\end{align*}

If $A$ is abelian, then $\Inn A=1$ and $\Out A=\Aut A$.
If, moreover, $\bkappa\in \Out A$ is a band,
then $\bkappa\in \Aut A$ and $\bkappa^2=\id_A$.
We regard the pair $(A,\bkappa)$ as a $\Gamma$-module,
where $\upgam\hm a=\bkappa(a)$ for $a\in A$.
Then
\[ \Hr^2(A,\bkappa)=\Hr^2\big(\Gamma,(A,\bkappa)\big),\]
the usual (abelian) second group cohomology of $\Gamma$
with coefficients in the $\Gamma$-module $(A,\bkappa)$.
\end{subsec}

\begin{subsec}\label{ss:Delta}
Let
\begin{align}\label{a:ABC}
1\to A\labelto{i} B\labelto{j} C\to 1
\end{align}
be a short exact sequence of $\Gamma$-groups.
We do not assume that $A$ is abelian.
We construct a map
\[\Delta\colon \Hr^1(\Gamma,C)\to \Hr^2\hm A\]
as follows.
We identify $A$ with $i(A)\subseteq B$.
Let $\xi\in \Hr^1(\Ga,C),\ \xi=[c],\  c\in \Zr^1(\Ga,C)$.
We write $c=j(b)$ for some $b\in B$, and set
\[ a=a_b\coloneqq b\hs\upgam b,\quad f=f_b\coloneqq\inn(b)\circ\gamma,
    \ \text{that is, $f(x)=b\cdot\!\upgam x\cdot b^{-1}$ for}\ x\in A.\]
Since $c\upgam\hm c=1$, we have $a\in A$.
Moreover, for $x\in A$ we have
\[f^2(x)=b\cdot\!\upgam(b\hs\upgam\hm x\hs b^{-1})\cdot b^{-1}
     =b\upgam b \cdot x\cdot \upgam b^{-1} b^{-1}
     =(b\upgam b)\cdot x \cdot(b\upgam b)^{-1}=a\hs x\hs a^{-1}.\]
Thus $f^2=\inn(a)$.
Furthermore,
\[f(a)=b\cdot\!\upgam(b\upgam b)\cdot b^{-1}= b\upgam b\hs b\hs b^{-1}=b\upgam b=a.\]
Thus $(a,f)\in \Zr^2\hm A$.
We set
\[\Delta_b(c)=(a,f)\in \Zr^2\hm A,\quad \Delta[c]=[\Delta_b(c)]=[a,f]\in\Hr^2\hm A.\]

If $b'\in B$ is another lift of $c$, then $b'=sb$ for some $s\in A$.
We obtain
\begin{align*}
&f'\coloneqq f_{b'}=\inn(b')\circ\gamma=\inn(s)\circ\inn(b)\circ\gamma=\inn(s)\circ f'\\
&a'\coloneqq a_{b'}=b'\cdot\!\upgam  b'=s\hs b\hs\upgam\! s\upgam b
            = s\cdot b\hs\upgam\! s\, b^{-1}\cdot b\hs\upgam b=s\cdot f(s)\cdot a.
\end{align*}
Thus $(a',f')\sim (a,s)$ and our map $\Delta$ is well defined.
\end{subsec}

\begin{proposition}\label{p:Delta}
The sequence
\[ \Hr^1(\Ga,B)\labelto{j_*}\Hr^1(\Ga,C)\labelto{\Delta}\Hr^2\hs A\]
is exact, that is, $\im j_*=\ker \Delta$, where
\[ \ker\Delta=\{\xi\in \Hr^1(\Gamma,C)\mid \Delta(\xi)\in \Nr^2\hm A\}.\]
\end{proposition}

\begin{proof}
Let $c\in \Zr^1(\Ga,C)$ and assume that  $[c]\in\im j_*$\hs.
Then $c=j(b)$ for some $b\in \Zr^1(\Ga, B)$.
We have $b\upgam b=1$, whence $a_b=b\upgam b=1$
and $\Delta[c]=[1,f_b]\in \Nr^2\hm A$, as required.

Conversely, let $c\in \Zr^1(\Ga,C)$ and assume that $\Delta[c]\in \Nr^2\hm A$.
Write  $c=j(b)$ where $b\in B$, then
\[\Delta[c]=[a,f]\ \ \text{where}\ \ a=b\upgam b,\ f=\inn(b)\circ\gamma.\]
The condition $\Delta[c]\in \Nr^2\hm A$ means that the exists $s\in A$ such that
\[s\cdot f(s)\cdot a=1,\ \ \text{that is,}\ \ s\cdot
      (b\hs \upgam\hm s\hs b^{-1})\cdot (b \upgam b)=1.\]
Thus can be written as
\[(s\hs b)\cdot\hm\upgam(s\hs b)=1.\]
Set $b'=sb$, then $j(b')=c$ and $b'\cdot\hm\upgam b'=1$.
We see that $[c]=j_*[b']$, where $b'\in \Zr^1(\Ga, B)$ and $[b']\in \Hr^1(\Ga,B)$.
Thus $[c]\in\im j_*$, as required.
\end{proof}

\begin{corollary}[of the proof of Proposition \ref{p:Delta}]
\label{c:lifting-1-cocycle}
For  the sequence \eqref{a:ABC}, let $c\in \Zr^1(\Ga,C)$, $c=j(b)$, $b\in B$.
Set
\[\Delta(b)=(a,f) \ \ \text{where}\ \ a=b\upgam b,\ f=\inn(b)\circ\gamma.\]
Assume that there exists $s\in A$ such that
\[s\cdot f(s)\cdot a=1.\]
Set $b'=sb\in B$. Then
\[c=j(b')\ \  \text{and}\ \  b'\in \Zr^1(\Ga,B).\]
\end{corollary}

\begin{subsec}
Now let $A$ be an  {\em algebraic group} over $\C$.
We write
\begin{align*}
\Zr^2(\R, A)=\big\{\hs (a,f)\, \mid\,  a\in A,\ f\in\SAut_\anti(A),\ f^2=\inn(a),\ f(a)=a\big\}
\end{align*}
where $\SAut_\anti (A)$ denotes the set of {\em anti-regular semi-automorphisms} of $A$; see Appendix \ref{app:A}.
This definition coincides with  that of Subsection \ref{ss:Z2}
when the algebraic group $A$ is finite; see Remark \ref{r:finite} in  Appendix \ref{app:A}.
As above, we define
\[  \Hr^2(\R, A)=\Zr^2(\R,A)\hs/\hm\sim\]
where the equivalence relation $\sim$  is defined in terms of the action of $A$ on $\Zr^2(\R, A)$.
Note that our definition of $\Hr^2(\R, A)$ is not standard.
For short, we write $\Zr^2\hm A$ for  $\Zr^2(\R,A)$, and we write  $\Hr^2\hm A$ for $\Hr^2(\R, A)$.
\end{subsec}

\begin{subsec}
Let $A$ be an algebraic $\C$-group.  We consider the {\em group of semi-automorphisms}
\[\SAut A = \Aut A\hs\cup\hs\SAut_\anti(A)\]
in which the multiplication  law is given by composition.
Here $\Aut A$ is the group of regular (usual) automorphisms of $A$.
It is easy to see that the subgroup $\Inn A\subset \SAut A$ is normal.
We define
\[\SOut A = \SAut A/\Inn A,\qquad \SOut_\anti(A)= \SAut_\anti(A)/\Inn A\subset \SOut A.\]
For a $\C$-group $A$, by a {\em band (lien)} we mean an element $\bkappa$
of the set $\SOut_\anti A$ such that $\bkappa^2=1$.
As above, we write
\begin{align*}
&\Zr^2(A,\bkappa)=\{(a,f)\in \Zr^2\hm A\mid f\cdot\Inn A=\bkappa\},\\
&\Hr^2(A,\bkappa)=\Zr^2(A,\bkappa)/\hm\sim.
\end{align*}

If $A$ is abelian, then
\[\Inn A=1,\quad  \SOut A=\SAut A,\quad \text{and}\quad\SOut_\anti(A)=\SAut_\anti(A).\]
If, moreover, $\bkappa\in \SOut_\anti(A)$ is a band,
then $\bkappa\in \Aut_\anti(A)$ and $\bkappa^2=\id_A$.
Thus $\bkappa$ is a real structure on $A$, and
we obtain an abelian  $\R$-group $\AA= (A,\bkappa)$.
Then
\[ \Hr^2(A,\bkappa)=\Hr^2(\R,\AA),\]
the usual (abelian) second Galois cohomology of $\AA=(A,\bkappa)$.
\end{subsec}

\section{Sansuc's lemma and $\Hr^1$ for connected non-reductive groups}
\label{s:non-reductive}

In this section we reduce  Problem \ref{mp:1}
for a connected $\R$-group, not necessarily reductive,
to the case of a connected {\em reductive} $\R$-group
treated in Section \ref{s:connected}.
We need the following proposition.

\begin{proposition}\label{p:Sansuc}
Let $\GG$ be any linear algebraic $\R$-group (not necessarily connected).
Let $\UU\subseteq \GG$ be a unipotent normal $\R$-subgroup.
Write $\GGbar=\GG/\UU$ and consider the canonical epimorphism
\[r\colon \GG\to \GGbar.\]
Then the induced map on cohomology
\begin{align*}
r_*\colon \Hr^1\hs \GG\to \Hr^1\hs \GGbar
\end{align*}
is bijective.
\end{proposition}

This is Sansuc's lemma \cite[Lemma 1.13]{Sansuc},
see also \cite[Proposition 3.1]{BDR}.
The proofs  in \cite{Sansuc} and \cite{BDR} over a field of characteristic 0
use induction on the dimension of $\UU$.
Here we prove Sansuc's lemma over $\R$ without induction, in one step for surjectivity
and in one step for injectivity.

\begin{proof}
Let $\xibar\in \Hr^1\hs\GGbar$, \,$\xibar=[\gbar]$,
where $\gbar\in \Zr^1\hs\GGbar\,\subseteq\Gbar$.
As in Subsection \ref{ss:Delta}, we lift $\gbar$ to some $g\in G$
and set $u=g\cdot\hm\upgam g\in U,\ f=\inn(g)\circ\gamma\colon\, U\to U$.
Then $f^2=\inn(u)$ and $f(u)=u$; see Subsection \ref{ss:Delta}.

We prove in one step that the map $r_*$ of the proposition is surjective.
By abuse of notation we write also $f$ for $df\colon\Lie U\to\Lie U$.
Since $U$ is unipotent, we have isomorphisms of $\C$-varieties
\[ \exp\colon\Lie U\to U\quad\text{and}\quad \log\colon U\to\Lie U,\]
and these isomorphisms are $f$-equivariant.
Set $s=\exp(-\frac12\log u)\in U$; then $s^2=u^{-1}$ and $f(s)=s$.
We have
\[s\cdot f(s)\cdot u=1.\]
We set $g'=sg$. Then $r(g')=\gbar$ and
$$g'\cdot\hm\upgam\hm g'=sg\cdot \upgam\hm s \upgam\hm g= s\cdot g\upgam\hm s\hs  g^{-1}\cdot g\upgam\hm g=s\cdot f(s)\cdot u=1.$$
We have lifted the 1-cocycle $\gbar\in \Zr^1\hs\GGbar$
to the 1-cocycle $g'\in \Zr^1\hs\GG$.
Thus the map $r_*$  is indeed surjective.

Now we prove in one step that the map $r_*$ of the proposition is injective.
Let $y\in \Zr^1\GG$ be a lift of $\gbar\in\Zr^1\hs\GGbar$,
and let $y'\in \Zr^1\hs\GG$ be another lift of $\gbar$.
We have $y'=u y$ for some $u\in U$.
Write $f=\inn(y)\circ\gamma\colon U\to U$.
We have
\[1=y'\hm\cdot\hm\upgam y'=u\hs y\upgam\hm u\upgam\hm y
       =u\cdot f(u)\cdot y\upgam\hm y=u\cdot f(u).\]
Hence $f(u)=u^{-1}$.

Set $s=\exp(-\frac12\log u)$; then $s^2=u^{-1}$ and $f(s)=s^{-1}$.
We obtain
\[ s^{-1}\hs y\hs \upgam\hm s
      =s^{-1}\cdot y\upgam\hm s\hs y^{-1} \cdot y
      =s^{-1}\cdot f(s) \cdot y =s^{-2}\cdot y=u\cdot y=y'. \]
We see that  $y\sim y'$.
We conclude that the lift $[y]\in \Hr^1\hs\GG$
of $[\gbar]\in \Hr^1\hs\GGbar$ is unique.
Thus the map $r_*$  is indeed injective.
\end{proof}

The proof yields two immediate algorithmic constructions which we
summarize in two corollaries.

\begin{corollary}[of the proof of Proposition \ref{p:Sansuc}]
\label{c:Sansuc-2}
Let $\gbar\in \Zr^1\hs\GGbar$, and let $g\in G$ be such that $r(g)=\gbar$.
Set
\[u=g\cdot\hm\upgam g\in U,\quad s=\exp(-\tfrac12\log u)\in U,\quad g'=sg.\]
Then $g'\in\Zr^1\hs\GG$ and $r(g')=\gbar$.
\end{corollary}

\begin{corollary}[of the proof of Proposition \ref{p:Sansuc}]
\label{c:Sansuc-3}
Let $g,g'\in \Zr^1\hs\GG$, and let $\gbar,\gbar'$
denote their images in $\Zr^1\hs\GGbar$.
Assume that $\gbar\sim \gbar'$, that is,
there exists $\sbar\in\Gbar$ such that
\[ \gbar'=\sbar^{-1}\cdot\gbar\cdot\hm\upgam\sbar.\]
Let $s\in G$ be a lift of $\sbar$.
Write
\[g''=s^{-1}\cdot g\cdot\hm\upgam s\in G\quad\text{and}\quad g'=u g'';\]
then $u\in U.$
Set
\[t=\exp(-\tfrac12\log u)\in U \quad\text{and}\quad s'=st.\]
Then
\[s^{\prime\,\hs-1}\cdot g\cdot\hm\upgam s'=g'.\]
\end{corollary}

\begin{subsec}
Let $\GG$ be a  connected linear algebraic group over $\R$,
not necessarily reductive.
Let $\GG^\uu$ denote the unipotent radical of $\GG$,
that is, the largest  normal unipotent subgroup.
Note that  $\GG^\uu$ is connected, because in characteristic 0 any unipotent group is connected
(which follows immediately from the existence of mutually inverse regular maps $\exp$ and $\log$).
Set $\GG^\red=\GG/\GG^\uu$, and let
\[ r\colon \GG\to\GG^\red\]
denote the canonical epimorphism.
Consider the induced map on the first Galois cohomology
\[r_*\colon \Hr^1\hs\GG\to\Hr^1\hs\GG^\red;\]
by Sansuc's lemma (Proposition \ref{p:Sansuc}) this map is bijective.
\end{subsec}

\begin{subsec}\label{ss:non-red-decomp}
Let $\GG$ be a connected $\R$-group, not necessarily reductive.
Write $\gR=\Lie\GG$, $\gR^\uu=\Lie\GG^\uu$, where $\GG^\uu$ is the unipotent
radical of $\GG$. As seen in Section \ref{app:B}, we have algorithms to
construct  a decomposition
\begin{equation*}
\gR=\gR^\uu\dotplus\gR^\rr
\end{equation*}
where $\gR^\rr\subseteq\gR$ is a the Lie algebra of a connected reductive $\R$-subgroup $\GG^\rr$,
and $\gR^\uu$ is the Lie algebra of a unipotent group.
Clearly, we have $\GG^\rr\cap\GG^\uu=\{1\}$,
and therefore the canonical homomorphism $\GG\to\GG^\red$
induces an isomorphism $\GG^\rr\isoto \GG^\red$.
\end{subsec}

\begin{lemma}\label{l:LL}
The inclusion map $\GG^\rr\into \GG$ induces a bijection $\Hr^1\GG^\rr\isoto \Hr^1\GG$.
\end{lemma}

\begin{proof}
We have a commutative diagram
\[
\xymatrix{
\Hr^1\GG^\rr\ar[r] \ar[d]_-\sim     &\Hr^1\GG\ar[d]^-\sim \\
\Hr^1\GG^\red\ar@{=}[r]       &\Hr^1\GG^\red
}
\]
in which the left-hand vertical arrow is bijective because
it is induced by an isomorphism of $\R$-groups,
and the right-hand vertical arrow is bijective by Sansuc's lemma.
Thus the top horizontal arrow is bijective, as required.
\end{proof}

\begin{subsec}
Let $\GG$ be a connected  $\R$-group, not necessarily reductive.
Using computer, we can construct a basis of a Lie subalgebra
$\gR^\rr\subset\gR$ as above.
Then using the methods of Section \ref{s:connected},
we can construct a set of representatives  $g_i$ of $\Hr^1\GG^\rr$;
they are contained in $\Zr^1\GG$.
By Lemma \ref{l:LL} this set is a set of representatives
of $\Hr^1\GG$, which solves Problem \ref{mp:1}
for a connected $\R$-group $\GG$.
\end{subsec}

\begin{subsec}\label{sub:pb02nonred}
We consider Problem \ref{mp:2} for our
(not necessarily reductive) connected $\R$-group $\GG$.

We consider the projection
\[d\pi^\rr\colon\gR\to\gR^\rr,\quad x=x^\uu+x^\rr\,\mapsto\, x^\rr\ \ \text{for}\ \ x^\uu\in\gR^\uu,\, x^\rr\in\gR^\rr.\]
In Subsection \ref{sec:B5} we have given an algorithm to explicitly construct the
homomorphism
\[\pi^\rr\colon \GG\to\GG^\rr.\]

Let $g\in\Zr^1\GG$, and let
$g^\rr=\pi^\rr(g)\in\Zr^1 \GG^\rr$ denote the image of $g$ in $G^\rr$.

Since $\GG^\rr$ is a connected reductive group,
we can use the method of Section \ref{s:connected-equivalence}
in order to find $i$ and $s_r\in G^\rr$
such that
\begin{equation}\label{e:sr-gr}
 s_r^{-1}\cdot g^\rr\cdot\sigma(s_r)=g_i\in \Zr^1 \GG^\rr\subseteq \Zr^1\GG.
\end{equation}

Set
\[ u=s_r^{-1}\cdot g\cdot\sigma(s_r)\cdot g_i^{-1}.\]
It follows from \eqref{e:sr-gr} that $\pi^\rr(u)=1$, whence  $u\in G^\uu$.
Moreover, since  $u g_i\in \Zr^1\GG$,
an easy calculation shows that
\[u\cdot g_i\hs\sigma(u)\hs g_i^{-1}=1,\]
that is, $u\in \Zr^1\GG_i$ where
\[\GG_i=\hs_{g_i}\hm\GG=(G,\sigma_i),\quad\ \sigma_i=\inn(g_i)\circ\sigma.\]
Thus $u\in \Zr^1(G^\uu,\sigma_i)$, that is,  $u\cdot\sigma_i(u)=1$ and $\sigma_i(u)=u^{-1}$.

Set  $s_i=\exp(\tfrac12\log u )$. Then
\[ s_i^2=u,\quad\ \sigma_i(s_i)=s_i^{-1},\quad\  s_i^{-1}\cdot u\cdot\sigma_i(s_i)=1.\]

We set $s=s_r s_i$. We calculate:
\[s^{-1} g \sigma(s)= s_i^{-1}\big(s_r^{-1} g\hs \sigma(s_r) \big)\sigma(s_i)=s_i^{-1} u g_i \hs\sigma(s_i).\]
We have
\[g_i\hs \sigma(s_i)g_i^{-1}=s_i^{-1},\quad\text{whence}\quad g_i\hs\sigma(s_i)=s_i^{-1} g_i\hs.\]
We obtain
\[ s_i^{-1} u\hs g_i \hs\sigma(s_i)=s_i^{-1} u \hs s_i^{-1} g_i=g_i\hs.\]
Thus
\[s^{-1} g\hs \sigma(s)=g_i\hs.\]
We have found a cocycle $g_i$ from our list and an element $s\in G$ such that $s^{-1} g\hs \sigma(s)=g_i$\hs,
which solves Problem \ref{mp:2} for our connected $\R$-group $\GG$ that is not necessarily reductive.
\end{subsec}

\section{Computing abelian $\Hr^2$ for a quasi-torus}
\label{s:H2-quasi-torus}

Let $\CC$ be a quasi-torus over $\R$. For computational purposes we assume that
we know a basis of the Lie algebra of $\CC^0$, the finite group $\CC^\ff =
\CC/\CC^0$, and for each element of $\CC^\ff$ we know a preimage in $\CC$.
We wish to compute the second (abelian) Galois cohomology group $\Hr^2\hs\CC\coloneqq \Hr^2(\R,
\CC)$ in the following sense:

\begin{problem}\label{p:1}
We need a method (algorithm, computer program)
giving us a list
\[c_1,\dots, c_n\in\Zr^2\hs\CC\]
of 2-cocycles representing all cohomology classes
in $\Hr^2\hs\CC$ so that for each cohomology class we have exactly one cocycle.
\end{problem}

\begin{problem}\label{p:2}
Assume that we have solved Problem \ref{p:1} for $\CC$.
We need a computer program permitting us, for any 2-cocycle
$c\in \Zr^2\hs\CC$, to determine  $i$
such that $c\sim c_i$, and to find an element $s\in C$ such that
\[ c=s\cdot\hm\upgam\hm s\cdot c_i\hs.\]
\end{problem}

We also consider the following problem:

\begin{problem}\label{p:3}
We need a computer program permitting us, for any 2-cocycle
$c\in \Zr^2\hs\CC$, to determine whether it is a 2-coboundary, that is,
whether $c\in \Br^2\hs\CC$, and if yes, to find $s\in C$ such that
\begin{equation}\label{e:c-s-g-s}
c= s\cdot\hm\upgam\hm s.
\end{equation}
\end{problem}

Note that Problem \ref{p:2} can be reduced to Problem \ref{p:3}.
Indeed, for all $i$ with $1\le i\le n$, let us compute $cc_i^{-1}\in \Zr^2\hs\CC$
and, using our solution of Problem \ref{p:3},
determine $i$ for which $cc_i^{-1}$ is a 2-coboundary
and find $s\in C$ such that $cc_i^{-1}= s\cdot\hm\upgam\hm s$.
Then $c\sim c_i$ and $c=s\cdot\hm\upgam\hm s\cdot c_i$.

\begin{subsec}
If $\CC$ is a {\em finite} abelian $\R$-group,
then Problems \ref{p:1} and \ref{p:3} can be solved by brute force,
using the definitions in Subsection \ref{ss:H1-abelian}; see Remark \ref{r:brute-force}.
\end{subsec}

\begin{subsec}
Let $\CC$ be an $\R$-quasi-torus.
We embed $\CC$ into an $\R$-torus $\TT$ and set $\TT'=\TT/\CC$.
We have an isomorphism \eqref{e:ATT'}
\[ \H^1\big(\X_*(\TT)\to\X_*(\TT')\big)\labelto{\sim}
    \Hr^2\hs\CC,\ \  [\nu,\nu']\mapsto [c]
    \ \, \text{for}\  (\nu,\nu')\in \Zr^1\big(\X_*(\TT)\to\X_*(\TT')\big), \]
where $t\in T$ is a preimage of $\nu'(-1)$ and
\[c=\nu(-1)\cdot t\cdot\!\upgam t;\]
see  \eqref{e:nu-nu'-a}.
We can compute $\H^1\big(\X_*(\TT)\to\X_*(\TT')\big)$ using computer.
This solves Problem \ref{p:1}.
\end{subsec}

\begin{subsec}\label{ss:C-is-torus}
We consider Problem \ref{p:3} for the case when $\CC$ is a {\em torus}.
In \ref{tor:3} it is shown how to write $\CC$ as product of standard tori
$\FF_i$, $\EE_i$ and $\DD_i$. Accordingly we can write the element $c\in \Zr^2\hs\CC$
as a product of elements $\phi_i(s_i)$, $\epsilon_i(t_i)$, $\delta_i(u_i,v_i)$,
with $s_i,t_i,u_i,v_i\in \C^\times$. Then $\phi_i(s_i)\in \Zr^2\hs \FF_i$,
$\epsilon(t_i)\in \Zr^2\hs \EE_i$, $\delta_i(u_i,v_i)\in \Zr^2\hs \DD_i$.
We solve the problem for each subtorus in the following way:
\begin{itemize}
\item We have $\Zr^2\FF_i = \Br^2 \FF_i = \{ \phi_i(z) | |z| = 1 \}$.
  We find $\hat s_i \in \C^\times$ with $\hat s_i^2 = s_i$ and $|s_i|=1$. Then
  $\phi_i(\hat s_i) \upgam \phi_i(\hat s_i ) = \phi_i(s_i)$.
\item We have $\Zr^2 \EE_i = \{ \epsilon_i(z) \mid z \in \R^\times \}$ and
  $\Br^2 \EE_i = \{ \epsilon_i(z) \mid z\in \R^\times, z > 0\}$.
  So $\epsilon(t_i)\in \Br^2 E_i$ if and only if $t_i\in \R$, $t_i>0$.
  In that case we set $\hat t_i = \sqrt{t_i}$ and $\epsilon_i(t_i) =
  \epsilon_i(\hat t_i)\upgam\epsilon_i(\hat t_i)$. If $t_i <0$ then
  $\epsilon_i(t_i)\not\in \Br^2 E_i$ and hence $c\not\in \Br^2 \CC$.
\item We have $\Zr^2 \DD_i = \Br^2 \DD_i = \{ \delta_i(z,\bar{z})\, |\,  z \in \C^\times\}$.
  So $v_i = \bar{u_i}$. Then $\delta_i(u_i,\bar{u_i}) = \delta(u_i,1)\, \upgam
  \delta_i(u_i,1)$.
\end{itemize}

We form the product of the constructed elements and find $s\in C$ such that
$c = s \upgam s$.

\end{subsec}

\begin{subsec}
We consider Problem \ref{p:3} in the general case,
 when $\CC$ is an arbitrary $\R$-quasi-torus.
We have a short exact sequence
\[1\to \CC^0\to\CC\to\CC^\ff\to 1,\]
where the identity component  $\CC^0$ is an $\R$-torus,
and the component group $\CC^\ff\coloneqq \CC/\CC^0$ is a finite abelian $\R$-group.

Let $c\in\Zr^2\hs\CC$. We choose representatives $r_1,\dots, r_m\in C$ of $C^\ff=C/C^0$.
For each $i=1,\dots,m$ we compute
\[z^{(i)}=c\cdot (r_i\upgam r_i)^{-1}\]
and determine whether $z^{(i)}\in C^0$.
If yes, then $z^{(i)}\in \Zr^2\CC^0$, and using the method of Subsection \ref{ss:C-is-torus},
we determine whether there exists an element $s_0^{(i)}\in C^0$ such that
\begin{equation}\label{e:z^i}
z^{(i)}= s_0^{(i)}\cdot\upgam s_0^{(i)}.
\end{equation}
If yes, then we set
\[s=r_i\cdot s_0^{(i)}\in C,\]
and we obtain that $c=s\hs\upgam\hm  s$, that is, $s$ satisfies  \eqref{e:c-s-g-s}.
If for all $i=1,\dots,m$ we have either $z^{(i)}\notin C^0$, or $z^{(i)}\in C^0$
but there is no $s_0^{(i)}\in C^0$ satisfying \eqref{e:z^i},
then we conclude that there is no $s\in C$ satisfying \eqref{e:c-s-g-s}.
This solves Problem \ref{p:3}.
\end{subsec}

\section{Neutralizing a 2-cocycle of a connected reductive group}
\label{s:neutralizing}

\begin{subsec}
Let $H$ be a connected reductive $\C$-group.
Let $(h,f)\in\Zr^2 H$ be a 2-cocycle, that is,
\[h\in H,\ \ f\in\SAut_\anti(H),\quad f^2=\inn(h),\ \ f(h)=h;\]
in particular, we assume that $f$ is anti-regular (= anti-holomorphic).
We wish to find a {\em neutralizing element for $(h,f)$},
that is, an element $d\in H$ such that
\[d\cdot f(d)\cdot h=1,\]
or to prove that there is no neutralizing element.
\end{subsec}

\begin{definition}
A {\em pinning of $H$} is a tuple $\big(T,B,\{X_\alpha\}_{\alpha\in S}\big)$,
where $T\subseteq H$ is a maximal torus in $H$,
$B\subseteq H$ is a Borel subgroup containing $T$,
$S=S(H,T,B)$ is the system of simple roots corresponding to $(H,T,B)$,
and $X_\alpha\in {\mathfrak h}_\alpha$ is a nonzero root
vector in the one-dimensional root subspace
${\mathfrak h}_\alpha\subset {\mathfrak h}$
corresponding to a simple root $\alpha$;
see \cite[Definition 1.5.4]{Conrad}.
\end{definition}

\begin{subsec}\label{neut:3}
Consider the pinning $f(T,B,\{X_\alpha\}_{\alpha\in S}\big)$ of $H$.
It is well known that there exists an element $a\in H$ such that
\[a\cdot f(T,B,\{X_\alpha\}_{\alpha\in S}\big)\cdot a^{-1}= (T,B,\{X_\alpha\}_{\alpha\in S}\big).\]
Set
\[(h',f')=a*(h,f),\quad\ \text{that is,}\quad\ h'=a\cdot f(a)\cdot h,\ \  f'=\inn(a)\circ f.\]
Then we know that $(h',f')$ is a nonabelian 2-cocycle:
\[ f^{\prime\,2}=\inn(h'),\quad f'(h')=h'.\]
We have
\[f'(T,B,\{X_\alpha\}_{\alpha\in S}\big)=(T,B,\{X_\alpha\}_{\alpha\in S}\big),\]
whence
\[ \inn(h')(T,B,\{X_\alpha\}_{\alpha\in S})=f'\big(f'(T,B,\{X_\alpha\}_{\alpha\in S})\big)=(T,B,\{X_\alpha\}_{\alpha\in S}).\]
Since the only inner automorphism preserving the pinning
$(T,B,\{X_\alpha\}_{\alpha\in S}\big)$ is the identity automorphism,
we see that $\inn(h')=\id_H$, $f^{\prime\,2}=\id_H$, $h'\in C:=Z(H)$.

The anti-regular involution $f'$ of $H$ defines
real forms $\HH=(H,f')$ of $H$ and $\CC=(C,f'|_C)$ of $C$,
and we have $h'\in \CC(\R)=\Zr^2\CC$.
We consider the (abelian) cohomology class $[h']\in\Hr^2\CC$.

Let $\HH^\ssc$ denote the universal cover
of the commutator subgroup $[\HH,\HH]$ of $\HH$ (which is simply connected);
see  \cite[Proposition (2:24)(ii)]{BT72}
or  \cite[Corollary A.4.11]{CGP}.
We write $\CC^\ssc=Z(\HH^\ssc)$.
We have a canonical homomorphisms
\begin{align*}
&\rho\colon\HH^\ssc\onto[\HH,\HH]\into\HH,\\
&\rho_\cCC\colon\CC^\ssc\to\CC,
\end{align*}
and the induced homomorphism
\[\rho_{\cCC*}\hs\colon\,\Hr^2\CC^\ssc\to\Hr^2\CC.\]
\end{subsec}

\begin{theorem}
The second cohomology class $[h,f]=[h',f']\in \Hr^2\hs H$ is neutral if and only if
\begin{align}\label{a:rhoC}
[h']\hs\in\hs\im [\rho_{\cCC*}\colon\Hr^2\CC^\ssc\to\Hr^2\CC].
\end{align}
\end{theorem}

\begin{proof}
See \cite[Theorem 5.5]{Borovoi-Duke}.
\end{proof}

\begin{subsec}
We can compute $\Hr^2\CC$ and $\Hr^2\CC^\ssc$
and to check whether \eqref{a:rhoC} is satisfied
using methods of Section \ref{s:H2-quasi-torus}.
If \eqref{a:rhoC} is satisfied for $h'$,
then it is also satisfied for $(h')^{-1}$,
and using methods of Section \ref{s:H2-quasi-torus}
we find  $c\in C$ and  $z^\ssc\in\Zr^2\CC^\ssc$ such that
\[c\cdot\hm\upgam c\cdot\rho(z^\ssc)\cdot h'=1.\]
We explain how to construct neutralizing elements
for the nonabelian 2-cocycles $(h,f)$ and $(h',f')$
using these $c$ and $z^\ssc$.

Let $\TT^\ssc\subset\HH^\ssc$ be a fundamental torus, that is,
a maximal torus containing a maximal compact torus.
By \cite[Proof of Lemma 10.4]{Kottwitz86},
the group of $\R$-points $\TT^\ssc(\R)$
is isomorphic to the direct product of groups isomorphic to $U(1)$ and $\C^\times$.
It follows that any element of $\TT^\ssc(\R)$
is the square of another element of this group.
In particular, we can write $z^\ssc=(t^\ssc)^2$, where $t^\ssc\in\TT^\ssc(\R)$.

Set $b=c\cdot\rho(t^\ssc)\in H$.
We calculate:
\[b\cdot f'(b)\cdot h'=b\cdot\hm\upgam b\cdot h'
    =c\cdot\hm\upgam c\cdot\rho(t^\ssc)^2\cdot h'
    =c\cdot\hm\upgam c\cdot\rho(z^\ssc)\cdot h'=1.\]
Thus $b*(h',f')\in \Nr^2 H$.
Since $(h',f')=a*(h,f)$, we conclude that $(ba)*(h,f)\in \Nr^2 H$.
Thus $b$ is a neutralizing element for the 2-cocycle $(h',f')$,
and $ba$ is a neutralizing element for the 2-cocycle $(h,f)$.
\end{subsec}

\begin{subsec}\label{neut:7}
Here we comment on the implementation of the steps outlined above. The input to
our algorithm is a basis of the complex
Lie algebra $\hg = \Lie H$ and the differential of the automorphism $f$,
which by abuse of notation we also denote by $f$.
This is an anti-regular  semi-automorphism of $\hg$.

First we comment on how to find the element $a$ in \ref{neut:3}. Let $\tl$ be
a Cartan subalgebra of $\hg$.
Then $\tl=\Lie T$, where $T$ is a maximal torus in $H$.
Consider the root decomposition
\[ \hg=\tl\oplus\bigoplus_{\alpha\in \Phi}\hg_\alpha\hs,\]
and let $S\subset \Phi$ be a system of simple roots.
We choose a {\em pinning} of $(\g,\hg)$:
a family of nonzero root vectors $X_\alpha\in \hg_\alpha$ for $\alpha\in S$.
We say that the root vectors $X_\alpha$ are {\em simple root vectors}.
Then $\tl'=f(\tl)$ is a Cartan subalgebra of $\hg$ and the
$X_\alpha'=f(X_\alpha)$ are simple root vectors with respect to $\tl'$. We
construct a (regular) automorphism $g$ of $\hg$ such that $g(\tl') = \tl$ and
$g(X_\alpha') = X_\alpha$. We check using the algorithm of \cite{cmt} whether
$g$ is an inner automorphism. If it is not, we compose $g$ with an outer
automorphism that fixes $\tl$ and permutes
the simple root vectors $X_\alpha$ and such that the
result is inner. We denote this composite automorphism again by $g$.
Then using the algorithm of \cite{cmt}, we write $g$ as a product of
$\exp( \ad X )$ where $X\in \hg$ is nilpotent. We let $a$ be the same product,
where instead of $\exp( \ad X )$ we write $\exp(X)$. Then $g(Y) = aYa^{-1}$ for
$Y\in \hg$. Hence $af(\tl)a^{-1} = \tl$ and $\{ af(X_\alpha)a^{-1} \} =
\{ X_\alpha \}$. It follows that the automorphism $Y\mapsto af(Y)a^{-1}$ maps
the Borel subalgebra spanned by $\tl$ along with all positive root vectors to itself.
Hence the element $a$ has the required properties.

Now we construct the derived subgroup $H'=[H,H]$ of $H$.
The universal cover $H^\ssc$ of  $H'$
 has Lie algebra $\Lie H^\ssc=\Lie H'=\hg'\coloneqq [\hg,\hg]$.
It is known how to find the highest weights
of a representation $r : \hg' \to \gl(V)$
such that the corresponding representation of $H^\ssc$ with differential $r$ is faithful
(see, for instance, \cite[Chapter 3, Theorem 2.18]{GOV}).
The representation $r$
can then be explicitly constructed using the algorithm of \cite{wdgrep}.
Namely, we use a fixed Cartan subalgebra and corresponding root system of
$r(\hg')$. For a root $\alpha$ in this root system we fix a root vector
$X_\alpha\in r(\hg')$. We construct these root vectors  $X_\alpha$ in such a way that
they are part of a Chevalley basis of the semisimple Lie algebra $r(\h')$; see \cite[\S 25]{hum}.
We identify $H^\ssc$ with the subgroup in $\GL(V)$
generated by the elements $\exp( t X_\alpha )$.
Then the surjective map
$\rho : H^\ssc \to [H,H]$ is induced by mapping
$\exp( t X_\alpha ) \mapsto \exp( t\cdot r^{-1} (X_\alpha) )$.

We explain how to construct the elements of the center of $H^\ssc$.
We fix a set of simple roots $\alpha_1,\ldots,\alpha_\ell$ of $r(\hg')$.
For a simple root $\alpha$
and $t\in \C$ we set $x_\alpha(t) = \exp( t X_\alpha )$. Furthermore, for $t\in
\C^\times$ we use the following standard notation
$$w_\alpha(t) = x_\alpha(t)x_{-\alpha}(-t^{-1})x_\alpha(t) \text{ and }
h_\alpha(t) = w_\alpha(t) w_\alpha(1)^{-1}.$$
Then the center of $H^\ssc$ consists of all products
$$\prod_{j=1}^\ell h_{\alpha_j}(t_j) \text{ such that } \prod_{j=1}^\ell t_j^{\langle
\alpha_i, \alpha_j^\vee\rangle} =1 \text{ for }1\leq i\leq \ell,$$
(see \cite[Chapter 3, Lemma 28]{steinberg}).
We can thus find the elements of the center of $H^\ssc$ by computing
the Hermite normal form $A$ of the Cartan matrix $(\langle \alpha_i,\alpha_j^\vee
\rangle)_{i,j}$. Then the equations $\prod_{j=1}^\ell t_j^{\langle
\alpha_i, \alpha_j^\vee\rangle} =1$ are equivalent to the equations
$\prod_{j=1}^\ell t_j^{A(i,j)} =1$. It is straightforward to solve the latter as
$A$ is upper triangular.

We can then compute by brute force $\Hr^2 (\R, \CC^\ssc)$  as well as the
image of $\rho_{\cCC,*} : \Hr^2 (\R, \CC^\ssc)\to \Hr^2 (\R, \CC)$. Then with
the methods of Section \ref{s:H2-quasi-torus} we either find $c\in C$ and
$z^\ssc\in
\Zr^2(\R, \CC^\ssc)$ such that $c\cdot \upgam c \cdot \rho(z^\ssc) \cdot h'=1$,
or conclude that no such elements exist. In the latter case there is no
neutralizing element for $(h,f)$. In the former case we continue.

We can compute the Lie algebra $\tgR^\ssc$ of a fundamental torus  $\TT^\ssc$
of $\HH^\ssc$. Then with the methods of Subsection \ref{tor:1} we can compute
an isomorphism $\lambda : (\C^\times)^m \to T^\ssc$ and we can find
$u_i\in \C^\times$ such that $z^\ssc = \lambda(u_1,\ldots,u_m)$. Then we
set $t^\ssc = \lambda(\sqrt{u_1},\ldots,\sqrt{u_m})$ and find $t^\ssc \in
\TT^\ssc(\R)$ such that $(t^\ssc)^2 = z^\ssc$.
Then $d = c\cdot \rho(t^\ssc)\cdot a$ is a neutralizing element for $(h,f)$.
\end{subsec}

\section{Neutralizing a 2-cocycle of a connected non-reductive group}
\label{s:neutralizing-nonred}

Let $H$ be a connected $\C$-group, not necessarily reductive.
Let $H^\uu$ denote the unipotent radical of $H$.

\begin{subsec}
Let $(h,f)\in \Zr^2 H$ be a 2-cocycle.
This means that $h\in H$, $f\in \SAut_\anti (H)$, $f(h)=h$, and $f^2=\inn (h)$.
We wish to neutralize $(h,f)$, that is, to find $d\in H$ such that $d*(h,f)=(1,f')$
for some $f'\in\SAut_\anti(H)$.
Recall that
\[d*(h,f)=(d\cdot f(d)\cdot h,\, \inn(d)\circ f).\]
Thus we wish to find $d\in H$ such that $d\cdot f(d)\cdot h=1$.
\end{subsec}

\begin{subsec}\label{sub:2coc}
Write $\h=\Lie H$, $\h^\uu=\Lie H^\uu$.
As in Section \ref{app:B}, we construct a decomposition
\begin{equation}\label{e:h-dotplus}
\h=\h^\uu\dotplus\h^\rr
\end{equation}
where $\h^\rr=\Lie H^\rr$ for some connected reductive subgroup $H^\rr\subseteq H$ such that $H=H^\uu\rtimes H^\rr$.

Consider the differential $df\colon\h\to \h$ and the Lie subalgebra $df(\h^\rr)\subseteq \h$.
Then $df(\h^\uu)=\h^\uu$ and we have a decomposition
\[\h=\h^\uu\dotplus df(\h^\rr).\]
As in Subsection \ref{sec:B3} we can find an element $s\in H^\uu$ such that
\[\Ad(s)\big(df(\h^\rr)\big)=\h^\rr,\]
in the language of matrices:
\[s\cdot df(\h^\rr)\cdot s^{-1} =\h^\rr.\]
We set $(h_1,f_1)=s*(h,f)\in \Zr^2\hs H$, that is,
\[  h_1=s\cdot f(s)\cdot h,\qquad f_1=\inn(s)\circ f.\]
Then $df_1(\h^\rr)=\h^\rr$.

The decomposition \eqref{e:h-dotplus} defines a projection map $\h\to \h^\rr$
and a homomorphism
\[\pi^\rr\colon H\to H^\rr,\]
which we can compute explicitly using the algorithm given in Subsection
\ref{sec:B5}. We set
\[h_1^\rr=\pi^\rr(h_1)\in H^\rr,\qquad f_1^\rr=f_1|_{H^\rr}\hs.\]
Then $(h_1^\rr,f_1^\rr)\in \Zr^2\hs H^\rr$.
Indeed, we have
\[\pi^\rr\circ f_1=f_1^\rr\circ\pi^\rr,\]
whence
\begin{gather*}
f_1^\rr(h_1^\rr)=f_1^\rr(\pi^\rr(h_1))=\pi^\rr(f_1(h_1))=\pi^\rr(h_1)=h_1^\rr,\\
(f_1^\rr)^2=\big(f_1|_{H^\rr}\big)^2=f_1^2|_{H^\rr}=\inn(h_1)|_{H^\rr}=\inn(\pi^\rr(h_1))=\inn(h_1^\rr).
\end{gather*}
\end{subsec}

\begin{subsec}
As in Section \ref{s:neutralizing}, we determine
whether the class $[h_1^\rr,f_1^\rr]$  of the 2-cocycle $(h_1^\rr,f_1^\rr)\in \Zr^2\hs H^\rr$
of the connected reductive group $H^\rr$ is neutral.
Indeed, suppose for the sake of contradiction that the cohomology class $[h,f]$ is neutral.
Thus $[h_1,f_1]=s*(h,f)$ is neutral, and therefore there exists $d_1\in H$ such that $d_1*(h_1,f_1)$ is a neutral cocycle,
that is, \[d_1\cdot f_1(d_1)\cdot h_1=1.\]
Set $d_1^\rr=\pi^\rr(d_1)$. Then
 \[d_1^\rr\cdot f_1^\rr(d_1^\rr)\cdot h_1^\rr=\pi^\rr\big(d_1\cdot f_1(d_1)\cdot h_1\big)=\pi^\rr(1)=1,\]
that is, $d_1^\rr*(h_1^\rr,f_1^\rr)$ is a neutral cocycle,
whence the class $[h_1^\rr,f_1^\rr]$ is neutral, which is a contradiction.

Now assume that the class $\big[h_1^\rr,f_1^\rr\big]$ is neutral, that is, there is an element $d_1^\rr\in H^\rr\subseteq H$
such that $d_1^\rr *\big(h_1^\rr,f_1^\rr\big)\in\Zr^2\hs H^\rr$ is a neutral cocycle.
We describe an algorithm of  neutralizing $(h,f)$.

Set
\[\big(h_2^\rr, f_2^\rr\big)=d_1^\rr *\big(h_1^\rr,f_1^\rr\big)\in \Zr^2\hs H^\rr;\]
then $1=h_2^\rr=d_1^\rr\cdot f_1^\rr\big(d_1^\rr\big)\cdot h_1^\rr$.
Set
\[(h_2, f_2)=d_1^\rr *(h_1,f_1)\in \Zr^2\hs H;\]
then $\pi^\rr(h_2)=h_2^\rr=1$, whence $h_2\in H^\uu$.
We set
\[h_2^\uu=h_2\hs,\qquad f_2^\uu=f_2|_{H^\uu}\hs;\]
then $(h_2^\uu,f_2^\uu)\in\Zr^2\hs H^\uu$.
\end{subsec}

\begin{subsec}
Since $H^\uu$ is a unipotent group,
by Douai's theorem \cite[Theorem IV.1.3]{Douai},
see also \cite[Corollary 4.2]{Borovoi-Duke},
any cocycle in $\Zr^2\hs H^\uu$ can be neutralized.
The proofs in \cite{Douai} and \cite{Borovoi-Duke}
over a field of characteristic 0 proceed by induction on $\dim H^\uu$.
Here we neutralize the cocycle  $(h_2^\uu,f_2^\uu)$ over $\R$ in one step.

Since  $H^\uu$ is unipotent, there exist mutually inverse regular maps
\[\exp\colon \h^\uu\to  H^\uu,\qquad \log\colon H^\uu\to\h^\uu.\]
Set
\[d_2^\uu=\exp(-\tfrac12 \log h_2^\uu)\in H^\uu\subset H.\]
Then
\[f_2^\uu(d_2^\uu)=d^\uu,\qquad (d_2^\uu)^2=(h_2^\uu)^{-1}.\]
It follows that
\[ d_2^\uu\cdot f_2^\uu(d_2^\uu)\cdot h_2^\uu=1,\]
that is, $d_2^\uu$ neutralizes the 2-cocycle $(h_2^\uu,f_2^\uu)\in \Zr^2\hs H^\uu,$
and hence the 2-cocycle  $d_2^\uu *(h_2, f_2)\in \Zr^2\hs H$ is neutral
(recall that $h_2^\uu=h_2$ and $f_2^\uu=f_2|_{H^\uu}$).

Since
\[(h_2,f_2)=d_1^\rr * (h_1,f_1)\quad\ \text{and}\quad\ (h_1,f_1)=s*(h,f),\]
we see that the 2-cocycle
\[ (d_2^\uu\cdot d_1^\rr\cdot s)*(h,f)\in \Zr^2\hs H\]
is neutral, and the element  $d\coloneqq d_2^\uu\, d_1^\rr\hm s$
neutralizes our original 2-cocycle $(h,f)$.
\end{subsec}

\section{$\Hr^1$ for non-connected $\R$-groups}
\label{s:non-connected}

In this section we reduce Problems \ref{mp:1}
and \ref{mp:2} for a not necessary connected  $\R$-group
to the case of a connected  $\R$-group
treated in Sections \ref{s:connected}, \ref{s:connected-equivalence},  and \ref{s:non-reductive}.

\begin{subsec}
Let $\GG$ be a linear  $\R$-group, not necessarily connected or reductive.
We wish to compute $\Hr^1\hs \GG:=\Hr^1(\R,\GG)$.
We write $\GG=(G, \sigma)$, where $G=\GG(\C)$.
Then $\Hr^1\hs\GG=\Hr^1\big(\Gamma,(G,\sigma)\big)$.

We write $\GG^0$ for the identity component of $G$, and set $\GG^\ff=\GG/\GG^0$.
Then $\GG^\ff$ is a finite $\Ga$-group.
We have a short exact sequence
\[1\to\GG^0\labelto{\iota}\GG\labelto{\pi}\GG^\ff\to1.\]

We wish to compute $\Hr^1\GG$, that is, to write a set
$\Xg\subset \Zr^1\hs\GG$ of representatives
of all cohomology classes in $\Hr^1\hs\GG$.

We start with computing $\Hr^1\hs\GG^\ff$. Since the group $\GG^\ff$ us finite,
we can find a list of representatives
$g_1^\ff=1,g_2^\ff,\,\dots,\,g_k^\ff\in \Zr^1\hs\GG^\ff$
of all cohomology classes in $\Hr^1\hs\GG^\ff$ by brute force;
see Remark \ref{r:brute-force}.
Now in order to compute $\Hr^1\GG$ it suffices to compute $\pi_*^{-1}\xi^\ff$
for each cohomology class $\xi^\ff=[g^\ff]\in \Hr^1\GG^\ff$.
\end{subsec}

\begin{subsec}\label{ss:H1-ker}
We consider the cohomology class $[1]\in \Hr^1\hs\GG^\ff$
and compute its preimage  $\pi_*^{-1}[1]=\ker\pi_*\subseteq\Hr^1\hs\GG$.
By Proposition \ref{p:Delta} the short exact sequence
\[1\to \GG^0\labelto{\iota}\GG\labelto\pi\GG^\ff\to 1\]
induces a cohomology exact sequence
\[\GG^\ff(\R)\labelto\delta \Hr^1\hs\GG^0\labelto{\iota_*}\Hr^1\hs\GG\labelto{\pi_*}
     \Hr^1\hs\GG^\ff\labelto\Delta\Hr^2\hs\GG^0.\]
Thus  $\ker\pi_*\subseteq\Hr^1\hs\GG$ coincides with $\im\iota_*$.
Moreover, the fibers of the map
\[\iota_*\colon \Hr^1\hs\GG^0\to\Hr^1\hs\GG \]
are the orbits of $\GG^\ff(\R)$ in $\Hr^1\hs\GG^0$;
see Serre \cite[Section I.5.5, Proposition 39(ii)]{Serre}.

Consider the anti-regular involution
\[\sigma\colon G\to G,\quad g\mapsto\upgam g.\]
We consider also the induced anti-regular involutions
\[\sigma^0\colon G^0\to G^0\quad\text{and}\quad \sigma^\ff\colon G^\ff\to G^\ff.\]
For simplicity, we write $\sigma$ for $\sigma^0$ and $\sigma^\ff$.

We describe the right action of $\GG^\ff(\R)= (G^\ff)^\sigma$ on $\Hr^1\hs\GG^0$.
Let $b^\ff\in(G^\ff)^\sigma$. We lift $b^\ff$ to some element $b\in G$.
Let $g\in\Zr^1\hs\GG$. Then
\begin{equation}\label{e:G-sig-action}
 [g]*b^\ff=\big[b^{-1} g\,\sigma(b)\big].
\end{equation}

Using computer as explained in Sections \ref{s:connected} and \ref{s:non-reductive},
we solve Problem \ref{mp:1} for $G^0$,
that is, find a  set of 1-cocycles
\[X_1\subset\Zr^1\hs\GG^0\]
representing all cohomology classes.
The group $G^\ff$ is finite, and we can
find the subgroup $(G^\ff)^\sigma$ by brute force.
For any $x\in X_1$ and any  $b^\ff\in(G^\ff)^\sigma$ we construct a 1-cocycle
\[b^{-1} x\,\sigma(b)\in \Zr^1\hs\GG^0\]
as in \eqref{e:G-sig-action}, representing $[x]*b^\ff$.
Using computer as explained in Sections \ref{s:connected} and \ref{s:non-reductive},
we find elements $x'\in X_1$ and $b_0\in G^0$ such that
\[b_0^{-1}b^{-1} x\,\sigma(b)\sigma(b_0)= x'.\]
Thus we obtain an action of $(G^\ff)^\sigma$ on  $X_1$.
We choose a subset $X^{\num}_1\subseteq X_1$
of representatives of orbits for this action.
Then this subset $X^{\num}_1\subseteq X_1\subset\Zr^1\hs\GG^0$
is a set of representatives of all cohomology classes in\, $\ker\pi_*=\pi_*^{-1}[1]$.
\end{subsec}

\begin{subsec}\label{ss:b-x0}
Now let $g\in \Zr^1\hs\GG^0\subseteq\Zr^1\hs\GG$.
Then $[g]\in\ker\big[\Hr^1\hs\GG\to\Hr^1\hs\GG^\ff\big]$.
We wish to find $b\in G$ and $x^{\num}\in X_1^{\num}$ such that
\begin{equation*}
b^{-1}\cdot g\cdot \sigma(b)=x^{\num},
\end{equation*}
Since the set $X_1\subset \Zr^1\hs\GG^0$ represents
all cohomology classes in $\GG^0$,
we can, using computer as explained in Section \ref{s:connected-equivalence},
find $b_0\in G^0$ such that
\[b_0^{-1}\cdot g\cdot \sigma(b_0)=x\quad\text{for some}\ \, x\in X_1.\]
Since the subset $X_1^{\num}$ represents all orbits of $(G^\ff)^\sigma$,
there exists $b'\in G$ such that
\[(b')^{-1}\cdot x\cdot\sigma(b')=x^{\num}\quad
     \text{where}\ \, x^{\num}\in X_1^{\num}\hs.\]
\end{subsec}

\begin{subsec}\label{ss:gi-gi-ff}
Let $\xi^\ff\in \Hr^1\hs\GG^\ff$.
Write  $\xi^\ff=[g_i^\ff]$ where $g_i^\ff\in \Zr^1\GG^\ff$.
Write $g_i^\ff=\pi(g_i)$ where $g_i\in G$.
As in Subsection \ref{ss:Delta}, we construct $\Delta(g_i)=(a,f)\in \Zr^2\hs G^0$,
where $a=g_i\hm\upgam\hm g_i,\ f=\inn(g_i)\circ\gamma$.
We have $\Delta[\xi^\ff]=[a,f]\in \Hr^2\hs G^0$.
By Proposition \ref{p:Delta}, $\pi_*^{-1} [\xi^\ff]\neq\emptyset$
if and only if $\Delta[g^\ff]\in \Nr^2\hs G^0$.
In Sections \ref{s:neutralizing} and \ref{s:neutralizing-nonred}
we gave a method  to determine whether
$(a,f)=\Delta[x^\ff]$ is contained in $\Nr^2\hs G^0$.
If $(a,f)\notin \Nr^2\hs G^0$, then we conclude that  $\xi^\ff$ cannot be lifted to $\Hr^1\GG$.
If $(a,f)\in \Nr^2\hs G^0$, then Sections \ref{s:neutralizing} and \ref{s:neutralizing-nonred}
give a method to find $s\in G^0$ such that $s\cdot f(s)\cdot a=1$.
For such $s$ we set $g_i'=sg_i$. By Corollary \ref{c:lifting-1-cocycle},
$g_i'\in \Zr^1\GG$ and $j(g_i')=g_i^\ff$.
For simplicity, we write $g_i$ for $g_i'$.
We have lifted $g_i^\ff\in \Zr^1\hs\GG^\ff$ to some $g_i\in \Zr^1\GG$,
and thus we have lifted $\xi^\ff=[g_i^\ff]\in \Hr^1\hs\GG^\ff$ to some $\xi=[g_i]\in \Hr^1\GG$.
\end{subsec}

\begin{subsec}
Recall that  $\GG=(G,\sigma)$ where $\sigma(g)=\upgam g$.
Then $\sigma\in\SAut_\anti(G)$ and $\sigma^2=1$.
For those $i$ for which we have lifted $g_i^\ff\in\Zr^1\GG^\ff$ to $g_i\in\Zr^1\GG$,
we consider $\sigmai=\inn(g_i)\circ\sigma$.
Then $\sigmai\in\SAut_\anti(G)$.
Since $g_i\in \Zr^1(\Gamma,G)$, we have $(\sigmai)^2=1$.
Thus $\sigmai$ is a real structure on $G$.
We consider the real form $\GG_i=\hs_{\gi}\hm\GG=(G,\sigmai)$,
its identity component $\GG^0_i=\hs_{\gi}\hm\GG^0$,
and its component group $\GG_i^\ff=\GG_i/\GG_i^0$.
We have a short exact sequence
\[1\to\GG_i^0\labelto{ }\GG_i\labelto{\pi_i}\GG_i^\ff\to 1\]
and a cohomology exact sequence
\[\GG_i^\ff(\R)\labelto{ } \Hr^1\hs\GG_i^0\labelto{ }
       \Hr^1\hs\GG_i\labelto{\pi_{i*}}\Hr^1\hs\GG_i^\ff.\]
As in Subsection \ref{ss:H1-ker},
we can use computer in order to find a set of 1-cocycles
$X^{\num}_i\subset \Zr^1\hs\GG_i$ containing 1,
representing all cohomology classes in $\ker \pi_{i*}$\hs.

Consider the map
\[ G\to G,\quad g\mapsto g\hs g_i\ \ \text{for}\ g\in G.\]
This map restricts to a bijection
\[ t_i\colon \Zr^1\hs\GG_i\to \Zr^1\hs\GG\]
and induces a bijection
\[\tau_i\colon \Hr^1\hs\GG_i\to \Hr^1\hs\GG;\]
see Serre \cite[Section I.5.3, Proposition 35 bis]{Serre}.
Moreover, this map $\tau_i$ restricts to a bijection
\[\ker\pi_{i*}\labelto\sim\pi_*^{-1}[g_i^\ff];\]
see Serre \cite[Section I.5.5, Corollary 2 of Proposition 39]{Serre}.
Thus the set $X^{\num}_i\cdot g_i$ is a set
of representatives of all cohomology classes
in $\pi_*^{-1}[g_i^\ff]$, and the union
\begin{equation}\label{e:Xg}
\Xg\coloneqq\bigcup_i X^{\num}_i g_i\,\subseteq\, \Zr^1\hs\GG
\end{equation}
is a set of representatives of all cohomology classes in $ \Hr^1\hs\GG$.
Thus we can solve Problem \ref{mp:1}
for a not necessarily connected and not necessarily reductive $\R$-group $\GG$.
\end{subsec}

\begin{subsec}
Here we summarize the algorithm for computing $\Hr^1 \GG$ where
$\GG=(G,\sigma)$ is
not necessarily connected. The input to our algorithm consists of the following:
\begin{itemize}
\item a  basis of the real  Lie algebra  $\Lie \GG^0$,
\item a multiplication table as well as a table giving the action of complex conjugation on
  the finite group $G^\ff = G/G^0$,
\item for each $c^\ff\in G^\ff$ a preimage $g\in G$.
\end{itemize}
The algorithm starts by computing $\Hr^1 \GG^\ff$ (by brute force; see Remark \ref{r:brute-force}). We set
$H_1 = \emptyset$, and for each $[c^\ff]\in \Hr^1 \GG^\ff$ we do the following:
\begin{enumerate}
\item Let $g\in G$ be a preimage of $c^\ff$ and set $h=g\upgam g$, $f = \inn(g) \circ \sigma$.
\item Use the algorithm of Section \ref{s:neutralizing} to decide
  whether $[h,f] \in \Nr^2\hs \GG^0$. If not then discard $[c^\ff]$, if yes
  then find $s\in G^0$ with $s\cdot f(s)\cdot h=1$.
\item Set $\hat g = sg$, $\hat \sigma = \inn(\hat g) \circ \sigma$, and consider the twisted group
  $_{\hat g}\GG= (G,\hat\sigma)$. Compute $\Hr^1{} _{\hat g}\GG^0$ and
  $(\hs_{\hat g}\GG^\ff)(\R)=(G^\ff)^{\hat\sigma}$.
\item Find a set $X_1^\num$ of orbit representatives relative to the action of $(\hs_{\hat g}\GG^\ff)(\R)$
   on $\Hr^1{} _{\hat g}\GG^0$. Replace
  $H_1$ by  $H_1 \cup X_1^\num\hat g$.
\end{enumerate}
We remark that if $\GG^0$ is a torus, then step (2) becomes easier. In this case
we consider the twisted group $_g\GG = (G,\tilde\sigma)$
where $\tilde\sigma=\inn(g)\circ\sigma$. Then $c^\ff$ lifts to a
cocycle in $\GG$ if and only if $h\in \Br^2{} _g\GG^0$. Indeed,
there is $b\in G^0$ with $(bg) \cdot \sigma(bg) = 1$ if and only if
$b\hs g\hs \sigma(b) g^{-1} g\hs \sigma(g) = 1$, which is the same as $b\hs \tilde\sigma(b)
h=1$, which in turn is equivalent to $h\in \Br^2{} _g\GG^0$. Note that with
the algorithm of Subsection \ref{ss:C-is-torus} we can decide whether
$h\in \Br^2 {} _g\GG^0$, and in the affirmative case find an element  $c\in
G^0$ with $h=c \hs\tilde\sigma(c)$. Then with $b=c^{-1}$ we have that $\hat g
=bg$ is a cocycle in $\GG$, and in Step (3) we work with this $\hat g$.
\end{subsec}

\begin{subsec}
We consider Problem \ref{mp:2}
for our not necessarily connected $\R$-group $\GG$.
We have a set $\Xg\subset \Zr^1\GG\subset G$ of 1-cocycles
representing all cohomology classes in $\Hr^1\GG$.
By \eqref{e:Xg} we have $\Xg=\bigcup_i X^{\num}_i g_i$
where  $X^{\num}_i g_i$ is a set
of representatives of all cohomology classes
in $\pi_*^{-1}[g_i^\ff]\subseteq\Hr^1\GG$.

Let $c\in\Zr^1\hs\GG\subset G$ be a 1-cocycle, and let $c^\ff\in\Zr^1\hs\GG^\ff$
denote its image in $G^\ff$.
Since $\Xg$ is the set of representatives of {\em all} classes in $\Hr^1\GG$,
we know that there exists an element $s\in G$ such that $s^{-1} c\hs \upgam s\in \Xg$
(but we do not know this element $s$).
It follows that there exists an element $s^\ff\in G^\ff$ such that
\begin{equation}\label{e:sf-gf}
(s^\ff)^{-1}\cdot g^\ff\cdot\upgam s^\ff=g_i^\ff.
\end{equation}
for some 1-cocycle $g_i^\ff\in \Zr^1\GG^\ff$
that can be lifted to $\Zr^1\GG$.
Since $G^\ff$ is a finite group, we can find such $s^\ff$
and $g_i^\ff$   by brute force.

Let $g_i\in\Zr^1\hs\GG$ be the lift of $g_i^\ff$
constructed in Subsection \ref{ss:gi-gi-ff}.
Let $s\in G$ be such that $s^\ff=s\cdot G^0$.
Set
\begin{equation}\label{e:h}
h= s^{-1}\cdot g\cdot\upgam s\cdot g_i^{-1}.
\end{equation}
It follows from \eqref{e:sf-gf} that $h\in G^0$.
Moreover, since $hg_i\in \Zr^1\hs\GG$, we see that
$h\in \Zr^1\hs \GG_i$ where $\GG_i={}_{g_i}\GG=(G,\sigma_i)$
with $\sigma_i=\inn(g_i)\circ\sigma$.
Thus $h\in\Zr^1\hs \GG_i^0$.
Then using the method of Subsection \ref{ss:b-x0},
we can find $b\in G$ such that
\begin{equation}\label{e:x-num}
b^{-1}\cdot h\cdot\sigmai(b)=x_i^{\num}\quad\text{where}\ \, x_i^{\num}\in X^{\num}_i.
\end{equation}
It follows from \eqref{e:h} and \eqref{e:x-num} that
\[x_i^{\num}=b^{-1}\cdot s^{-1} g\upgam\hm s \hs g_i^{-1}\cdot g_i\upgam b\hs g_i^{-1}
        =(sb)^{-1} g\,\sigma(sb)\cdot g_i^{-1},\]
whence
\[(sb)^{-1} g\,\sigma(sb)=x_i^{\num} g_i\in  X_i^{\num}g_i\hs.\]
We see that we can find  elements $b'=sb\in G$ and
$x_i'=x_i^{\num}g_i\in X_i^\num g_i\subseteq\Xg$
such that
\[(b')^{-1}\cdot g\cdot\hm\upgam b'=x_i'.\]
Thus we can solve Problem \ref{mp:2}
for a not necessarily connected and  not necessarily reductive $\R$-group $\GG$.
\end{subsec}

\section{Implementation and practical experiences}\label{sec:gap}

We have implemented the algorithms in the computational algebra system
{\sf GAP}4, \cite{GAP4}. Especially useful was some functionality for
working with real semisimple Lie algebras in the package {\sf CoReLG}; see \cite{corelg}.

We need to say some words on the field of definition of our algebras and
matrices. First of all we want to stress that we need exact arithmetic.
But also the algorithms occasionally require to take $n$-th roots of field
elements (see Subsections \ref{equiv:1}, \ref{ss:C-is-torus}, \ref{neut:7})
and also the algorithm of \cite{cmt} needs to take $n$-th roots). However, {\sf GAP} does
not implement a field where such operations are possible. So we have
implemented our own field for this. We have taken the field {\tt SqrtField}
from the package {\sf CoReLG} and transformed it into a dynamic field. That is,
every time we need an $n$-th root that is not contained in the field, the
latter is extended. For the number theoretic computations that are necessary
(e.g., factorization of polynomials over number fields) we use the system
{\sf SageMath}, \cite{sage}, which is called from {\sf GAP}.

Because the field operations of the field {\tt SqrtField} are implemented
in the {\sf GAP} language, and are not supported by the {\sf GAP} kernel,
using this field inevitably slows the computations down. If the group is
connected, then for computing the first Galois cohomology set it is not
necessary to take $n$-th roots, and  so it is possible to use a field with
{\sf GAP} kernel support such as $\Q(i)$.
We show in Table \ref{table:1} the difference that  it makes.
 The example inputs that we use are the groups
$\SO(p,q)$, which are defined as follows:
$$\SO(p,q) = \{ g\in \GL(n,\C) \mid g^T I_{p,q} g = I_{p,q}, \det(g)=1\},
    \text{ where } I_{p,q}=\begin{pmatrix} I_p & 0 \\ 0 & -I_q\end{pmatrix}$$
with Lie algebra
$$\mathfrak{so}(p,q) = \{ X\in M_n(\C) \mid X^TI_{p,q} +I_{p,q} X = 0\}.$$
A real basis of the latter is the input to our algorithms.

The sixth column of Table \ref{table:1}
has the total running time when using the {\sf GAP}
field {\tt CF(4)} (which is $\Q(i)$). The seventh column has the running time
when using the field {\tt SqrtField}. We see that the difference is rather
big.

A few more things can be remarked. The fifth column has the time needed for
initialization (when using $\Q(i)$). In the initialization phase a
multiplication table of the Lie algebra is computed, as well as a Cartan
decomposition. We see that the initialization takes a large part of the time
for the lower ranks, whereas it is not so important for the higher ranks.
This is explained by the fact that we compute $\Hr^1 T$, where $T$ is the
centralizer of $T_0$, on which we take the orbits of $W_0$ (see Section
\ref{s:connected}). Now the size of  $\Hr^1 T$ is $2^{\dim T_0}$, so for the
larger ranks this becomes the dominant factor.

We also note that the order of the Weyl group does not play a very
important role. For $\SO(13,17)$ we have $|W|=21424936845312000$ and
$|W_0|=475634073600$. A generating set of $W_0$, consisting of 17 elements, is
found in 142 seconds using the method outlined in Remark \ref{rem:realweyl}
(in this case the real Weyl group turns out to be equal to $W_0$ so there were
45045 coset representatives to check in order to establish this).
The action of these 17 elements on the $2^{14}$ elements of $\Hr^1 T_0$ is
computed in 218 seconds. Finally, a set of orbit representatives is computed
in 14 milliseconds.

\begin{table}[htb]
\begin{tabular}{|l|r|r|r|r|r|r|}
  \hline
  $G$ & $\dim G$ & $\dim T_0$ & $|\Hr^1 \GG|$ & \multicolumn{3}{|c|}{Running time (s)}\\
  \hline
   &  &  & & \multicolumn{2}{|c|}{$\Q(i)$} &  {\sf SqrtField} \\
  \hline
   & & & & init & total & \\
  \hline
  $\SO(6,9)$ & 105 & 7 & 8 & 3.9 & 5.2 & 140 \\
  $\SO(7,11)$ & 153 & 8 & 9 & 10.9 & 14.8 & 388 \\
  $\SO(7,13)$ & 190 & 9 & 10 & 19.7 & 27.3 & 1162\\
  $\SO(13,17)$ & 435 & 14 & 15 & 218 & 669 &  \\
  \hline
\end{tabular}
\caption{Running times (in seconds) of the algorithm to
         compute $\Hr^1\hs \SO(p,q)$.}\label{table:1}
\end{table}

\begin{remark}
We can check by hand the validity of the cardinalities $\#\Hr^1 \GG$ for $\GG=\SO(p,q)$ in Table 1.
Assume that at least one of the numbers $p,q$ are odd. Then $\Hr^1 \GG$
classifies symmetric matrices of quadratic forms in $p+q$ variables with determinant $-1$;
see Serre \cite[Section III.1.1]{Serre}.
By Sylvester's law of inertia, see, for instance, \cite[Corollary XV.4.2]{Lang},
such a matrix is equivalent to a diagonal matrix
of the form $\diag(-1,\dots,-1,+1,\dots,+1)$ with $r$ times $-1$,
and this number  $r$ is uniquely determined and  odd.
Thus $|\Hr^1 \GG|$ equals the number of odd natural numbers between 1 and $p+q$,
which is clearly equal to $\big\lceil\frac{p+q}{2}\big\rceil$.
For the groups $\SO(p,q)$ in Table 1 we obtain
\[
\big\lceil\tfrac{6+9}{2}\big\rceil=8, \quad \big\lceil\tfrac{7+11}{2}\big\rceil=9,\quad
    \big\lceil\tfrac{7+13}{2}\big\rceil=10,\quad \big\lceil\tfrac{13+17}{2}\big\rceil=15.
\]
\end{remark}

\newcommand{\DDD}{{\sf D}}
\newcommand{\EEE}{{\sf E}}

For a set of examples of non-connected groups we consider the groups
$\Aut \gR$, where $\gR$ is the split real form
of the simple Lie algebra of type $\DDD_n$ for $n=4,5,6$ or $\EEE_6$.
Here the component group has order 2, except for $\DDD_4$, in which case the
component group has order 6. For non-connected groups we can only use the
field {\tt SqrtField}. The running times for $n=4,5,6$ are displayed in
Table \ref{table:2}. We see that the running times rise sharply with the
dimension. This is mainly caused by the inefficiency of the field operations
in our field of definition.
On the other hand, we also see that the program is able to
deal with groups of moderate rank and dimension.

\begin{table}[htb]
\begin{tabular}{|l|r|r|r|}
  \hline
  type & $\dim G$ & $|\Hr^1 \GG|$ & time (s) \\ \hline
  $\DDD_4$ & 28 & 5 & 1506\\
  $\DDD_5$ & 45 & 7 & 3187\\
  $\DDD_6$ & 66 & 8 & 21533\\
  $\EEE_6$ & 78 & 5 & 25821\\
  \hline
\end{tabular}
\caption{Running times (in seconds) of the algorithm to compute
  $\Hr^1 (\Aut \g)$.}\label{table:2}
\end{table}

\begin{remark}
We can check by hand the validity of the cardinalities $\#\Hr^1 \GG$ for $\GG=\Aut\gR$ in Table 2.
Here  $\Hr^1 \GG$ classifies  the real forms of the complex Lie algebra $\g=\gR_\C$;
see Serre \cite[Section III.1.1]{Serre}.

Classification of real forms of simple Lie algebras,
in particular, of exceptional simple Lie algebras,
goes back to \'Elie Cartan's paper \cite{Cartan}.
For a  modern classification see  Kac \cite{Kac69};
see also Helgason \cite[Table V in Section X.6]{Helgason},
and  Onishchik and Vinberg \cite[Tables 7 and 9]{OV}.

The real forms of the complex Lie algebra $\so(2n)$ of type $\DDD_n$ for $n>4$
are the real Lie algebras $\so(p,2n-p)$ for $0\le p\le n$ and the quaternionic form $\so^*(2n)$,
and we obtain together $n+2$ pairwise non-isomorphic  real forms.
The similar count holds for $n=4$ except for the fact that $\so^*(8)\simeq\so(2,6)$;
see, for instance, \cite{Helgason}, Section X.6.4, case (viii).
We obtain together $n+1=5$ pairwise non-isomorphic real forms.
Thus for $\DDD_4$, $\DDD_5$, $\DDD_6$ we obtain the numbers of real forms
(the cardinalities of $\Hr^1 \GG$) $5,7,8$, respectively.

For the complex Lie algebra $\eee_6$ of type $\EEE_6$, there are 4
pairwise non-isomorphic noncompact real forms; see, for instance, \cite[Table V]{Helgason}.
Together with the compact form, we obtain 5 pairwise non-isomorphic real forms of $\eee_6$,
that is, the cardinality of $\Hr^1 \GG$ for $\GG=\Aut(\eeeR_6)$ is 5.
\end{remark}

\appendix

\section{Anti-regular maps}
\label{app:A}

Here we discuss anti-regular maps of complex affine varieties
and anti-regular automorphisms of complex  algebraic groups.
For the notion of a semi-linear morphism of schemes over an arbitrary field
see \cite{Borovoi20}.
In this appendix  we write $Y(\C)$ (and not just $Y$)
for the set of $\C$-points of a $\C$-variety $Y$.

Let $X$ and $Y$ be affine varieties over $\C$. By an affine variety over $\C$ we mean
a reduced affine scheme of finite type over $\C$, not necessarily irreducible.
We may assume that our affine variety $Y$ is embedded as a Zariski-closed subset
into the affine space $\C^n$ for some natural $n$.
Then the {\em regular functions} on $Y(\C)$ are the restrictions of the polynomials on $\C^n$.

\begin{definition}\label{d:reg}
We say that a map on $\C$-points
$\varphi\colon X(\C)\to Y(\C)$ is {\em regular},
if it comes from a morphism of varieties $X\to Y$.
In other words, $\varphi$ is regular
if for any regular function $f\colon Y(\C)\to\C$,
the function $\varphi^*(f)\colon X(\C)\to \C$ given by
\[ (\varphi^*\hm f)(x)=f(\varphi(x)) \]
is regular.
\end{definition}

\begin{definition}\label{d:reg-anti-reg}
We say that $\varphi\colon X(\C)\to Y(\C)$ is {\em anti-regular},
if for any regular function $f\colon Y(\C)\to\C$,
the function  $\ov{\varphi^*\hm f}\colon X(\C)\to \C$ given by
\[ \big(\hs\ov{\varphi^*\hm f}\hs\big)(x)=\overline{f(\varphi(x))} \]
is regular, where the bar denotes  complex conjugation.
\end{definition}

\begin{lemma}[easy]
\label{l:reg-anti-reg}
\begin{enumerate}
\item[(i)] The composition of two regular maps,
and the composition of two anti-regular maps, are regular.
\item[(ii)] The composition of a regular map and an anti-regular map,
in any order, is anti-regular.
\end{enumerate}
\end{lemma}

\begin{remark}\label{r:finite}
If the variety $X$ is finite (as a set), then all functions $X(\C)\to\C$ are regular,
and therefore any map $X(\C)\to Y(\C)$ is simultaneously regular and anti-regular.
\end{remark}

\begin{subsec}
Let $\YY$ be a {\em real} affine variety.
In the coordinate language, the reader may regard $\YY$
as an algebraic subset in $\C^n$ (for some positive integer $n$)
defined by polynomial equations with {\em real} coefficients.
More conceptually, the reader may assume that $\YY$
is a reduced affine scheme of finite type over $\R$.
With any of these two equivalent definitions,
$\YY$ defines a covariant functor
\[A\rightsquigarrow \YY(A)\]
from the category of commutative unital $\R$-algebras to the category of sets.
Applying this functor to the $\R$-algebra $\C$ and the morphism of $\R$-algebras
\[\gamma\colon \C\to \C,\quad z\mapsto\ov z\ \ \text{for}\ z\in\C,\]
we obtain a set (a complex analytic space) $\YY(\C)=Y(\C)$ together with a map
\begin{equation}\label{e:tau-X}
\sigma_\yY=\YY(\gamma)\hs\colon\hs Y(\C)\to Y(\C),
\end{equation}
where we write $Y=\YY\times_\R\C$, the base change of $\YY$ from $\R$ to $\C$.
By functoriality we have $\sigma_\yY^2=\id$,
and by Lemma \ref{l:anti} below the map $\sigma_\yY$ is anti-regular.
We say that $\sigma_\yY$ is an {\em anti-regular involution} of $Y$.
\end{subsec}

\begin{lemma}\label{l:anti}
For a real affine variety $\YY$,
the map \eqref{e:tau-X} defined above is anti-regular.
\end{lemma}

\begin{proof}
We may and will assume that $\YY$
is embedded into the real affine space $\A_\R^n$
for some positive integer $n$.
For a complex point $y\in Y(\C)\subset \C^n$
with coordinates $(y_1,\dots,y_n)$,
the point $\sigma_\yY(y)$ has coordinates $(\ov y_1,\dots,\ov y_n)$.
A regular function $f$ on $Y$ is  the restriction to $Y$ of a polynomial $P$
in the coordinates $y_i$ with certain coefficients $c_\alpha$.
An easy calculation shows that $\sigma_\yY^*f$ is the restriction to $Y$
of the complex conjugate polynomial $\upgam P$
(the polynomial with coefficients $\upgam c_\alpha$),
hence a regular function on $Y$.
It follows that the map $\sigma_\yY$ is anti-regular, as required.
\end{proof}

\begin{subsec}\label{ss:G-anti}
Let $\GG$ be a real algebraic group.
As above, it defines a complex algebraic group $G=\GG\times_\R\C$
and an anti-regular  group automorphism
\[\sigma_\gG=\GG(\gam)\hs\colon\hs G(\C)\to G(\C)\]
such that $\sigma_\gG^2=\id$;
see, for instance, \cite[Section 1.1]{BT21}.
We say that $\sigma_\gG$ is an {\em anti-regular involution} of $G$.
Thus from $\GG$ we obtain a pair $(G,\sigma_\gG)$.

Conversely, by Galois descent any pair $(G,\sigma)$, where $G$
is a {\em complex} algebraic group
and $\sigma\colon G(\C)\to G(\C)$ is an anti-regular involution of $G$,
comes from a unique (up to a canonical isomorphism)
{\em real} algebraic group $\GG$;
see Serre \cite[V.4.20,  Corollary 2 of Proposition 12]{Serre-AGCF},
or the book ``N\'eron models'' \cite[Section 6.2, Example B]{BLR},
or Jahnel \cite[Theorem 2.2]{Jahnel}.
We do not use this fact.
For us, a real algebraic group is a pair $(G,\sigma)$ as above,
and we write $\GG=(G,\sigma)$.
We say that the real algebraic group $(G,\sigma)$
is a {\em real form} of the complex algebraic group $G$.

Note that if $G$ is reductive or unipotent,
then any anti-holomorphic involution of $G$ is anti-regular;
see Cornulier  \cite{Cornulier}.
The hypothesis that $G$ is either reductive or unipotent, is necessary:
the commutative algebraic group $\G_{{\rm a},\C}\times\GmC=\C\times\C^\times$
has the  anti-holomorphic involution
$(z,w)\mapsto \big(\ov{z},\exp(i\ov{z})\hs\ov{w}\big)$
that is not anti-regular.
\end{subsec}

\begin{subsec}\label{ss:SAut}
Let $G$ be a complex algebraic group.
By a {\em semi-automorphism} of $G(\C)$ we mean a regular or anti-regular automorphism.
We denote by $\SAut G$ the group of semi-automorphisms of $G(\C)$,
where the group law is defined by composition of maps.
In other words, $\SAut G$ is the union of the group of regular (usual) automorphisms $\Aut G$
and the set of anti-regular automorphisms $\SAut_\anti(G)$.
\end{subsec}


\begin{thebibliography}{99}

\bibitem
{at}
J. Adams and O. Ta\"{\i}bi,
{Galois and Cartan cohomology of real groups},
{\em Duke Math. J.} {\bf 167} (2018), no. 6, 1057–1097.

\bibitem
{AW}
M.\,F. Atiyah and C.\,T.\,C. Wall,
 {Cohomology of groups}, in:
  {\em {A}lgebraic {N}umber {T}heory} (J.W.S. Cassels and A. Fr{\"o}hlich, eds.),
2nd edition, London Math. Soc., London, 2010,  pp.~94--115.

\bibitem
{bcptr}
E. Bergshoeff, W. Chemissany,  A. Ploegh,  M. Trigiante, and T. Van Riet,
{Generating geodesic flows and supergravity solutions},
{\em Nuclear Phys.} B {\bf  812} (2009), no. 3, 343--401.

\bibitem
{Berhuy}
G. Berhuy, {\em An Introduction to Galois Cohomology and its Applications,}
 Cambridge University Press, Cambridge, 2010.

\bibitem
{Borel-66}
A. Borel,
{Linear algebraic groups},
in: \emph{Algebraic Groups and Discontinuous Subgroups},
Proc. Sympos. Pure Math., Vol. 9, Amer. Math. Soc., Providence, R.I., 1966, pp. 3--19.


\bibitem
{BS}
A. Borel et J.-P. Serre,
{Th\'eor\`emes de finitude en cohomologie galoisienne},
{\em Comm. Math. Helv}. {\bf 39} (1964), 111--164.
(= A. Borel, {\em \OE uvres: collected papers}, 64, Vol. II, Springer-Verlag, Berlin, 1983).

\bibitem
{BT72}
A.  Borel and J. Tits,
{Compl\'ements \`a l’article: ``Groupes r\'eductifs''},
{\em Inst. Hautes Etudes Sci. Publ. Math.} {\bf 41} (1972), 253--276.


\bibitem
{B88}
M.\,V.~Borovoi,
{Galois cohomology of real reductive groups,
and real forms of simple Lie algebras.} {\em Functional. Anal. Appl.}
{\bf 22}:2 (1988), 135--136.

\bibitem
{Borovoi-Duke}
M.~Borovoi,
{Abelianization of the second nonabelian {G}alois
cohomology}, {\em Duke Math. J.}  {\bf 72} (1993),  no.~1,  217--239.

\bibitem
{Borovoi20}
M.~Borovoi, with an appendix by G.~Gagliardi,
{Equivariant models of spherical varieties}, {\em Transform. Groups} {\bf 25} (2020), 391--439.

\bibitem
{Borovoi22-CiM}
M.~Borovoi,
{Galois cohomology of reductive algebraic groups over the field of real numbers,}
{\em Commun. Math.} {\bf 30} (2022), no.  3, 191--201.

\bibitem
{BDR}
M.~Borovoi, C.~Daw, and J.~Ren,
{Conjugation of semisimple subgroups over real number fields of bounded degree,}
{\em Proc. Amer. Math. Soc.} {\bf 149} (2021), no. 12, 4973--4984.

\bibitem
{BE}
M. Borovoi and Z. Evenor,
{Real homogeneous spaces, Galois cohomology, and Reeder puzzles},
{\em J. Algebra} {\bf 467} (2016), 307--365.

\bibitem
{BGR}
M.~Borovoi, A.\,A.~Gornitskii, and Z.~Rosengarten,
{Galois cohomology of real quasi-connected reductive groups,}
{\em Arch. Math. (Basel)} {\bf 118} (2022), 27--38.

\bibitem
{BGL21}
M.~Borovoi, W.\,A.~de Graaf, and H.\,V.~L\^e,
\emph{Real graded  Lie algebras, Galois cohomology,
and classification of trivectors in $\R^9$},
{\tt arXiv:2106.00246 [math.RT]}.

\bibitem
{BK}
M.~Borovoi and T.~Kaletha,
\emph{Galois cohomology of reductive groups over global fields},
{\tt arXiv:2303.04120 [math.NT]}.

\bibitem
{BT21}
M. Borovoi and D.\,A.~Timashev,
{Galois cohomology of real semisimple groups via Kac labelings},
{\em Transform. Groups} {\bf 26} (2021), no. 2, 433–477.

\bibitem
{BT*}
M.~Borovoi and D.\,A.~Timashev,
\emph{Galois cohomology and component group of a real reductive group},
to appear in Israel J. Math., DOI: 10.1007/s11856-023-2526-4.

\bibitem
{BLR}
S. Bosch, W. L\"utkebohmert, and M. Raynaud,
\emph{N\'eron models},
Ergebnisse der Mathematik und ihrer Grenzgebiete (3), 21, Springer-Verlag, Berlin, 1990.

\bibitem
{bou1}
N. Bourbaki, {\em Groupes et alg\`ebres de {L}ie}, Chapitre 1, Hermann, Paris,
1972.

\bibitem
{Brown}
K.\,S.~Brown,
{\em Cohomology of groups,}
Graduate Texts in Mathematics, 87,
Springer-Verlag, New York, 1994.

\bibitem
{Cartan}
E. Cartan,
{Les groupes r\'eels simples, finis et continus}, {\em Ann. Sci. \'Ecole Norm. Sup.} (3) {\bf 31} (1914), 263--355.

\bibitem
{Casselman}
W.~Casselman,
{ Computations in real tori}, in: {\em Representation theory of real reductive Lie groups},
volume 472 of Contemp. Math., pages 137--151, Amer. Math. Soc., Providence, RI, 2008.

\bibitem
{cgpt}
W. Chemissany, P. Giaccone, D. Ruggeri, and M. Trigiante,
{Black hole solutions to the F4-model and their orbits (I)},
{\em Nuclear Phys.} B {\bf 863} (2012), no. 1, 260--328.

\bibitem
{cmt}
A.\,M.~Cohen, S.\,H.~Murray, and D.\,E.~Taylor,
 {Computing in groups of {L}ie type},
 {\em  Math. Comp.}, {\bf 73}(247) (2004), 1477--1498.

\bibitem
{Conrad}
B. Conrad,
Reductive group schemes, in: \emph{Autour des sch\'emas en groupes}, Vol. I, 93--444,
Panor. Synth\`eses, 42/43, Soc. Math. France, Paris, 2014.

\bibitem
{C}
B. Conrad,
{Non-split reductive groups over $\bf Z$,} in:
{\em  Autours des sch\'emas en groupes}, Vol. II, 193--253,
Panor. Synth\`eses, 46, Soc. Math. France, Paris, 2015.

\bibitem
{CGP}
B. Conrad, O. Gabber, G. Prasad,
 {\em Pseudo-reductive groups}, 2nd edition. New Mathematical Monographs, 26,
 Cambridge University Press, Cambridge, 2015.

\bibitem
{Cornulier}
Y.~Cornulier
(\url{https://mathoverflow.net/users/14094/ycor}),
{\em Anti-holomorphic involutions of a complex linear algebraic group}, MathOverflow,
\url{https://mathoverflow.net/q/342328} (version: 2019-09-24).

\bibitem
{CR}
 C.\,W. Curtis and  I. Reiner,
{\em Representation Theory of Finite Groups and Associative Algebras,}
 AMS Chelsea Publishing, Providence, RI, 2006.

\bibitem
{dfg}
H.~Dietrich, P.~Faccin, and W.\,A.~de Graaf,
{Computing with real {L}ie algebras: real forms, {C}artan
decompositions, and {C}artan subalgebras},
{\em J. Symbolic Comput.}, {\bf 56} (2013), 27--45.

\bibitem
{dg}
H.~Dietrich and W.\,A.~de Graaf,
{A computational approach to the {K}ostant-{S}ekiguchi correspondence},
 {\em Pacific J. Math.}, {\bf 265} (2013), 349--379.

\bibitem
{corelg}
H.~Dietrich, P.~Faccin, and W.~de~Graaf,
\emph{{CoReLG}, computing with real
  {L}ie algebras, {V}ersion 1.54},
  {\texttt{https://gap-packages.github.io/}\discretionary
    {}{}{}\texttt{corelg/}}, Jan 2020, Refereed GAP package.

\bibitem
{dg21}
H.~Dietrich and W.\,A.~de Graaf,
{Computing the real Weyl group},
{\em J. Symbolic Comput.} {\bf 104} (2021), 1–14.

\bibitem
{Djokovic}
D.\,\v{Z}. Djokovi\'{c},
{Classification of trivectors of an eight-dimensional real vector space},
 {\em  Linear and Multilinear Algebra}, {\bf 13} (1983), 3--39.

\bibitem
{Douai}
J.-C. Douai,
\emph{2-Cohomologie galoisienne des groupes semi-simples},
  th\`{e}se, Universit\'{e} de Lille I, 1976,
\'Editions universitaires  europ\'eennes, Saarbr\"ucken, 2010.

\bibitem
{FSS}
Y.\,Z. Flicker, C. Scheiderer, and  R. Sujatha,
{Grothendieck's theorem on non-abelian  $\Hr^2$  and local-global principles},
{\em J. Amer. Math. Soc.} {\bf 11} (1998), no. 3, 731--750.

\bibitem
{Florence}
M. Florence,
{Z\'ero-cycles de degr\'e un sur les espaces homog\`enes},
{\em Int. Math. Res. Not.} {\bf (2004)}, no. 54, 2897--2914.

\bibitem
{GAP4}
The GAP~Group,
\emph{GAP -- Groups, Algorithms, and Programming,
Version 4.12.2};
2022,
 {\tt https://www.gap-system.org}.

\bibitem
{GOV}
 V.\,V.~Gorbatsevich, A.\,L.~Onishchik, and  E.\,B. Vinberg,
{\em Structure of Lie groups and Lie algebras,}
{Lie Groups and Lie Algebras III,}
Encyclopaedia of Mathematical Sciences, Vol. 41,
Springer-Verlag,  Berlin, 1994.

\bibitem
{wdg1}
W.\,A. ~de Graaf,
{\em Lie Algebras: Theory and Algorithms}, volume~56 of
  North-Holland Mathematical Library, Elsevier Science, 2000.

\bibitem
{wdgrep}
W.\,A.~de Graaf,
Constructing representations of split semisimple Lie algebras,
{\em J. Pure Appl. Algebra} {\bf 164} (2001), 87--107.

\bibitem
{wdg}
W.\,A.~de Graaf,
{\em Computation with linear algebraic groups},
Monographs and Research Notes in Mathematics,
CRC Press, Boca Raton, FL, 2017.

\bibitem
{Helgason}
S. Helgason,
{\em  Differential geometry, Lie groups, and symmetric spaces},
Grad. Stud. Math., 34,
American Mathematical Society, Providence, RI, 2001.

\bibitem
{hit1}
N. Hitchin,
{The geometry of three-forms in six dimensions},
{\em J. Differential Geom.} {\bf 55} (2000), no. 3, 547--576.

\bibitem
{hit2}
N. Hitchin,
Stable forms and special metrics,
in: {\em Global differential geometry: the mathematical legacy of Alfred Gray (Bilbao, 2000)},
Contemp. Math., 288, American Mathematical Society, Providence, RI, 2001, 70--89.

\bibitem
{hum}
J.\,E.~Humphreys,
{\em Introduction to {L}ie algebras and representation theory},
  volume~9 of {Graduate Texts in Mathematics},
Springer-Verlag, New York-Berlin, 1978.

\bibitem
{jac}
N.~Jacobson,
{\em {Lie Algebras}}, Dover, New York, 1979.

\bibitem
{Jahnel}
J. Jahnel,
{\em The Brauer-Severi variety associated with a central simple algebra:
a survey}, Preprint server: Linear Algebraic Groups and Related Structures, no. 52, 2000,
\url{https://www.math.uni-bielefeld.de/LAG/man/052.pdf}.

\bibitem
{Kac68}
V.\,G. Kac,
{Graded {L}ie algebras and symmetric spaces},
{\em Funkcional. Anal. i Prilo\v{z}en.}, {\bf 2}(2)  (1968), 93--94,
English tansl.: {\em Funct. Anal.  Appl.},  {\bf 2}(2) (1968), 182--183.

\bibitem
{Kac69}
V.\,G. Kac,
\newblock {\em Automorphisms of finite order of semisimple {L}ie algebras},
\newblock {Funkcional. Anal. i Prilo\v{z}en.}, {\bf  3}(3) (1969), 94--96.
English transl.: Funct. Anal.  Appl., {\bf 3}(3) (1969), 252--254.

\bibitem
{knapp}
A.\,W.~Knapp,
{\em Lie groups beyond an introduction}, Vol. 140 of { Progress
  in Mathematics}.
Birkh\"auser Boston Inc., Boston, MA, second edition, 2002.

\bibitem
{Kottwitz86}
R.\,E.~Kottwitz,
{Stable trace formula: elliptic singular terms},
{\em Math. Ann.} 275 (1986), no. 3, 36--399.

\bibitem
{Lang}
S. Lang,
{\em Algebra},
Grad. Texts in Math., 211
Springer-Verlag, New York, 2002.

\bibitem
{LA}
G. Lucchini Arteche,
{Approximation faible et principe de Hasse
pour des espaces homog\`enes \`a stabilisateur fini r\'esoluble},
{\em Math. Ann.} {\bf 360} (2014), no.3--4, 1021--1039.

\bibitem
{Milne-AG}
J.\,S. Milne,
{\em Algebraic Groups:
The theory of group schemes of finite type over a field,}
Cambridge Studies in Advanced Mathematics, 170,
Cambridge University Press, Cambridge, 2017.

\bibitem
{NP}
A. Nair and D. Prasad,
{Cohomological representations for real reductive groups,}
{\em J. Lond. Math. Soc.} (2) {\bf 104} (2021), no. 4, 1515--1571.

\bibitem
{OV}
A.\,L. Onishchik and E.\,B. Vinberg,
{\em Lie groups and algebraic groups},
Springer Ser. Soviet Math.,
Springer-Verlag, Berlin, 1990.


\bibitem
{Sansuc}
J.-J. Sansuc, {Groupe de Brauer et arithm\'etique des groupes alg\'ebriques lin\'eaires
sur un corps de nombres}, {\em J. Reine Angew. Math.} {\bf 327} (1981), 12--80.

\bibitem
{Sch}
C. Scheiderer,
{Hasse principles and approximation theorems for homogeneous spaces
over fields of virtual cohomological dimension one},
{\em Invent. Math.} {\bf 125} (1996), no. 2, 307--365.

\bibitem
{Serre-AGCF}
J.-P. Serre,
{\em Algebraic Groups and Class Fields},
Graduate Texts in Mathematics, Vol. 117, Springer-Verlag, New York, 1988.

\bibitem
{Serre}
J.-P. Serre,
\emph{Galois cohomology}, Springer-Verlag, Berlin, 1997.

\bibitem
{sims}
C.~C. Sims,
{\em Computation with Finitely Presented Groups},
Cambridge University Press, Cambridge, 1994.

\bibitem
{Springer}
T.\,A.~Springer,
Nonabelian {$\Hr^{2}$} in {G}alois cohomology. In:
\emph{Algebraic {G}roups and {D}iscontinuous {S}ubgroups ({P}roc. {S}ympos. {P}ure
{M}ath., {B}oulder, {C}olo., 1965), Proc. Sympos. Pure Math. IX,}  pp.~164--182.
Amer. Math. Soc., Providence, R.I., 1966.

\bibitem
{sage}
W.\,A. Stein et~al.,
 \emph{{S}age {M}athematics {S}oftware
({V}ersion 9.2)}, The Sage Development Team, 2020, {\tt
https://www.sagemath.org}.


\bibitem
{steinberg}
R.~Steinberg,
{\em Lectures on {C}hevalley groups},
Univ. Lecture Ser., 66, American Mathematical Society, Providence, RI, 2016.


\bibitem
{tauyu}
P. Tauvel and R.\,W.\,T. Yu,
{\em Lie algebras and algebraic groups},
Springer Monographs in Mathematics. Springer-Verlag, Berlin, 2005.

\bibitem
{Vinberg76}
E. B. Vinberg,
{The {W}eyl group of a graded {L}ie algebra},
{\em  Izv. Akad. Nauk SSSR Ser. Mat.} {\bf 40} (1976), no. 3, 488--526, 709.
English translation: {\em Math. USSR-Izv.} {\bf 10} (1976), 463--495.


\bibitem
{VE1978}
{  E.\,B. Vinberg  and A.\,G. Elashvili},
{A classification of the trivectors of nine-dimen\-sional space}  (Russian),
{\em Trudy Sem. Vektor. Tenzor. Anal.} {\bf 18} (1978), 197--233,
English translation: {\em Selecta Math. Sov.} {\bf 7} (1988), 63--98.

\bibitem
{Vos}
V.\,E. Voskresenski\u\i,
{\em Algebraic Groups and Their Birational Invariants,}
Translations of Mathematical Monographs, 179,
Amer. Math. Soc., Providence, RI, 1998.

\bibitem
{win}
D.\,J.  Winter,
{\em Abstract Lie algebras}, The M.I.T. Press, Cambridge,
Mass.-London, 1972.

\end{thebibliography}
\end{document}